\documentclass[12pt]{article}
\usepackage[utf8]{inputenc}
\usepackage[T1]{fontenc}
\usepackage{lmodern}
\usepackage[english]{babel}
\usepackage{amsmath,mathtools,amsthm,amssymb,amsfonts}
\usepackage{color}
\usepackage{comment}

\usepackage{mathrsfs}
\usepackage{graphicx}
\usepackage{enumitem}
\usepackage{tikz} 
\usepackage{pgfplots}
\usepackage{esvect}
\usepackage{esint}
\usepackage{pdfpages}
\usepackage{nicefrac}
\usepackage{accents}
\usepackage[symbol]{footmisc}
\usetikzlibrary{shapes,arrows}
\usetikzlibrary{calc}
\usetikzlibrary{shapes.geometric}
\usetikzlibrary{shapes.arrows}
\usetikzlibrary{arrows.meta}
\usetikzlibrary{decorations.markings}
\usetikzlibrary{fit}
\usetikzlibrary{patterns}
\usetikzlibrary{hobby}
\usepgfplotslibrary{patchplots}
\pgfplotsset{compat=1.11}

\renewcommand{\thefootnote}{\fnsymbol{footnote}}
\newcommand{\definedas}{\mathrel{\raise.095ex\hbox{\rm :}\mkern-5.2mu=}}

\newcommand{\R}{\mathbb{R}}
\newcommand{\N}{\mathbb{N}}

\newcommand{\landau}[1]{O\left(#1\right)}

\renewcommand{\d}{\,\mathrm{d}}
\newcommand{\graph}{\,\mathrm{graph}}

\newcommand{\btr}[1]{\left\vert#1\right\vert}
\newcommand{\norm}[1]{\btr{\btr{#1}}}

\newcommand{\spann}[1]{\left\langle#1\right\rangle}

\newcommand{\divf}{\mathrm{div}}

\newcommand{\Ric}{\mathrm{Ric}}

\newcommand{\tr}{\text{tr}}
\newcommand{\dive}{\text{div}}

\newcommand{\gspann}[1]{g\left(#1\right)}
\newcommand{\gamasp}[1]{\gamma\left(#1\right)}

\newcommand{\newbtr}[1]{\vert#1\vert}

\let\OLDthebibliography\thebibliography
\renewcommand\thebibliography[1]{
	\OLDthebibliography{#1}
	\setlength{\parskip}{0pt}
	\setlength{\itemsep}{1pt plus 0.3ex}
}

\theoremstyle{plain}
\newtheorem{thm}{Theorem}[section]
\newtheorem{prop}[thm]{Proposition}

\newtheorem{lem}[thm]{Lemma}

\theoremstyle{definition}
\newtheorem{defi}[thm]{Definition}

\newtheorem{bem}[thm]{Remark}
\newtheorem{kor}[thm]{Corollary}
\newtheorem{bsp}[thm]{Example}

\begin{document}
\begin{center}
	{\LARGE {On the evolution of hypersurfaces along their inverse spacetime mean curvature}\par}
\end{center}
\vspace{0.5cm}
\begin{center}
	{\large Gerhard Huisken\footnote[1]{gerhard.huisken@uni-tuebingen.de}, Markus Wolff\footnote[2]{wolff@math.uni-tuebingen.de}}\\
	\vspace{0.4cm}
	{\large Department of Mathematics, University of T\"ubingen}\\
	{\large T\"ubingen, Germany}
\end{center}
\vspace{0.4cm}
\begin{abstract}
	We construct weak solutions for the evolution of hypersurfaces along their inverse space-time mean curvature in asymptotically flat maximal initial data set. As the speed of the new flow is given by a space-time invariant, it can detect both future- and past-trapped apparent horizons. The weak solution extends concepts developed by Huisken-Ilmanen for inverse mean curvature flow and by Moore for inverse null mean curvature flow.
\end{abstract}
\renewcommand{\thefootnote}{\arabic{footnote}}
\setcounter{footnote}{0}
	
	\section{Introduction}\label{sec_introduction}
	
	Let $(M^{n+1}, g, K), n\ge 2$, be a triple consisting of a smooth, complete Riemannian manifold $(M^{n+1}, g)$ with one asymptotically Euclidean end and possibly possessing a compact inner boundary, together with a symmetric $(0,2)$-tensor field $K$ satisfying appropriate decay conditions near infinity. Such triples arise naturally as space-like hypersurfaces with induced metric $g$ and second fundamental form $K$ in a Lorentzian manifold $(L^{n+2},h)$ modeling the outer region of an isolated gravitating system in General Relativity. They can serve as an  {\em initial data set} for Einstein's equations if suitable constraints for $g, K$ are satisfied. For a smooth hypersurface immersion $F\colon\Sigma^n \to M^{n+1}$ we denote by $\nu$ the (outer) unit normal, by $H$ the corresponding mean curvature and by $P = tr_{\Sigma} K$ the n-trace of $K$. We then call the quantity $\sqrt{H^2 -P^2}$ the {\em space-time mean curvature} of $F(\Sigma^n, t)$ since in the case  $M^{n+1} \subset (L^{n+2}, h)$ this quantity equals the length of the mean curvature vector of the 2-co-dimension surface $\Sigma^n$ in the ambient manifold $(L^{n+2}, h)$.   

We say that a family of hypersurfaces $F\colon\Sigma^n\times[0,T)\to M^{n+1}$ is a smooth solution to {\em inverse space-time mean curvature flow} (STIMCF)
 if $H^2 -P^2 >0$ and
	\begin{equation}\label{stimcf}
		\frac{\partial F}{\partial t}(x,t) = \frac{\nu}{\sqrt{H^2-P^2}}(x,t), \qquad
	 x\in\Sigma^n, \quad t\ge0.
	\end{equation}

In this paper we develop the theory of smooth and weak solutions to this flow and prove a general existence result for global weak solutions of the flow on exterior regions $M^{n+1} \setminus E_0$ if $n<6$ and $tr_g K \equiv 0$ (Theorem \ref{thm_mainresult}), corresponding to the case of {\em maximal} initial data sets. We also show that the flow can detect outermost surfaces where $H = |P|$ (Corollary \ref{kor_detecthorizon}), including both future and past marginally outer trapped surfaces.

Inverse mean curvature flow was first suggested by Geroch and Jang--Wald as an approach to the Penrose inequality for the mass of an asymptotically flat manifold in \cite{geroch, jangwald} and has then been reformulated as a degenerate elliptic variational problem by Huisken--Ilmanen in \cite{huiskenilmanen} to prove the Riemannian version of the Penrose inequality, which generalises and quantifies the original proof of the Positive Mass Theorem by Schoen--Yau in \cite{schoenyau}. See also \cite{bray} for an alternative proof of the Riemannian Penrose inequality by Bray and see \cite{huiskenilmanen} for additional references to the Riemannian Penrose inequality. All these approaches benefit from the fact that in the Riemannian case (where $K\equiv 0$) an apparent horizon is characterized by the minimal surface equation $H=0$. In \cite{eichetal} Eichmair--Huang--Lee--Schoen used marginally outer trapped surfaces to give a proof of the space-time positive mass theorem in dimension less than eight.

To attack the Penrose inequality in the general case different flows have been proposed by Moore \cite{moore} where a weak formulation of the flow along the null mean curvature $H\pm P$ is developed with the help of a new variational principle addressing the anisotropic nature of this  flow speed, and by Frauendiener \cite{frauendiener} in direction of the inverse mean curvature vector in space-time. Bray--Hayward--Mars--Simon in \cite{brayhaywardmarssimon} considered flows in space-time in null and space-like directions. The flows considered by Frauendiener and Bray--Hayward--Mars--Simon have a monotone Hawking mass but lack a proof of existence as they are partially of parabolic and partially of anti-parabolic type. See also \cite{braykhuri, mars} for additional references. 

The approach of this paper is motivated by the space-time invariance of the speed function and in particular by the recent work of Cederbaum-Sakovich in \cite{cederbaumsakovich}, where they propose a foliation of constant space-time mean curvature surfaces as a notion of center of mass in non time-symmetric initial data sets, which can be understood as a Lorentzian analogue to the definition of center of mass in time-symmetry by constant mean curvature surfaces first proposed by Huisken--Yau \cite{huiskenyau}. We expect (work in progress) that the foliation near spacial infinity resulting from STIMCF will exhibit the same distinguished asymptotic convergence to the center of mass described for the foliation by constant space-time mean curvature surfaces in \cite{cederbaumsakovich}. We also note the interesting analytical similarities of our approach to the ground breaking work of Schoen--Yau on the positive energy theorem employing Jang's equation, which also detects surfaces with $H= \pm P$, \cite{schoenyau2}. 

We begin with basic properties of smooth solutions to STIMCF and prove in Theorem \ref{smooth} that short-time smooth solutions exist for smooth initial data $\Sigma^n_0 = \partial E_0$ and can be extended as long as the speed of the hypersurfaces is bounded.  We also establish an interior upper bound on $\Phi = \sqrt{H^2 - P^2}$ in Theorem \ref{thm_interiorgradientestimate}, which is a crucial prerequisite for the approximation of weak solutions further on. Our approach to weak solutions of equation (\ref{stimcf}) is based on a variational principle for a level-set formulation inspired by the work of Huisken-Ilmanen \cite{huiskenilmanen} on inverse mean curvature flow and the extension by Moore \cite{moore} to the anisotropic inverse null mean curvature flow. To be precise, we reformulate the flow equation (\ref{stimcf}) in a level-set ansatz as a degenerate elliptic PDE

\begin{equation}\label{level-seteqn}
	\divf_g \left(\frac{\nabla u}{\btr{\nabla u}_g}\right)
	=\sqrt{\btr{\nabla u}^2_g+\left(\left(g^{ij}-\frac{\nabla^iu\nabla^ju}{\btr{\nabla u}_g^2}\right)K_{ij}\right)^2}
	\end{equation}
	
on the exterior region $\Omega \definedas M^{n+1} \setminus E_0$ and construct weak solutions for suitable initial data by an elliptic regularisation as used by Ilmanen in \cite{ilmanen}. Approximating solutions $u_{\varepsilon}$ satisfying 

	\[
	\divf_g\left(\frac{\nabla u_\varepsilon}{\sqrt{\btr{\nabla u_\varepsilon}^2_g+\varepsilon^2}}\right)
	=\sqrt{\varepsilon^2+\btr{\nabla u_\varepsilon}^2_g+\left(\left(g^{ij}-\frac{\nabla^iu_\varepsilon\nabla^ju_\varepsilon}{\varepsilon^2+\btr{\nabla u_\varepsilon}^2_g}\right)K_{ij}\right)^2}
	\]

subject to appropriate boundary conditions are constructed in Theorem \ref{thm_ellipticregexistence}, based on precise regularity estimates (Theorem \ref{thm_aprioriestimates}). The lower bound in the interior region requires the assumption $\tr_gK=0$, i.e., the restriction to maximal initial data sets. Rescaling yields translating solutions ${\hat\Sigma^{\varepsilon}_t = \operatorname{graph}(\frac{u_{\varepsilon}}{\varepsilon} - \frac{t}{\varepsilon})}$ of STIMCF one dimension higher in $M^{n+1} \times \R$ and the general interior estimate on the inverse speed $\phi = \sqrt{H^2 -P^2}$ (Theorem \ref{thm_interiorgradientestimate}) can then be applied in $M^{n+1} \times \R$ to prove a uniform gradient estimate and thereby pre-compactness for the $u_{\varepsilon}$. Assuming now $n < 6$ and employing compactness and regularity results from geometric measure theory we show in Section 5 that the surfaces $\hat\Sigma^{\varepsilon}_t$ are uniformly controlled in $C^{1,\alpha}_{loc}$ ensuring sub-convergence as $\varepsilon \to 0$ to a foliation of $\Omega \times \R$ by $C^{1,\alpha}$-surfaces with a corresponding normal vector field $\hat\nu \in C^{0,\alpha}_{loc}(\Omega \times \R)$ - even above jump regions in $\Omega$ where the limiting function $u$ has plateaus with $\nabla u = 0$.

Setting $P_\nu\definedas\left(g^{ij}-\nu^i\nu^j\right)K_{ij}$ for the n-trace of $K$ orthogonal to the vectorfield $\nu$ and freezing $\sqrt{\btr{\nabla u}^2+P_\nu^2}$ as a bulk term energy, we can interpret \eqref{level-seteqn} as the Euler-Lagrange equation to the functional
	\begin{align}\label{comprinfuncintro}
	\mathcal{J}_{u,\nu}^A(v)\definedas\int\limits_{A}\btr{\nabla v}+v\sqrt{\btr{\nabla u}^2+P_\nu^2},
	\end{align}
with $v\in C^{0,1}_{loc}(\Omega)$ and $A \subset \Omega$ compact, as long as the vector field $\nu$ is well defined. The key idea in characterising weak solutions is then to use this functional not on $(u_{\varepsilon}, \nu_{\varepsilon})$ but on
the level-set data $(U_{\varepsilon}(x, z) = u_{\varepsilon}(x) - \varepsilon z, \hat\nu_{\varepsilon})$ describing the flow of the 
$\hat\Sigma^{\varepsilon}_t \subset \Omega \times \R$ to formulate a  variational principle on the product $\Omega \times \R$ that can be passed to the limit with the pair $(U(x,z) = u(x), \hat\nu)$ as $\varepsilon \to 0$ (Theorem \ref{thm_compactness2}). The main result on existence of solutions to this variational principle is then proven in Section 7 (Theorem \ref{thm_mainresult}). In Remark \ref{bem_frauend} we explain how the methods of this paper apply to solve a projected version of the flow proposed by Frauendiener in \cite{frauendiener}. 
Section 8 establishes the outward optimizing property of the level-set solutions (Theorem \ref{thm_outwardoptimziation}) and gives a precise description of jump regions. As a consequence we see that the flow can detect outermost generalized apparent horizons when starting from some initial data (Corollary \ref{kor_detecthorizon}).   
We conclude with a description of the asymptotic behavior of the flow as $t \to \infty$ in the last section. \par
\vskip 0.3cm
\noindent
{\it Acknowledgement.} The authors thank Carla Cederbaum and Sascha Eichmann for inspiring discussions. We would further like to thank Axel Fehrenbach and Olivia Vi\v{c}\'{a}nek Mart\'{i}nez  for providing several illustrations.

	\section{Preliminaries}\label{sec_preliminaries}
		To simplify notation, we will in the following write $M=M^{n+1}, \Sigma=\Sigma^n$ and briefly introduce some notation, definitions and assumptions on the background manifold $(M,g)$. 
		The \emph{Riemann curvature tensor} $\operatorname{Riem}_{\alpha\beta\gamma\delta}$, the \emph{Ricci curvature} $\Ric_{\alpha\beta}$ and the \emph{scalar curvature} $\operatorname{R}$ of $(M,g)$ are given in terms of the Christoffel-symbols $\{\Gamma^{\alpha}_{\beta \gamma}\}$ of the metric $g$ as
		\begin{align*}
		&\operatorname{Riem}_{\alpha\beta\gamma\delta}\definedas
		g_{\rho\gamma}\left(\Gamma_{\beta\delta,\alpha}^\rho-\Gamma_{\alpha\delta,\beta}+\Gamma_{\beta\delta}^\lambda\Gamma_{\alpha\lambda}^\rho-\Gamma_{\alpha\delta}^\lambda\Gamma_{\beta\lambda}^\rho\right),\\
		&\Ric_{\alpha\beta}\definedas g^{\gamma\delta}\operatorname{Riem}_{\alpha\gamma\beta\delta}
		=\Gamma_{\alpha\beta,\rho}^\rho-\Gamma_{\rho\beta,\alpha}^\rho+\Gamma_{\alpha\beta}^\lambda\Gamma_{\rho\lambda}^\rho-\Gamma_{\alpha\lambda}^\rho\Gamma_{\rho\beta}^\lambda,\\
		&\operatorname{R}\definedas g^{\alpha\beta}\Ric_{\alpha\beta}
		\end{align*}
		respectively. If we want to place emphasis on the metric $g$, we write $\operatorname{Riem}_g$, $\Ric_g$ , and $\operatorname{R}_g$.
		
		If the triple $(M, g, K)$ arises as a spacelike hypersurface with second fundamental form $K$ and future-pointing timelike unit normal $N$ inside a Lorentzian manifold governed by Einsteins Equations for a specific stress-energy tensor $T$, the Gauss- and Codazzi equations  yield the \emph{constraint equations},
		\begin{align*}
			\operatorname{R}_g+(\tr_g K)^2-\btr{K}_g^2&=16\pi\mu,\\
			\dive_g(K-\tr_g Kg)&=8\pi J,
		\end{align*}
		where $\mu\definedas T(N,N)$ is called the \emph{energy density}, $J\definedas T(\cdot,N)$ is called the \emph{momentum density}. Motivated by this setting we call a triple $(M,g,K)$ an {\em initial data set} and take the lefthand side of the above equations as the definition of $\mu$ and $J$ respectively. We call  an initial data set $(M,g,K)$
 \emph{time-symmetric} if $K\equiv 0$ and \emph{maximal} if $\tr_g K\equiv 0$.

		We further say that an initial data set $(M,g,K)$ satisfies the \emph{dominant energy condition} (DEC), if
		\begin{align}\label{eq_dec}
			\mu\ge \btr{J}.
		\end{align}
		To establish existence of solutions to our flow, we additionally require the initial data set $(M,g,K)$ to satisfy certain fall-off conditions at spatial infinity. The next definition makes this mathematically rigorous by introducing proper decay conditions, cf. \cite{cederbaumsakovich}.
		\begin{defi}\label{defi_asymflat}
			Let  $(M,g,K)$ be an initial data set and let $\varepsilon\in(0,\frac{1}{2}]$. We call $(M,g,K)$ \emph{asymptotically flat} (with one end), if there exists a compact subset $\overline{\mathcal{B}}\subset M$, such that $M\setminus\overline{\mathcal{B}}$ is diffeomorphic to $\R^{n+1}\setminus \overline{B_{1}(0)}$, and such that under this diffeomorphism, the metric tensor $g$ and second fundamental form $K$\footnote{we identify $g$ and $K$ with their respective pullbacks via the diffeomorphism for simplicity} satisfy
			\begin{align*}
				&\btr{g_{\alpha\beta}-\delta_{\alpha\beta}}+\btr{x}\btr{g_{\alpha\beta,\gamma}}+\btr{x}^2\btr{g_{\alpha\beta, \gamma\delta}}\le C\btr{x}^{-n+\frac{3}{2}-\varepsilon},\\
				&\btr{K_{\alpha\beta}}+\btr{x}\btr{K_{\alpha\beta,\gamma}}\le C\btr{x}^{-n+\frac{1}{2}-\varepsilon},
			\end{align*}
			as $\btr{x}\to\infty$, where $\btr{\cdot}$ is the Euclidean distance with respect to the coordinates, $C>0$ a real constant, and the derivatives are taken with respect to the Euclidean metric.
		\end{defi}
		\begin{bem}\label{bem_asymflat}
			In General Relativity these decay assumptions together with the integrability of the constraints ensure that a notion of total Energy and momentum of the system is well-defined. See \cite{eichetal} for a precise description of energy, momentum and natural decay assumptions in dimensions $n<7$. Since inverse space-time mean curvature flow is highly motivated by General Relativity, we will generally assume these decay conditions. However in view of Proposition \ref{prop_blowdown} and the comments following Lemma \ref{lem_subsolution}, we can establish existence and asymptotic behavior of the flow under much weaker decay (especially in higher dimensions). See Proposition \ref{prop_blowdown}, equation \eqref{asympdecay2} for the minimal decay assumptions.
		\end{bem}
		For a hypersurface $(\Sigma,\gamma)\subset (M,g)$ with induced Riemannian metric $\gamma$ the \emph{second fundamental form} $h=\{h_{ij}\}_{1\le i,j \le n}$ is defined as $h(X,Y)\definedas- g(\nu,\nabla_XY)$, where $\nu$ is a choice of unit normal to $\Sigma$. 
We will write $D$ for the covariant derivative w.r.t. the metric $\gamma$ and $\Delta_{\gamma}$ for the corresponding Laplace-Beltrami operator. We define the mean curvature of $(\Sigma,\gamma)$ as $H\definedas \tr_\gamma h$ and the mean curvature vector  as $\vec{H}\definedas-H\nu$. If $(M,g,K)$ arises again as a spacelike submanifold of a Lorentzian $(L,h)$ manifold with second fundamental form $K$, $(\Sigma,\gamma)$ is a codimension-$2$ submanifold of $(L,h)$ and the information about the additional extrinsic curvature is encoded in $K$. Thus, we define the {\em future and past expansions} of $\Sigma$ as $\theta_\pm\definedas H\pm P$ and call $\mathcal{H}\definedas\sqrt{H^2-P^2}$ the {\em spacetime mean curvature}, where $P\definedas \tr_\gamma \Sigma$. We call $(\Sigma,\gamma)$ a \emph{marginally outer/inner trapped surface (MOTS/MITS)}, if $\theta_\pm\equiv0$, and we call $(\Sigma,\gamma)$ a \emph{generalized apparent horizon}, if $\mathcal{H}\equiv0$. Note that there are examples of surfaces satisfying $\mathcal{H}\equiv0$ that are neither MOTS or MITS, compare \cite{carrascomars}.
		
	\section{The smooth flow}\label{sec_smoothflow}
		We first concentrate on the properties of smooth solutions of STIMCF. Assuming that $\sqrt{H^2-P^2}{\Big\vert_{\Sigma_0}}>0$, we establish evolution equations for geometric quantities and show that the initial value problem for the flow is parabolic with surfaces $\Sigma_t\definedas F(t,\Sigma)$ expanding smoothly as long as their speed remains bounded. The main result of this section (Theorem \ref{thm_interiorgradientestimate}) provides an upper bound on the inverse speed restricted to a ball with sufficiently small radius depending on the geometry of $(M, g, K)$. This will later serve as an interior gradient estimate in the level-set formulation of the flow and will be derived from the 
		evolution equations for STIMCF.
		\begin{lem}\label{lem_evolutioneq} Smooth solutions of \eqref{stimcf} with $\Phi:=\sqrt{H^2-P^2}>0$ and $\Psi:=\frac{1}{\Phi}$ satisfy the following evolution equations:
			\begin{enumerate}
				\item[(i)] $\frac{\d}{\d t}\gamma_{ij}=2\Psi h_{ij}$, $\frac{\d}{\d t}\gamma^{ij}=-2\Psi h^{ij}$, $\frac{\d}{\d t}\d\mu_{\gamma}=\frac{H}{\sqrt{H^2-P^2}}\d\mu_{\gamma}$,
				\item[(ii)]
				$\frac{\d}{\d t}\nu=-D\Psi$,
				\item[(iii)]
				$\frac{\d}{\d t} h_{ij}=-\operatorname{Hess}_{\gamma}\Psi_{ij}+\Psi h_i^kh_{kj}-\Psi\operatorname{Riem}_{i\nu j\nu}$,
				\item[(iv)] 
				$\frac{\d}{\d t}H=-\Delta_\gamma\Psi-\Psi\left(\Ric(\nu,\nu)+\btr{h}^2_{\gamma}\right)$,
				\item[(v)] 
				$\frac{\d}{\d t}P=\Psi\tr_{\gamma}\nabla_\nu K+2K_{\nu i} D^i\Psi$,
				\item[(vi)]
				$\frac{\d}{\d t}\Phi=\Phi^{-2}\left(\frac{H}{\Phi}\Delta_\gamma\Phi-\frac{2H}{\Phi^2}\btr{D\Phi}_{\gamma}^2-H\left(\Ric(\nu,\nu)+\btr{h}^2_{\gamma}\right)-P\tr_{\gamma}\nabla_\nu K+\frac{2P}{\Phi}K_{\nu i} D^i\Phi\right)$.
			\end{enumerate}
		\end{lem}
		\begin{proof}
			(i)--(iv) follow as a direct consequence of the well-known evolution equations for general flows, see e.g. \cite[Theorem 3.2]{huiskenpolden}. For (v), we compute the evolution of $P$ as in \cite[Lemma 3]{moore}:
			\begin{align*}
			\frac{\d}{\d t}P
			&=\frac{\d}{\d t}(\tr_gK-\nu^{\alpha}\nu^{\beta}K_{\alpha\beta})\\
			&=\frac{\d}{\d t}\tr_gK-\nu^{\alpha}\nu^{\beta}\frac{\d}{\d t}K_{\alpha\beta}-2\nu^{\beta}K_{\alpha\beta}\frac{\d}{\d t}\nu^{\alpha}\\
			&=\Psi\left(\nabla_\nu\tr_gK-\left(\nabla_\nu K\right)(\nu,\nu)\right)+2K_{\nu i}D^i\Psi\\
			&=\Psi\left(\tr_g\nabla\nu K-\left(\nabla_\nu K\right)(\nu,\nu) \right)+2K_{\nu i}D^i\Psi\\
			&=\Psi\tr_{\gamma}\nabla_\nu K+2K_{\nu i}D^i\Psi.
			\end{align*}
			(vi) follows directly from (iv) and (v).
		\end{proof}
		
		While the main aim of this paper is the construction of a global weak solution to STIMCF, we begin by proving
that for smooth closed initial hypersurfaces $F_0$ with $\Phi > 0$ a smooth solution to the flow exists as long as $\Phi > 0$ remains true.
A corresponding result for inverse mean curvature flow in Euclidean space was shown in \cite[Corollary 2.3]{huiskenilmanen2}, for general ambient manifolds the result appears to be new also for inverse mean curvature flow.

\begin{thm}\label{smooth} Suppose the initial hypersurface $F_0\colon\Sigma\to M$ is a smooth immersion satisfying  
$\Phi>\delta_0>0$.  Then there exists $T>0$ depending on $\delta_0$ and the regularity of $F_0, (M,g,K)$ with a unique smooth solution $F\colon\Sigma\times[0,T)\to M$ to \eqref{stimcf} on $[0, T)$.  If the space-time mean curvature $\Phi$ remains bounded from below by a constant $\delta_1>0$ for all $t\in [0,T)$, then the solution can be extended beyond $T$. In particular, if 
$[0, T_{max}) $ is the maximal time interval of existence for a smooth
solution of \eqref{stimcf} with $T_{max} <\infty$, then the speed $\Psi = 1/\Phi$ is unbounded for $t\to T_{max}$.
\end{thm}

First note that short-time existence follows from the implicit function theorem as in \cite[Section 7]{huiskenpolden} since the
linearised operator 
\begin{align*}
\frac{d}{dt} - \frac{H}{\Phi^3}\Delta
\end{align*}
associated with the evolution system \eqref{stimcf} is parabolic modulo tangential diffeomorphisms since
by assumption $\Phi >\delta_0>0$ and hence $H > \delta_0 > 0$.

To prove that the solution can be extended as long as $\Phi$ remains positive, it is enough to prove uniform
bounds on the curvature: Higher derivative estimates then follow from the regularity theory of Krylov \cite{krylov} 
as in the case of inverse mean curvature flow \cite[Corollary 2.3]{huiskenilmanen2}, as $\Phi$ is a concave function 
of the second fundamental form.

To derive curvature estimates from the lower bound on $\Phi$ we begin with a global upper bound on $\Phi$. 
A more sophisticated interior estimate without assuming a lower bound on $\Phi$ will be established in 
Theorem \ref{thm_interiorgradientestimate}.

\begin{lem}\label{lem_upperboundPhi}
Under the assumptions of Theorem \ref{smooth} there is an upper bound for $\Phi$ on $M^{n+1} \times [0, T)$
$$
\max_{\Sigma \times [0, T)} \Phi \leq \max (\max_{\Sigma_0} \Phi, C(\delta_1)),
$$
where $C(\delta_1)$ depends on $\delta_1,\,\textcolor{red}{n},\, \sup_{\Sigma \times [0, T)} |{\rm Ric}|_g,\, \sup_{\Sigma \times [0, T)} |K|_g \,and \,\sup_{\Sigma \times [0, T)} |\nabla K|_g$.
\end{lem}

\begin{proof} 
Using $|h|^2 \geq (1/n) H^2$ and $H^2 = \Phi^2 + P^2 \geq \Phi^2 \geq \delta_1^2$  the RHS in the evolution equation for $\Phi$ in Lemma \ref{lem_evolutioneq} (vi) can be estimated from above to get
\begin{eqnarray}\nonumber
\frac{d}{dt} \Phi &\leq& H\Phi^{-3}\Delta_{\gamma} \Phi - 2\Phi^{-3}|D \phi|^2 - \frac{1}{n}H^3\Phi^{-2} \\ \nonumber
&+& \sup |{\rm Ric}|_g H\Phi^{-2}  + n \sup |P|_g \sup |\nabla K|_g \Phi^{-2} + 2 \sup |P|_g \sup |K|_g \Phi^{-3} |D\Phi| \\
\nonumber
&\leq& H\Phi^{-3}\Delta_{\gamma} \Phi - \Phi^{-3}|D \phi|^2 - \frac{1}{n}\Phi + C(\delta_1).
\end{eqnarray}
The conclusion of the lemma then follows from considering the first event where a new maximum of
$\Phi$ is attained. 
\end{proof}

To derive an estimate for the full second fundamental form we first have to rewrite the evolution equation 
for the second fundamental form such that a tensor version of the parabolic maximum principle can be applied to the 
largest eigenvalue. For this we need the Simons identity for the second derivatives of the second fundamental form, 
which we write as in \cite[Corollary 2.2]{huiskenpolden}:
\begin{eqnarray}\label{simonsidentity}
D_i D_j H &=& \Delta_{\gamma} h_{ij} - Hh_{im}h_{mj} +h_{ij} \vert h \vert^2 -
H {\rm Riem}_{\nu i \nu j}
\\ 
\nonumber
&&+ {\rm Ric}_{\nu\nu}h_{ij} - {\rm Riem}_{kikm}h_{mj} -  {\rm Riem}_{kjkm}h_{im}
\\ 
\nonumber
&&- {\rm Riem}_{kijm}h_{km} -  {\rm Riem}_{mjik}h_{km} -  \nabla_k  {\rm Riem}_{\nu jik} -
\nabla_i  {\rm Riem}_{\nu kjk},
\\  
\nonumber
 &=&  \Delta_{\gamma} h_{ij} - Hh_{im}h_{mj} +h_{ij} \vert h \vert^2 + Z_{ij}, 
 \\
 \nonumber
 Z_{ij} &:=& -H {\rm Riem}_{\nu i \nu j} + {\rm Ric}_{\nu\nu}h_{ij} - {\rm Riem}_{kikm}h_{mj} -  {\rm Riem}_{kjkm}h_{im}
\\ 
&&- {\rm Riem}_{kijm}h_{km} -  {\rm Riem}_{mjik}h_{km} -  \nabla_k  {\rm Riem}_{\nu jik} -
\nabla_i  {\rm Riem}_{\nu kjk}.
\nonumber
 \end{eqnarray}

From $D_j \Phi = \phi^{-1}(H D_j H - P D_j P)$ we may then compute
\begin{eqnarray}\label{hessianphi}
D_i D_j \Phi &= & - \Phi^{-2} D_i \Phi (H D_j H - P D_j P) + H \Phi^{-1} D_i D_j H  \\
\nonumber
&& + \Phi{-1} D_i H D_j H - \Phi^{-1}(P D_i D_j P + D_i P D_j P) \\
\nonumber
&=& H\Phi^{-1} D_i D_j H - \Phi^{-1} D_i \Phi D_j\Phi + \Phi^{-1} D_i H D_j H \\
\nonumber
&& - \Phi^{-1}(P D_i D_j P + D_i P D_j P) \\
\nonumber
&=&  H\Phi^{-1} \Big(\Delta_{\gamma} h_{ij} - Hh_i^m h^m_j + h_{ij} |h|^2  + Z_{ij} \Big)\\
\nonumber
&& -\Phi^{-1} D_i \Phi D_j \Phi + \Phi^{-1} D_i H D_j H - \Phi^{-1} (P D_i D_j P + D_i P D_j P).
\end{eqnarray}

Then the evolution equation for the second fundamental form in Lemma \ref{lem_evolutioneq} (iii) becomes
\begin{eqnarray}\label{evol_fundform}
\frac{\d}{\d t} h_{ij} &=& -D_i D_j\Psi +\Psi h_i^kh_{kj}-\Psi\operatorname{Riem}_{i\nu j\nu}\\
\nonumber
&=& \Phi^{-2}D_i D_j \Phi - 2 \Phi^{-3}D_i\Phi D_j\Phi  + \Phi^{-1} h_i^kh_{kj}-\Phi^{-1}\operatorname{Riem}_{i\nu j\nu}\\
\nonumber
&=& H\Phi^{-3}\Big(\Delta_{\gamma} h_{ij} - H h_i^k h_{kj} + h_{ij} |h|^2 +Z_{ij} \Big) 
+ \Phi^{-1} h_i^kh_{kj}-\Phi^{-1}\operatorname{Riem}_{i\nu j\nu}\\
\nonumber
&&  -3\Phi^{-3} D_i \Phi D_j \Phi + \Phi^{-3} D_i H D_j H - \Phi^{-3} (P D_i D_j P + D_i P D_j P). 
\end{eqnarray}

The aim is now to get an upper bound on the eigenvalues of the tensor $M_{ij} \definedas{\Phi h_{ij}}$. Combining
the evolution equations for $h_{ij}$ and $\Phi$ we get

\begin{eqnarray}\label{weightedFundForm_eq1}
\frac{d}{dt} M_{ij} &=& H\Phi^{-3} \Delta_{\gamma} M_{ij} -2H\Phi^{-4}D^k \Phi D_k M_{ij} -3 \Phi^{-2} D_i \Phi D_j \Phi 
+ \Phi^{-2}D_i H D_j H \\
 \nonumber
 && + h_{im}h^m_j (1-H^2\Phi^{-2}) - \operatorname{Riem}_{i\nu j\nu}\\
 \nonumber
 &&  + \Phi^{-1}\Big(H\Phi^{-1}Z_{ij} - \Phi^{-1}(P D_i D_j P + D_i P D_j P) \Big)\\
 \nonumber
 && + h_{ij} \Big(-H\Phi^{-2} \operatorname{Ric}_{\nu\nu} -P\Phi^{-2}tr_{\gamma}\nabla_{\nu} K 
 + 2P\Phi^{-3}K_{\nu m}D^m \Phi \Big).
\end{eqnarray}

Using now
\begin{align}
D_i\Phi D_j\Phi = H^2\Phi^{-2} D_i H D_j H + P^2 \Phi^{-2} D_iP D_j P - PH\Phi^{-2} (D_i P D_j H + D_jP D_i H) 
\end{align}
we can write the critical gradient terms in the evolution equation of $M_{ij}$ as
\begin{align}\label{eqn-M}
\begin{split}
-3\Phi^{-2}D_i\Phi D_j\Phi + \Phi^{-2}D_iH D_j H =& -\Phi^{-2}D_i\Phi D_j\Phi + (1 - 2H^2\Phi^{-2}) \Phi^{-2}D_iH D_j H \\
&-2P^2\Phi^{-4}D_i P D_j P \\
&+ 2PH\Phi^{-4} (D_i P D_j H + D_jP D_i H) .
\end{split}
\end{align}	

With $H^2 = \Phi^2 + P^2$ and setting $\xi_i := PD_i H - 2H D_i P$ we can rewrite the RHS of this equation as
\begin{eqnarray}\label{complete-square}
&&-\Phi^{-2}D_i\Phi D_j\Phi - (H^2 + P^2)\Phi^{-4}D_i H D_j H \\ \nonumber
&& -2P^2\Phi^{-4}D_i P D_j P + 2PH\Phi^{-4} (D_i P D_j H + D_jP D_i H) \\ \nonumber
= && -\Phi^{-2}D_i\Phi D_j\Phi - H^2 \Phi^{-4}D_i H D_j H -\Phi^{-4} \xi_i \xi_j + 2\Phi^{-4} (H^2 +\Phi^2) D_i PD_jP
\end{eqnarray}

such that
\begin{eqnarray}\label{weightedFundForm_eq2}
\frac{d}{dt} M_{ij} &=& H\Phi^{-3} \Delta_{\gamma} M_{ij} -2H\Phi^{-4}D^k \Phi D_k M_{ij} - \Phi^{-2} D_i \Phi D_j \Phi \\
 \nonumber
&& - H^2\Phi^{-4}\Phi^{-2}D_i H D_j H  -\Phi^{-4} \xi_i \xi_j - P^2 h_{im}h^m_j \\
 \nonumber
 && + \Phi^{-4} (2H^2 +\Phi^2) D_i PD_jP - \operatorname{Riem}_{i\nu j\nu} + H\Phi^{-2}Z_{ij} - \Phi^{-2}P D_i D_j P \\
 \nonumber
 && + h_{ij} \Big(-H\Phi^{-2} \operatorname{Ric}_{\nu\nu} -P\Phi^{-2}tr_{\gamma}\nabla_{\nu} K 
 + 2P\Phi^{-3}K_{\nu m}D^m \Phi \Big).
\end{eqnarray}

Now notice that, using the Codazzi equations $D_j h_{ik} = D_k h_{ij} + \operatorname{Riem}_{\nu jik}$,
\begin{eqnarray}\label{deriv_P}
D_iP &=& D_i(tr_g K - K_{\alpha\beta}\nu^{\alpha}\nu{\beta}) \\
\nonumber
&=& D_i tr_g K - D_iK_{\alpha\beta} \nu^{\alpha}\nu{\beta} -2K_{m\beta}h^m_i\nu^{\beta},\\
\nonumber
D_jD_i P &=& D_jD_i tr_g K - D_j D_i K_{\alpha\beta} \nu^{\alpha}\nu^{\beta} -2D_i K_{m\beta}h^m_j \nu^{\beta}\\
\nonumber
&& -2K_{ml}h^m_i h^l_j + 2h_{im}h^m_j K_{\alpha\beta}\nu^{\alpha}\nu{\beta} -2 D_j h_i^m K_{m\beta}\nu^{\beta}\\
\nonumber
&=& -2 D_k M_{ij} \Phi^{-1}K^k_{\beta}\nu^{\beta} + 2D_k\Phi h_{ij} K^k_{\beta}\nu^{\beta} + B_{ij}
\end{eqnarray}
where 
\begin{align}\label{estim_RHS}
|B_{ij}| \le C(|h|^2 \sup_{\Sigma \times [0, T)}|K| + |h|\sup_{\Sigma \times [0, T)}|\nabla K| + \sup_{\Sigma \times [0, T)}|\nabla^2K|).
\end{align}

We say that a symmetric bilinear form $T$ is non-negative and write $T_{ij}\geq 0$ if all eigenvalues of $\{T^i_j\}$ are non-negative or, equivalently, $T_{ij}X^i X^j \geq 0$ $\,\forall X\in T_p\Sigma, p\in \Sigma.$ Setting

\begin{align}\label{coefficient}
b^k \definedas -2H\Phi^{-4}D^k \Phi +2\Phi^{-3}P K^k_{\beta}\nu^{\beta}
\end{align}

we then conclude from (\ref{weightedFundForm_eq2},  \ref{deriv_P}, \ref{estim_RHS}) and Lemma \ref{lem_upperboundPhi}

\begin{eqnarray}
\frac{d}{dt} M_{ij} &\leq& H\Phi^{-3} \Delta_{\gamma} M_{ij} +b^k D_k M_{ij}  
+ C(1 + |h|^2 + \Phi^{-2}|D\Phi|^2) \gamma_{ij} + H\Phi^{-2} Z_{ij} \\
 \nonumber
 && + h_{ij} \Big(-H\Phi^{-2} \operatorname{Ric}_{\nu\nu} -P\Phi^{-2}tr_{\gamma}\nabla_{\nu} K 
 + 2P\Phi^{-3}K_{\nu m}D^m \Phi \Big),
\end{eqnarray}

where we used again $H^2 = \Phi^2 + P^2$ and  $C$ depends	on $\delta_1, \sup_{\Sigma \times [0,T)} |{\rm Riem}|_g, 
|K|_{C^2 (\Sigma \times [0,T))}$. Estimating now the quantity $Z_{ij}$ from (\ref{simonsidentity}) by
$$
Z_{ij}  \leq C(1 +  |h|) \gamma_{ij}
$$
with a constant depending on $ |{\rm Riem}|_{C^1(\Sigma \times [0, T))}$  we conclude

\begin{eqnarray}\label{final_estimate_M}
\frac{d}{dt} M_{ij} &\leq& H\Phi^{-3} \Delta_{\gamma} M_{ij} +b^k D_k M_{ij}  
+ C_1(1 + |h|^2 + \Phi^{-2}|D\Phi|^2) \gamma_{ij} 
\end{eqnarray}

where we again used  Lemma \ref{lem_upperboundPhi} and  the constant
$C_1$  depends on $\delta_1$, $n$, $\sup_{\Sigma \times [0,T)} |{\rm Riem}|_g$, $|K|_{C^2 (\Sigma \times [0,T))}$. To handle the terms on the RHS we consider the auxiliary tensor $N_{ij} \definedas \Phi^2 \gamma_{ij}$ and compute from $H \geq \Phi$, (\ref{coefficient}) and Lemma \ref{lem_evolutioneq} (i), (vi) 

\begin{eqnarray}\nonumber
\frac{d}{dt} N_{ij} &\leq& H\Phi^{-3}\Delta_{\gamma} N_{ij} - 2H\Phi^{-4}D^k \Phi D_k N_{ij} -2 H\Phi^{-3}|D\Phi|^2
-2|h|^2  \\ \nonumber
&+& C(1+|h|)\gamma_{ij}  + C |D\Phi|\\
\nonumber
&\leq& H\Phi^{-3}\Delta_{\gamma} N_{ij} + b^k D_k N_{ij} - \Phi^{-2}|D\Phi|^2 \gamma_{ij} - |h|^2\gamma_{ij} + C_2\gamma_{ij},
\end{eqnarray}
where the constant $C_2$ depends on $\delta_1$ and the same quantities as before. Now consider a combination
$\tilde M_{ij} \definedas M_{ij} + \alpha N_{ij} -\beta\gamma_{ij}$ for $\alpha = C_1 +1$ and $\beta>0$ such that $\tilde M_{ij} < 0$ at time $t=0$. We get from (\ref{final_estimate_M}) and Lemma \ref{lem_evolutioneq}

\begin{eqnarray}
\frac{d}{dt} \tilde M_{ij} &\leq& H\Phi^{-3} \Delta_{\gamma} \tilde M_{ij} +b^k D_k \tilde M_{ij} \\
 \nonumber
 && -|h|^2 \gamma_{ij} + (C_1 +1)C_2 \gamma_{ij} -2\beta \Phi^{-1}h_{ij}.
\end{eqnarray}
Similar as in \cite[Theorem 9.1]{hamilton} this parabolic differential inequality for the tensor $\tilde M_{ij}$ now leads to a contradiction for sufficiently large $\beta$ depending on $C_1, C_2$ if $\tilde M_{ij}$ reaches a zero eigenvalue somewhere for the first time, since $H\geq 0$, $|h|^2$ is quadratic in the largest eigenvalue $\lambda_n$ of $h$ and $\lambda_n >0$ at such a point. This implies the desired curvature bound since $H\geq 0$ and $\Phi\geq \delta_1$. \qed

We now establish an interior upper bound on $\Phi$ which will be crucial for the construction of weak solutions later on.
		\begin{thm}\label{thm_interiorgradientestimate}
			Let $x\in M^{n+1}$, let $d_x$ denote the distance to $x$, and let $R>0$ be such that $B_R(x)\subset\subset M^{n+1}$, $\Ric_g\ge -\frac{1}{100(n+1)R^2}$ in $B_R(x)$, and there exists a function $p\nolinebreak\in\nolinebreak C^2(B_R(x))$ such that
			\[
			p(x)=0,\text{ } \frac{3}{2}d_x^2\ge p\ge d_x^2\text{ on }B_R(x), \text{ }\btr{\nabla p}_g\le 3d_x,\text{ and }\nabla^2p\le 3g\text{ on }B_R(x).
			\]
			Assuming furthermore that $H_{\Sigma_s}>0$ for $s\in[0,t]$, there exists a constant $C(n)>0$ depending only on the dimension, such that
			\begin{align}
			\sup\limits_{F(\Sigma,[0,t])\cap B_{\nicefrac{R}{2}}(x)}\Phi\le
			\max\left(\max\limits_{\Sigma_0\cap B_R(x)}\Phi,C(n)\left(\frac{1}{R}+\sup\limits_{\overline B_R(x)}\btr{K}_g+\sup\limits_{\overline B_R(x)}\btr{\nabla K}_g^\frac{1}{2}\right)\right),
			\end{align}
			where $\Phi=\sqrt{H^2-P^2}$.
		\end{thm}
		\begin{bem}\label{bem_interiorgradientestimate}
			 We choose $p(y)\definedas\btr{x-y}^2$ in flat space with $R=\infty$, but in general $R>0$ will depend on the injectivity radius and Ricci curvature. However, as argued in the Remark to \cite[Definition 3.3]{huiskenilmanen}, each $x\in M^{n+1}$ admits a positive radius $R$, such that the assumptions are satisfied.
		\end{bem}
		\begin{proof}
			Similar as in Daskalopoulos--Huisken \cite{daskalopouloshuisken} we consider the function $\eta(y)\definedas(R^2-p(y))^2_{+}$ on $B_R(x)$, where $p$ and $R$ are as above. Then the evolution of $\varphi\definedas \Phi\cdot \eta$ is given by
			\begin{align*}
			\frac{\d}{\d t}\varphi
			=&\frac{H}{\Phi^3}\Delta_\gamma\varphi-\frac{2H}{\Phi^4}\gamasp{D\Phi, D\varphi}+\frac{2P}{\Phi^3}K_{\nu i}D^i\varphi\\
			&+\frac{2H}{\Phi^2}\eta^{\frac{1}{2}}\Delta_\gamma p-\frac{2H}{\Phi^2}\btr{Dp}^2-\frac{H\eta}{\Phi^2}\left(\Ric(\nu,\nu)+\btr{h}^2\right)\\
			&-\frac{P\eta}{\Phi^2}\tr_{\gamma}\nabla_\nu K+\frac{4P}{\Phi^2}K_{\nu i}D^i p-2\gspann{\nabla p,\nu}\eta^\frac{1}{2}.
			\end{align*}
			We now assume that $\varphi$ takes a new maximum in both space and time at a point $(t_0,x_0)$. Hence we compute at $(t_0,x_0)$
			\begin{align*}
			0
			\le&\frac{2H}{\Phi^2}\eta^{\frac{1}{2}}\Delta_\gamma p-\frac{2H}{\Phi^2}\btr{D p}^2-\frac{H\eta}{\Phi^2}\left(\Ric(\nu,\nu)+\btr{h}^2\right)\\
			&-\tr_{\gamma}\nabla_\nu K\frac{P\eta}{\Phi^2}+\frac{4P\eta^{\frac{1}{2}}}{\Phi^2}K_{\nu i} D^i p-2\gspann{\nabla p,\nu}\eta^\frac{1}{2}.
			\end{align*}
			By assumption we know that $H>0$, so 
			\begin{align*}
			\Phi=\sqrt{H^2-P^2}\le H\le\Phi+\btr{P},
			\end{align*}
			and therefore, we conclude that
			\[
			\frac{\btr{h}^2}{\Phi^2}\ge\frac{\btr{h}^2}{H^2}\ge\frac{1}{n}.
			\]
			Using this, along with the Cauchy--Schwarz inequality together with the assumptions on $\Ric$ and $p$, we get
			\begin{align*}
			0
			\le&6(n+1)\frac{\Phi +\btr{P}}{\Phi^2}\eta^{\frac{1}{2}}+\frac{18}{\Phi}R^2+\frac{\varphi +\btr{P}\eta}{100(n+1)\Phi^2R^2}-\frac{\varphi}{n}\\
			&+n\sup\limits_{\overline B_R(x)}\btr{\nabla K}_g\frac{\btr{P}\eta}{\Phi^2}+\frac{12\btr{P}\eta^{\frac{1}{2}}}{\Phi^2}\sup\limits_{\overline B_R(x)}\btr{K}_gR+6R\eta^\frac{1}{2},
			\end{align*}
			where we used $-H\le-\Phi$, whenever $-H$ is multiplied by a non-negative factor and $H\le\Phi+\btr{P}$ otherwise. We now have
			\begin{align*}
			\varphi
			\le& n\left(6(n+1)\frac{\Phi +\btr{P}}{\Phi^2}\eta^{\frac{1}{2}}+\frac{18}{\Phi}R^2+\frac{\varphi +\btr{P}\eta}{100(n+1)\Phi^2R^2}\right)\\
			&+n\left(n\sup\limits_{\overline B_R(x)}\btr{\nabla K}_g\frac{\btr{P}\eta}{\Phi^2}+\frac{12\btr{P}\eta^{\frac{1}{2}}}{\Phi^2}\sup\limits_{\overline B_R(x)}\btr{K}_gR+6R\eta^\frac{1}{2}\right).
			\end{align*}
			Finally, we use $\btr{P}(t_0,x_0)\le n\sup\limits_{\overline B_R(x)}\btr{K}_g$ and a rearrangement for the sake of clarity yields
			\begin{align*}
			\varphi
			\le&n\left(6(n+1)\frac{1 }{\Phi}\eta^{\frac{1}{2}}+\frac{18}{\Phi}R^2+\frac{\varphi }{100(n+1)\Phi^2R^2}+6R\eta^\frac{1}{2}\right)\\
			&+n\left(n^2\sup\limits_{\overline B_R(x)}\btr{K}_g\sup\limits_{\overline B_R(x)}\btr{\nabla K}_g\frac{\eta}{\Phi^2}+12n\sup\limits_{\overline B_R(x)}\btr{K}_g^2\frac{\eta^{\frac{1}{2}}}{\Phi^2}R\right)\\
			&+n\left(6(n+1)n\sup\limits_{\overline B_R(x)}\btr{K}_g\frac{\eta^{\frac{1}{2}}}{\Phi^2}+n\sup\limits_{\overline B_R(x)}\btr{K}_g\frac{\eta}{100(n+1)\Phi^2R^2}\right).
			\end{align*}
			We multiply the equation by $\varphi^2=\eta^2\cdot\Phi^2$, and obtain
			\begin{align*}
			\varphi^3
			\le
			&C(n)\left(\eta^{\frac{3}{2}}\varphi+R^2\eta\varphi\frac{\eta^2}{R^2}\varphi+R\eta^{\frac{1}{2}}\varphi^2 \right)\\
			&+C(n)\left(\sup\limits_{\overline B_R(x)}\btr{K}_g\left(\eta^{\frac{5}{2}}+\frac{\eta^3}{R^2}\right)+\sup\limits_{\overline B_R(x)}\btr{K}_g^2\eta^{\frac{5}{2}}R+\sup\limits_{\overline B_R(x)}\btr{K}_g\sup\limits_{\overline B_R(x)}\btr{\nabla K}_g\eta^3\right).
			\end{align*}
			Here $C(n)$ is an appropriate positive constant only depending on $n$, and which may vary from line to line in the following estimates. 
			We now employ Youngs inequality for a fixed $\varepsilon>0$ and obtain
			\begin{align*}
			\varphi^3
			\le&
			C(n)\left(\frac{4\varepsilon}{4}\varphi^3+\frac{4}{\varepsilon^2}
			\left(\eta^\frac{9}{4}+2R^3\eta^{\frac{3}{2}}+\frac{\eta^3}{R^3}\right)\right)\\
			&+C(n)\left(\sup\limits_{\overline B_R(x)}\btr{K}_g\left(\eta^{\frac{5}{2}}+\frac{\eta^3}{R^2}\right)+\sup\limits_{\overline B_R(x)}\btr{K}_g^2\eta^{\frac{5}{2}}R+\sup\limits_{\overline B_R(x)}\btr{K}_g\sup\limits_{\overline B_R(x)}\btr{\nabla K}_g\eta^3\right).
			\end{align*}
			It follows that
			\begin{align*}
			\left(1-C(n)\varepsilon\right)\varphi^3
			\le&
			C(n)\frac{4}{\varepsilon^2}
			\left(\eta^\frac{9}{4}+2R^3\eta^{\frac{3}{2}}+\frac{\eta^3}{R^3}\right)+C(n)\sup\limits_{\overline B_R(x)}\btr{K}_g\left(\eta^{\frac{5}{2}}+\frac{\eta^3}{R^2}\right)\\
			&+C(n)\left(\sup\limits_{\overline B_R(x)}\btr{K}_g^2\eta^{\frac{5}{2}}R+\sup\limits_{\overline B_R(x)}\btr{K}_g\sup\limits_{\overline B_R(x)}\btr{\nabla K}_g\eta^3\right).
			\end{align*}
			Choosing $0<\varepsilon(n)<\frac{1}{C(n)}$, we obtain the upper bound
			\begin{align*}
			\varphi^3
			\le&
			C(n)
			\left(\eta^\frac{9}{4}+2R^3\eta^{\frac{3}{2}}+\frac{\eta^3}{R^3}\right)+C(n)\sup\limits_{\overline B_R(x)}\btr{K}_g\left(\eta^{\frac{5}{2}}+\frac{\eta^3}{R^2}\right)\\
			&+C(n)\left(\sup\limits_{\overline B_R(x)}\btr{K}_g^2\eta^{\frac{5}{2}}R+\sup\limits_{\overline B_R(x)}\btr{K}_g\sup\limits_{\overline B_R(x)}\btr{\nabla K}_g\eta^3\right),
			\end{align*}
			at the maximal point $(t_0,x_0)$, and the righthand-side is now independent of $\varphi$. Since $\eta\le R^4$,
			\begin{align*}
			\varphi^3
			\le&
			4C(n)R^9
			+C(n)R^{12}\left(\frac{2}{R^2}\sup\limits_{\overline B_R(x)}\btr{K}_g+\frac{1}{R}\sup\limits_{\overline B_R(x)}\btr{K}_g^2+\sup\limits_{\overline B_R(x)}\btr{K}_g\sup\limits_{\overline B_R(x)}\btr{\nabla K}_g\right).
			\end{align*}
			Using Youngs inequality again, we find
			\begin{align*}
			\varphi^3(t_0,x_0)
			\le C(n)R^9+C(n)R^{12}\left(\sup\limits_{\overline B_R(x)}\btr{K}_g^3+\sup\limits_{\overline B_R(x)}\btr{\nabla K}_g^\frac{3}{2}\right).
			\end{align*}
			Hence, we can conclude that
			\begin{align*}
			\sup\limits_{F(\Sigma,[0,t])\cap B_{R(x)}}\varphi\le	\max\left(\max\limits_{\Sigma_0\cap B_R(x)}\varphi,C(n)\left(R^3+R^4\left(\sup\limits_{\overline B_R(x)}\btr{K}_g+\sup\limits_{\overline B_R(x)}\btr{\nabla K}_g^\frac{1}{2}\right)\right)\right).
			\end{align*}
			Finally, $p\le\frac{3}{2}\d_x$, so $\eta\vert_{B_{\frac{R}{2}}(x)}\ge\frac{25}{64}R^4$, which gives the desired estimate
			\begin{align*}
			\sup\limits_{F(\Sigma,[0,t])\cap B_{\nicefrac{R}{2}}(x)}\Phi\le
			\max\left(\max\limits_{\Sigma_0\cap B_R(x)}\Phi,C(n)\left(\frac{1}{R}+\sup\limits_{\overline B_R(x)}\btr{K}_g+\sup\limits_{\overline B_R(x)}\btr{\nabla K}_g^\frac{1}{2}\right)\right).
			\end{align*}
		\end{proof}
	
	\section{Level-set description and elliptic regularisation}\label{sec_levelsetdescription}
		To reformulate STIMCF as a level-set flow, we first assume that $\{\Sigma_t\}$ is a family of smooth hypersurfaces given as level-sets
		\[
		\Sigma_t=\partial\{x\in M\vert u(x)<t\}
		\]
		of a smooth scalar function $u\colon M\to \R$ with $\nabla_M u\not=0$. Then $u(y)=t$ if and only if there exists $x\in\Sigma$, such that $F(x,t)=y$, and we call $u$ the \emph{time-of-arrival function}. Since for fixed $x\in\Sigma$, $u\circ F(x,\cdot)$ is the identity map on the existence interval of $F$ we conclude that
		\begin{align}\label{prePDE}
		\btr{\nabla u}=\sqrt{H^2-P^2}.
		\end{align}
		In this smooth setting we have		
		\begin{align*}
		H=\dive_{g}\left(\frac{\nabla u}{\btr{\nabla u}_g}\right) \text{ and }
		P=\left(g^{\alpha\beta}-\frac{\nabla^{\alpha} u \nabla^{\beta} u}{\btr{\nabla u}^2_g}\right)K_{\alpha\beta}
		\end{align*}
		such that \eqref{prePDE} can be rearranged as
		\begin{align*}
		\dive_g\left(\frac{\nabla u}{\btr{\nabla u}_g}\right)= + \sqrt{\btr{\nabla u}_g^2+\left(\big(g^{\alpha\beta}-\frac{\nabla^{\alpha} u\nabla^{\beta} u}{\btr{\nabla u}^2_g}\big)K_{\alpha\beta}\right)^2}.
		\end{align*}
		The sign on the RHS is chosen such that STIMCF is consistent with inverse mean curvature flow in the time symmetric case and is further necessary to apply Theorem \ref{thm_interiorgradientestimate} below. If we now assume that $\Sigma_0=\partial E_0$, where $E_0$ is a  precompact $C^2$-domain in $M$, we are led to the degenerate elliptic boundary value problem
		\begin{align}\label{mainPDE}
		\begin{cases}
		\dive_g\left(\frac{\nabla u}{\btr{\nabla u}_g}\right)-\sqrt{\btr{\nabla u}_g^2+\left(\big(g^{\alpha\beta}-\frac{\nabla^{\alpha} u \nabla^{\beta} u}{\btr{\nabla u}^2_g}\big)K_{\alpha\beta}\right)^2}=0,\\
		u\vert_{\partial E_0}=0.
		\end{cases}
		\end{align}
		
		\begin{figure}[h]
			\centering
			\includegraphics{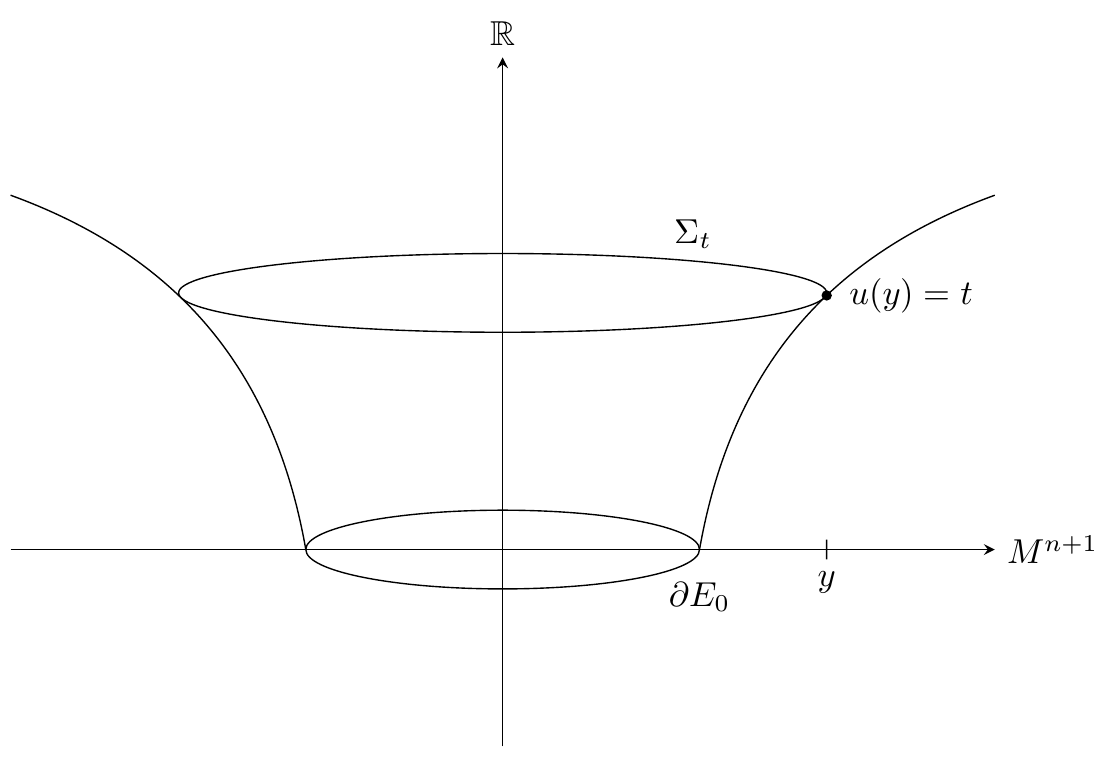}
			\caption{Time-of-arrival function $u$.}
		\end{figure}
		We want to find weak solutions of \eqref{mainPDE} by elliptic regularisation in a broad class of asymptotically flat exterior regions and begin with the construction of subsolutions, motivated by the smooth spherical solutions $u(x)=n\ln\left(\frac{\btr{x}}{R_0}\right)$ in the trivial initial data set $(\R^{n+1},\delta,0)$.
		\begin{lem}\label{lem_subsolution} Let $0<\alpha<n$ and let $(M, g, K)$ be an asymptotically flat exterior region with coordinate system $x: M\setminus \overline{\mathcal{B}}\to  \R^{n+1}\setminus \overline{B_{1}(0)}$ as in Definition \ref{defi_asymflat}. Then, using the notation  $\mathcal{O}_{R_0}= x^{-1}\left(\R^{n+1}\setminus \overline{B}_{R_0}(0)\right)$, there exists $R_0=R_0(\alpha,g,K)>1$, such that
			\[
			v\colon \mathcal{O}_{R_0}\to\R, \, v(y) = \alpha\ln(\btr{x(y)})-\alpha\ln(R_0),
			\]
			is a smooth strict subsolution of \eqref{mainPDE} with $v\vert_{\partial \mathcal{O}_{R_0}}=0$.
		\end{lem}
		\begin{proof}
			For the sake of simplicity, we express the decay in Definition \ref{defi_asymflat} in terms of Landau symbols, e.g., $g_{ij}=\delta_{ij}+O\left(\frac{1}{\btr{x}^{n-1}}\right)$. W.l.o.g. $R_0>1$, so $\mathcal{O}_{R_0}$ is well-defined and a straightforward computation yields
			\begin{align*}
			\dive_g\left(\frac{\nabla v}{\btr{\nabla v}_g}\right)&-\sqrt{\btr{\nabla v}_g^2+\left(\big(g^{\alpha\beta}-\frac{\nabla^{\alpha} v \nabla^{\beta} v}{\btr{\nabla v}^2_g}\big)K_{\alpha\beta}\right)^2}\ge\frac{1}{\btr{x}}\left((n-\alpha)+\landau{\frac{1}{\btr{x}^{n-1}}}\right).
			\end{align*}
			So $v$ is a subsolution of \eqref{mainPDE}, if $\btr{x}>R_0$, where $R_0=R_0(\alpha,g,K)>1$ can be chosen appropriately.
		\end{proof}
		\begin{bem}\label{bem_subsolution}
			Note that $v$ will remain a subsolution of the elliptic regularisation \eqref{ellreg} defined below on compact regions for sufficiently small $\varepsilon>0$  with respect to the same $R_0$.\\
			Furthermore, $v$ remains a strict subsolution, even if we assume the weaker decay assumptions \eqref{asympdecay2} below in Section \ref{sec_asymptoticbehavior}.
		\end{bem}
		Similar inverse mean curvature flow we expect jump regions in solutions where $\nabla_M u=0$ and \eqref{mainPDE} is not well-defined. To address this problem we use the method of \emph{elliptic regularisation} and approximate weak solutions to inverse space-time mean curvature flow by smooth solutions of strictly elliptic equations. Let $\varepsilon>0$ and consider the following strictly elliptic quasilinear PDE, writing now $\btr{\nabla u } = \btr{\nabla u}_g$ for simplicity,
		\begin{align}\label{ellreg}
		\dive_g\left(\frac{\nabla u_\varepsilon}{\sqrt{\varepsilon^2+\btr{\nabla u_\varepsilon}^2}}\right)-\sqrt{\varepsilon^2+\btr{\nabla u_\varepsilon}^2+\left(\big(g^{\alpha\beta}-\frac{\nabla^{\alpha} u_\varepsilon \nabla^{\beta} u_\varepsilon}{\varepsilon^2+\btr{\nabla u_\varepsilon}^2}\big)K_{\alpha\beta}\right)^2}=0.
		\end{align}
		Rescaling \eqref{ellreg} via $\widehat u_{\varepsilon}\definedas\frac{u_\varepsilon}{\varepsilon}$ gives
		\begin{align*}
		\dive_g\left(\frac{\nabla \widehat{u}_\varepsilon}{\sqrt{1+\btr{\nabla \widehat{u}_\varepsilon}^2}}\right)-\sqrt{\varepsilon^2+\varepsilon^2\btr{\nabla \widehat{u}_\varepsilon}^2+\left(\big(g^{\alpha\beta}-\frac{\nabla^{\alpha} \widehat{u}_\varepsilon \nabla^{\beta} \widehat{u}_\varepsilon}{1+\btr{\nabla \widehat{u}_\varepsilon}^2}\big)K_{\alpha\beta}\right)^2}=0,
		\end{align*}
		which is equivalent to
		\begin{align*}
		\sqrt{\dive_g\left(\frac{\nabla \widehat{u}_\varepsilon}{\sqrt{1+\btr{\nabla \widehat{u}_\varepsilon}^2}}\right)^2-\left(\big(g^{\alpha\beta}-\frac{\nabla^{\alpha} \widehat{u}_\varepsilon \nabla^{\beta} \widehat{u}_\varepsilon}{1+\btr{\nabla \widehat{u}_\varepsilon}^2}\big)K_{\alpha\beta}\right)^2}=\varepsilon\sqrt{1+\btr{\nabla \widehat{u}_{\varepsilon}}^2}.
		\end{align*}
		Then the left-hand side corresponds to the term $\sqrt{\widehat{H}^2_\varepsilon-\widehat P^2}$ of the hypersurfaces \linebreak $\widehat\Sigma^\varepsilon_t\definedas\graph\left(\widehat u_\varepsilon-\frac{t}{\varepsilon} \right)$ in the product manifold $(M \times\R, g+\d z^2,\widetilde{K})$, where we extent $K$ onto $M \times\R$ by $\widetilde{K}_{ij}\definedas K_{ij}$, $\widetilde{K}_{iz}=\widetilde{K}_{zz}\definedas0$. So the downward translating graphs 
$\widehat\Sigma_t^\varepsilon$ solve \eqref{mainPDE} in $M \times\R$ with $U_\varepsilon(y,z)=u_\varepsilon(y)-\varepsilon z$, since $\widehat\Sigma^\varepsilon_t=\{U_\varepsilon(y,z)=t\}$. Equivalently, given smooth solutions $u_\varepsilon$ to \eqref{ellreg}, the hyper\-surfaces $\widehat\Sigma_t^\varepsilon$ are smooth translating solutions of STIMCF in 
		$M\times \R$ with $\sqrt{\widehat H_\varepsilon-\widehat P^2}\Big\vert_{\widehat\Sigma_t^\varepsilon}>0$ along the hypersurfaces.
		\begin{figure}[h]
			\centering
			\includegraphics[scale=0.76]{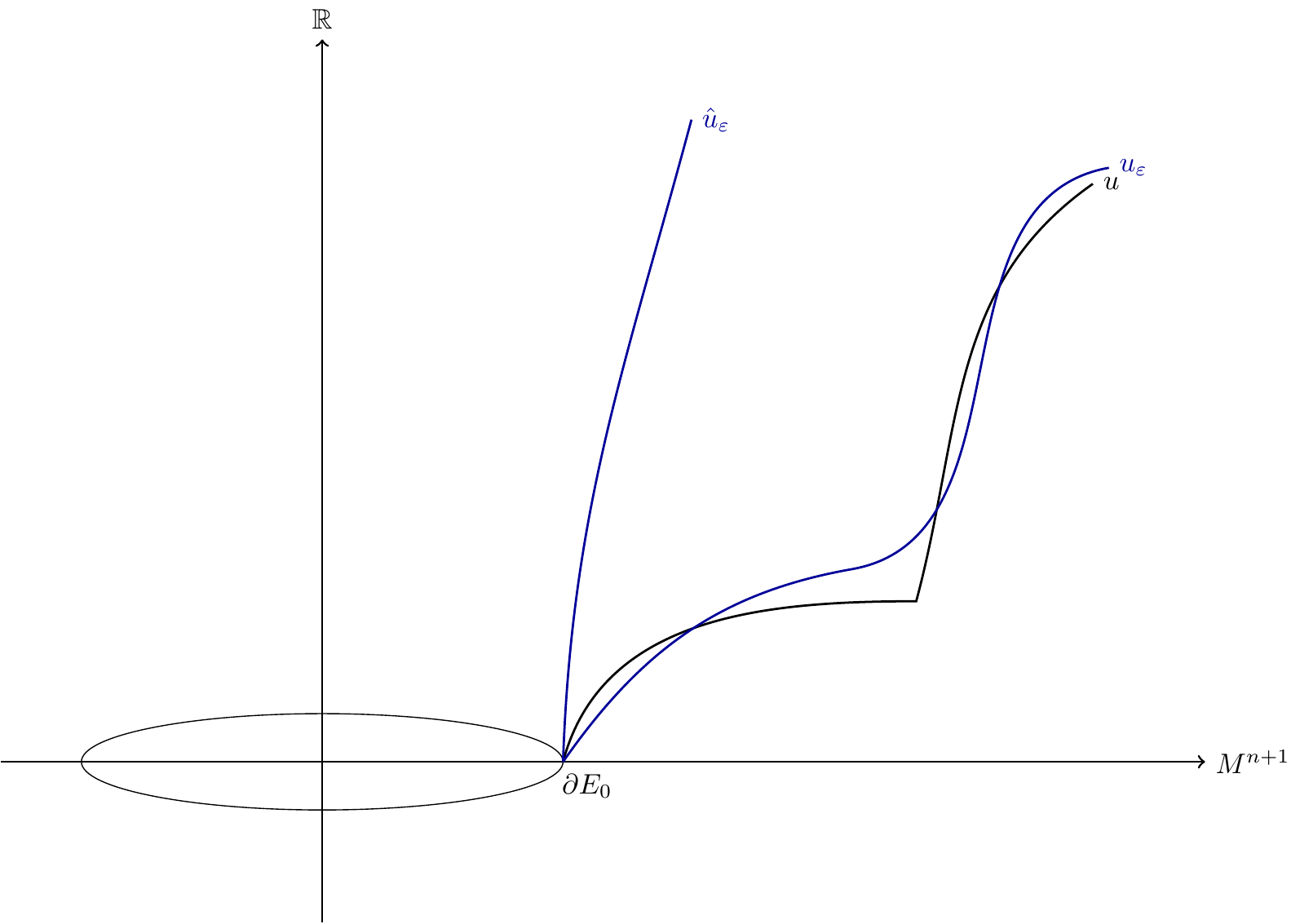}
			\caption{The time-of-arrival function $u$, the elliptic regularisation $u_\varepsilon$,\newline and the rescaling $\hat u_{\varepsilon}$ .}
		\end{figure}
		
We will dedicate the rest of this section to the existence of smooth solutions of the elliptic regularisation \eqref{ellreg}. 
Suppose that the initial hypersurface $\Sigma_0=\partial E_0$ is given as boundary of a precompact domain 
$E_0 \subset M$ and let $F\subset (M\setminus E_0)$ be another precompact domain. Then a solution $u_\varepsilon$ of the
regularisation \eqref{ellreg} can only exist on $F$ if $\varepsilon >0$ is sufficiently small, since the rescaled solution
$\widehat u_\varepsilon$ will satisfy
\begin{align*}
		\varepsilon\btr{F}&\le
		\int\limits_F\varepsilon\sqrt{1+\btr{\nabla \widehat{u}_\varepsilon}^2}\d y
		=\int\limits_F\dive_g\left(\frac{\nabla \widehat{u}_\varepsilon}{\sqrt{1+\btr{\nabla \widehat{u}_\varepsilon}^2}}\right)\d y=\int\limits_{\partial F}\frac{g(\nabla \widehat{u}_\varepsilon, \eta)}{\sqrt{1+\btr{\nabla \widehat{u}_\varepsilon}^2}}\d \sigma\le\btr{\partial F},
		\end{align*}
where $\eta$ is the unit normal to $\partial U$. Therefore we need to specify boundary data for \eqref{ellreg} on precompact domains $\Omega_L$ exhausting $M\setminus E_0$ as $L\to \infty$: We use the subsolution $v$ as in Lemma \ref{lem_subsolution} to define the domains $F_L=\{x\in M\colon x\in M\setminus \mathcal{O}_{R_0}\text{ or }v(x)<L\}$ for all $0\leq L < \infty$. 
On $\Omega_L\definedas F_L\setminus\overline E_0$ we consider the boundary value problem
		\begin{align}\label{initialellreg}
		\begin{cases}
		E^\varepsilon u_\varepsilon=0&\text{on }\Omega_L,\\
		u_\varepsilon=0&\text{on }\partial E_0,\\
		u_\varepsilon=L-2&\text{on }\partial F_L,
		\end{cases}
		\end{align}
		where  $E^\varepsilon u_\varepsilon$ is the ellipitic regularisation on the LHS of \eqref{ellreg}. As we only consider initial data sets with one asymptotically flat end and $v(x) \to \infty$ as $|x| \to \infty$, the domain $\Omega_L\definedas F_L\setminus\overline E_0$ is precompact. If $M$ contains multiple asymptotic ends, we can also proceed as in the following under the additional assumption, that $E_0$ contains all but one end.
		\begin{bem}\label{methodofcont}
			We prove existence of smooth solutions $u_\varepsilon$ of the elliptic regularisation \eqref{initialellreg} by using the method of continuity. For $s\in[0,1]$, we consider the boundary value problem
			\begin{align}\label{pdemethodcont}
			\begin{cases}
			E^{\varepsilon,s} u_{\varepsilon,s}=0&\text{on }\Omega_L,\\
			u_{\varepsilon,s}=0&\text{on }\partial E_0,\\
			u_{\varepsilon,s}=s(L-2)&\text{on }\partial F_L,
			\end{cases}
			\end{align}
			where the operator $E^{\varepsilon,s}u_{\varepsilon,s}$ is defined as
			\begin{align*}
			\dive_g\left(\frac{\nabla u_{\varepsilon,s}}{\sqrt{\varepsilon^2+\btr{\nabla u_{\varepsilon,s}}^2}}\right)-\sqrt{\varepsilon^2+\btr{\nabla u_{\varepsilon,s}}^2+s\left(\left(g^{\alpha\beta}-\frac{\nabla^{\alpha} u_{\varepsilon,s}\nabla^{\beta} u_{\varepsilon,s}}{\varepsilon^2+\btr{\nabla u_{\varepsilon,s}}^2}\right)K_{\alpha\beta}\right)^2},
			\end{align*}
			 and aim to show that $S_\varepsilon\definedas\{s\in[0,1]\colon E^{\varepsilon,s}\text{ admits a unique } C^{2,\alpha}(\overline{\Omega}_L)\text{ solution}\}=[0,1]$ for sufficiently small $\varepsilon>0$.
		\end{bem}
		From now on, we also impose the additional maximality condition $\tr_g K\equiv0$ on the initial data set $(M,g,K)$ to ensure the existence of a subsolution barrier in the compact region. Therefore the quasilinear operator $E^{\varepsilon,s}$ becomes
		\begin{align}\label{maxpde}
		E^{\varepsilon,s}=\dive_g\left(\frac{\nabla u_{\varepsilon,s}}{\sqrt{\varepsilon^2+\btr{\nabla u_{\varepsilon,s}}^2}}\right)-\sqrt{\varepsilon^2+\btr{\nabla u_{\varepsilon,s}}^2+s\left(\frac{\nabla^{\alpha} u_{\varepsilon,s}\nabla^{\beta} u_{\varepsilon,s}}{\varepsilon^2+\btr{\nabla u_{\varepsilon,s}}^2}K_{\alpha\beta}\right)^2}.
		\end{align}
		In view of the work of Cederbaum--Sakovich \cite{cederbaumsakovich} we can also interpret the method of continuity as modifying the underlying initial data set $(M,g,K)$ via $(M_s,g_s,K_s)\definedas(M,g,\sqrt{s}K)$ instead of the operator $E^{\varepsilon}$.
		\begin{thm}\label{thm_aprioriestimates}
			Let $(M,g,K)$ be an asymptotically flat, maximal initial data set and $E_0\subset M$ a precompact $C^{2,\alpha}$ domain. Then, for every $L>2$, satisfying $E_0\subset F_L$ and $d(\partial E_0,\partial F_L)>2$, there exists an $\varepsilon(L)>0$, such that for $0<\varepsilon<\varepsilon(L)$ and $s\in[0,1]$ a smooth solution $u_{\varepsilon,s}$ to \eqref{pdemethodcont} satisfies the following a-priori estimates:
			\begin{enumerate}
				\item[\emph{(i)}] $u_{\varepsilon,s}\ge-\varepsilon$ in $\overline{\Omega}_L$, $u_{\varepsilon,s}\ge v+(s-1)(L-1)-2$ in $\overline{F}_L\setminus F_0$,
				\item[\emph{(ii)}] $u_{\varepsilon,s}\le s(L-2)$ in $\overline{\Omega}_L$,
				\item[\emph{(iii)}] $\btr{\nabla u_{\varepsilon,s}}\le H_++\varepsilon$ on $\partial E_0$, $\btr{\nabla u_{\varepsilon,s}}\le C(L)$ on $\partial F_L$,
				\item[\emph{(iv)}] $\btr{\nabla u_{\varepsilon,s}} (y)\le \max\limits_{B_r(x)\cap\partial E_0}\btr{\nabla u}+\varepsilon+\frac{C(n)}{r}+C(n,\norm{K},\norm{\nabla K})$ for all $y\in\overline{\Omega}_L$,
				\item[\emph{(v)}] $\btr{u_{\varepsilon,s}}_{C^{2,\alpha}(\overline{\Omega}_L)}\le C(\varepsilon,L,n,g,\norm{K},\norm{\nabla K},\partial E_0)$,
			\end{enumerate}
			where $H_+$ is the positive part of the mean curvature of $\partial E_0$ and $r>0$ such that the conditions of Theorem \ref{thm_interiorgradientestimate} are satisfied at $y$ with $r$.
		\end{thm}
		\begin{bem}\label{bem_aprioriestimates}
			Note that $E^{\varepsilon,s}$ is uniformly elliptic for any smooth function on the compact domain $\Omega_L$ and the assumptions on $L$ ensure that the boundary data of \eqref{methodofcont} are realized by a $C^{2,\alpha}$-function $\varphi$, such that $\norm{\varphi}_{C^{2,\alpha}(\overline{\Omega_L})}\le C(\partial E_0,L,g,n,\norm{K})$. For $C^2$-domains, solutions in $C^{2,\alpha}(\Omega_L)\cap C^{1,\alpha}(\overline{\Omega}_L)$ still satisfying (i)-(iv) exist by approximation.
		\end{bem}
		\begin{proof}
			We prove (i)--(iii) by constructing appropriate barrier functions.
			
			\noindent Note that for fixed $\varepsilon, s$, with $\varepsilon$ sufficiently small (only depending on $L$), the elliptic regularisation $u_{\varepsilon,s}^{(IMCF)}$ corresponding to inverse mean curvature flow satisfying $(\star)_{\varepsilon,s}$ in \cite[p. 384]{huiskenilmanen} exists for the same boundary data due to \cite[Lemma 3.5]{huiskenilmanen}, and is a supersolution to \eqref{pdemethodcont}. By the maximum principle, $u_{\varepsilon,s}^{(IMCF)}$ is an upper barrier function, in particular
			\[
				u_{\varepsilon,s}\le u_{\varepsilon,s}^{(IMCF)}\le s(L-2),
			\]
			so (ii) is proven.
			
			Once we establish a lower bound for $u_{\varepsilon,s}$ (we in fact show $u_{\varepsilon,s}\ge -\varepsilon\ge -2$), it is straightforward, that an appropriate translation of the subsolution $v$ is a lower barrier function in the asymptotic region. More precisely, 
			\[
				u_{\varepsilon,s}\ge v_2\definedas v+(s-1)(L-2)\text{ in }\overline{F_L}\setminus F_0.
			\]
			Thus, to establish (i), it remains to show that $u_{\varepsilon,s}\ge -\varepsilon$ for $\varepsilon$ sufficiently small (only depending on $L$). We will do so by constructing a lower barrier that bridges between $\partial E_0$ to where $v$ is defined in the asymptotic region, allowing for unrestricted jumps in the compact region of $M$. Note that the assumption $\tr_g K=0$ will be crucial in establishing  the existence of this subsolution.
			Recall that that the distance function $d(\cdot,E_0)$ of $E_0$ is smooth on $M\setminus\left(\overline{E}_0\cup Cut(E_0)\right)$, where $Cut(E_0)$ denotes the Cut locus of $E_0$. We define $G_0\definedas E_0$ and $G_b\definedas \{y\colon d(y,E_0)<b\}$ and choose $b=b_L$ large enough, such that $F_L\subseteq G_{b_L}$. By \cite[Theorem 3.2]{huiskenpolden}, the evolution of the mean curvature for a surface moving in normal direction $\nu$ with speed $f$ is given by
			the Jacobi operator on the hypersurface applied to $f$ such that we obtain for the evolution of the mean curvature of the hypersurfaces $\partial G_b$ with unit normal $\nu_{b}$
			\[
			\frac{\d}{\d b}H=-\left(\Ric_g(\nu_{b},\nu_b)+\btr{h_b}^2\right)\le-\Ric(\nu_b,\nu_b)\le C_1(L)\text{ on }\partial G_b\setminus Cut(E_0)\text{, }0\le b\le b_L,
			\]
			where $C_1(L)\definedas(n+1)\max\limits_{i,j\in\{1,\dotsc,n+1\}}\max\limits_{y\in\overline{G}_{b_L}, }\btr{\Ric_{ij}(y)}$. Therefore, we can estimate the mean curvature of the surfaces $\partial G_b$, by
			\begin{align*}
			H_{\partial G_b}\le \max\limits_{\partial E_0}H_++C_1b\le C_2(L,\partial E_0)&&\text{ on }\partial G_b\setminus Cut(E_0)\text{, }0\le b\le b_L,
			\end{align*}
			where $H_+\definedas \max(0,H)$. We denote the maximum of the absolute value of the largest eigenvalue of $K$ over $\Omega_L$ by $\btr{\lambda}$ and for $A\definedas 2(C_2+\btr{\lambda} +4)>1$, we define the function $f(b)\definedas\frac{\varepsilon}{A}\left(-1+e^{-Ab}\right)$. If $\varepsilon\le e^{-Ab_L}$, it holds
			\begin{align}\label{eq_subsolODI}
			f'(C_2+2+\btr{\lambda})+\frac{\varepsilon^2f''}{\varepsilon^2+f'^2}-f'^2-\varepsilon^2>0
			\end{align}
			From now on, we will always assume that $\varepsilon\le e^{-Ab_L}$, and consider the function\linebreak $v_1(y)\definedas f(d(y, E_0))$. In particular, $\nabla v_1=f'\nu_b$, and
			\begin{align*}
			\dive_g\left(\frac{\nabla v_1}{\sqrt{\varepsilon^2+\btr{\nabla v_1}^2}}\right)
			&=	\dive_{\partial G_L}\left(\frac{\nabla v_1}{\sqrt{\varepsilon^2+\btr{\nabla v_1}^2}}\right)+\left(\nabla_{v_b}\left(\frac{\nabla v_1}{\sqrt{\varepsilon^2+\btr{\nabla v_1}^2}}\right)\right)^\perp\\
			&=\dive_{\partial G_L}\left(\frac{f'\nu_b}{\sqrt{\varepsilon^2+f'^2}}\right)+\left(\nabla_{v_b}\left(\frac{f'\nu_b}{\sqrt{\varepsilon^2+f'^2}}\right)\right)^\perp\\
			&=\frac{f'}{\sqrt{\varepsilon^2+f'^2}}H_{\partial G_b}+\frac{\varepsilon^2f''}{\sqrt{\left(f'^2+\varepsilon^2\right)^3}}.
			\end{align*}
			Hence 
			\begin{align*}
			\sqrt{\varepsilon^2+f'^2}E^{\varepsilon}v_1\ge C_2f'+\frac{\varepsilon^2 f''}{f'^2+\varepsilon^2}-\sqrt{\varepsilon^2+f'^2}\left(\sqrt{\varepsilon^2+f'^2}+\frac{f'^2\btr{\lambda}}{\varepsilon^2+f'^2}\right)>0
			\end{align*}
			due to \eqref{eq_subsolODI}. As in the proof of \cite[Lemma 3.4]{huiskenilmanen}, $v_1$ is in fact a \emph{viscosity} subsolution of $E^{\varepsilon,s}$ on all of $\Omega_L\setminus \overline{E_0}$, and since $u_{\varepsilon,s}\ge v_1$ on the boundary, $u_{\varepsilon,s}\ge v_1$, by the maximum principle for viscosity solutions.
			
			To obtain the boundary gradient estimates (iii), we argue as in the proof of \cite[Lemma 3.4]{huiskenilmanen}, using that
			\[v_1\le u_{\varepsilon,s}\le u_{\varepsilon,s}^{(IMCF)}, \text{ and }v_2\le u_{\varepsilon,s}\le u_{\varepsilon,s}^{(IMCF)}\]
			with equality on $\partial E_0$ and $\partial F_L$, respectively, and using the boundary gradient estimate \cite[(3.8)]{huiskenilmanen} for $u_{\varepsilon,s}^{(IMCF)}$.
			So we established the a-priori estimates (i)--(iii) as desired.
			
			We obtain the interior gradient estimate (iv) using Theorem \ref{thm_interiorgradientestimate} on $(M_s,g_s,K_s)$. As for $s=1$, the downward translating graphs $\widetilde\Sigma^{\varepsilon,s}_t\definedas\graph\left(\frac{u_{\varepsilon,s}}{\varepsilon}-\frac{t}{\varepsilon}\right)$ are smooth solutions to the inverse geometric flow with inverse speed $\Phi_s\definedas\sqrt{H^2-sP^2}$, and have strictly positive mean curvature since $E^{\varepsilon,s}u_{\varepsilon,s}=0$. Since $U_{\varepsilon,s}$ is the time-of-arrival function of the downward translating graphs, this implies
			\[
			\sqrt{H^2-sP^2}\Big\vert_{\widetilde{\Sigma}^{\varepsilon,s}_t}=\sqrt{\varepsilon^2+\btr{\nabla 
			u_{\varepsilon,s}}^2}.
			\]
			Furthermore, by the definition of $(M \times \R, g+\d z^2,\widetilde{K})$, the assumptions on the radius $r$ of Theorem \ref{thm_interiorgradientestimate} are satisfied for $(y,z)\in M \times\R$ on the $(n+2)$-dimensional ball $B^{n+2}_r(y,z)$ if they are satisfied for the $(n+1)$-dimensional ball $B_r(y)\subseteq M$. Then Theorem \ref{thm_interiorgradientestimate} implies
			\begin{align*}
			\btr{\nabla u_{\varepsilon,s}}_g(y)
			&\le\sqrt{\varepsilon^2+\btr{\nabla u_{\varepsilon,s}}_g^2(y)}\\
			&=\sqrt{H^2-sP^2}\Big\vert_{\widetilde{\Sigma}^{\varepsilon,s}_t}(y,z)\\
			&\le \max\limits_{\left(\partial E_0\times\R\right)\cap B^{n+2}_r(y,z)}\sqrt{H^2-sP^2}\Big\vert_{\widetilde{\Sigma}^{\varepsilon,s}_t}+\frac{C(n)}{r}+C(n,\norm{K},\norm{\nabla K})\\
			&\le \max\limits_{\partial E_0\cap B_r(y)}\btr{\nabla_Mu_{\varepsilon,s}}_g+\varepsilon+\frac{C(n)}{r}+C(n,\norm{K},\norm{\nabla K}),
			\end{align*}
			for all $y\in\overline{\Omega}_L$.
			Since we can estimate $\btr{\nabla u_{\varepsilon,s}}$ at the boundary $\partial E_0$, we can conclude
			\begin{align}\label{lipestimate}
			\btr{\nabla u_{\varepsilon,s}}(y)\le \max\limits_{\partial E_0\cap B_r(y)}H_++2\varepsilon+\frac{C(n)}{r}+C(n,\norm{K},\norm{\nabla K}).
			\end{align}
			Since $\Omega_L$ is precompact there is an $r'=r'(L,g)$ satisfying the assumptions of Theorem \ref{thm_interiorgradientestimate} for all $y\in\Omega_L$, so we obtain the Lipschitz estimate
			\begin{align}
			\norm{u_{\varepsilon,s}}_{C^{1}(\overline{\Omega}_L )}\le C(L,n,g,\norm{K},\norm{\nabla K},\partial E_0).
			\end{align}
			The DeGiorgi--Nash Theorem, \cite[Theorem 13.2]{gilbargtrudinger}, implies
			\[
			\norm{u_{\varepsilon,s}}_{C^{1,\alpha}(\overline{\Omega}_L )}\le C(\varepsilon,L,n,g,\norm{K},\norm{\nabla K},\partial E_0).
			\]
			In particular, the coefficients of $E^{\varepsilon,s}u_{\varepsilon,s}$ are Hölder continuous, so Schauder theory, \cite[Theorem 6.17]{gilbargtrudinger}, improves this to a $C^{2,\alpha}$-estimate,
			\[
			\norm{u_{\varepsilon,s}}_{C^{2,\alpha}(\overline{\Omega}_L )}\le C(\varepsilon,L,n,g,\norm{K},\norm{\nabla K},\partial E_0).
			\]
		\end{proof}
			\begin{thm}\label{thm_ellipticregexistence}
			Let $(M,g,K)$ be an asympotically flat, maximal initial data set and $E_0\subseteq M$ a precompact $C^{2,\alpha}$ domain. Then, for every $L>2$, satisfying $E_0\subset F_L$ and $d(\partial E_0,F_L)>2$, there exists an $\varepsilon_0(L)\le\varepsilon(L)$, such that a smooth solution $u_{\varepsilon}$ of \eqref{initialellreg} exists for all $\varepsilon<\varepsilon_0(L)$.
		\end{thm}
		\begin{proof}
			Using the method of continuity, it will suffice to show that the set 
			\[
			S_\varepsilon\definedas\{s\in[0,1]\colon E^{\varepsilon,s}\text{ admits a unique } C^{2,\alpha}(\overline{\Omega}_L)\text{ solution}\},
			\]
			is non-empty, closed and relatively open with respect to the induced topology on $[0,1]$.
			
			\noindent We first show that $0\in S_\varepsilon$ for sufficiently small $0<\varepsilon<\varepsilon(L)$, in particular $S_\varepsilon$ non-empty.
			
			Note that for $s=0$, the operator $E^{0,\varepsilon}$ corresponds to the operator $E^{1,\varepsilon}$ on the inital data set $(M,g,0)$, so STIMCF reduces to inverse mean curvature flow. Then the Approximate Existence Lemma \cite[Lemma 3.5]{huiskenilmanen} applies, if $\varepsilon$ is sufficiently small. So $0\in S_\varepsilon$, if $\varepsilon<\varepsilon_1$.
			From now on, we fix $0<\varepsilon<\varepsilon_0(L)\definedas\min(\varepsilon(L),\varepsilon_1)$ and vary $s$.
			
			Next we show, that $S_\varepsilon$ is relatively open in $[0,1]$.
			
			\noindent So assume that $s_0\in S_\varepsilon$, i.e., there exists a unique $C^{2,\alpha}$-solution $u_{\varepsilon,s_0}$ to \eqref{pdemethodcont} for $s_0\in[0,1]$.
			For the boundary value map $\pi\colon u\mapsto u{\Big\vert_{\partial\Omega_L}}$, we see that the existence of a solution to \eqref{pdemethodcont} is equivalent to  $G(w,s)=0$, , where
			\begin{align*}
			G\colon C^{2,\alpha}(\overline{\Omega}_L)\times\R&\to C^\alpha(\overline{\Omega}_L)\times C^{2,\alpha}(\partial\overline{\Omega}_L)\\
			(w,s)&\mapsto(E^{\varepsilon,s}(w),\pi(w)-s(L-2)\chi_{\partial F_L}).
			\end{align*}
			Note that $G$ is continuously differentiable and the differential at $(u_{s_0},s_0)$ is given as
			\[
			DG_{(u_{s_0},s_0)}\colon C^{2,\alpha}_0(\overline{\Omega}_L)\to C^\alpha_0(\overline{\Omega}_L)\times C^{2,\alpha}_0(\partial\Omega_L), 
			h\mapsto (A^{ij}\operatorname{Hess} (h)_{ij}+B_k\nabla_Mh^k,\pi(h)),
			\]
			where $A^{ij}$, and $B_k$ are Hölder continuous, since $u_{s_0}\in C^{2,\alpha}$. Schauder theory \cite[6.14]{gilbargtrudinger} implies that $DG_{(u_{s_0},s_0)}$ is invertible with continuous inverse, due to the Schauder estimates, \cite[Theorem 6.6]{gilbargtrudinger}. Using the implicit function theorem, \cite[Theorem 17.6]{gilbargtrudinger} there exists $\delta>0$, such that $G(\cdot, s)$ has a unique solution $u_s\in C^{2,\alpha}(\overline{\Omega}_L)$ for all $s\in(s_0-\delta,s_0+\delta)$. Hence \eqref{pdemethodcont} has a unique solution $u_{\varepsilon,s}$ in $C^{2,\alpha}(\overline{\Omega}_L)$ for all $s\in[0,1]\cap(s_0-\delta,s_0+\delta)$. So $S_\varepsilon$ is relatively open in $[0,1]$. It remains to show that $S_\varepsilon$ is closed.
			
			Let $(s_k)_{k\in\N}$ be a sequence in $S_\varepsilon$, such that it exists an $s\in[0,1]$ with $s_k\to s$ as $k\to\infty$. Using Theorem \ref{thm_aprioriestimates} (5), we know that the solutions $u_{s_k}$ satisfy 
			\[
			\norm{u_{s_k}}_{C^{2,\alpha}(\overline{\Omega}_L)}\le C(\varepsilon,L,n,g,\norm{K},\norm{\nabla K},\partial E_0),
			\]
			independent on $k$. By Arzel\`a--Ascoli, we acquire a subsequence $(u_{k_i})_{i\in\N}$, such that $u_{k_i}\to u_{\infty}$ uniformly in $C^{2}(\overline{\Omega}_L)$. In particular, $u_{\infty}\in C^{2}(\overline{\Omega}_L)$, and $E^{\varepsilon,s}u_\infty=0$ with $u_{\infty}=0$ along $\partial E_0$, and $u_{\infty}=s(L-2)$ along $\partial F_L$, respectively. Schauder theory yields that $u_s=u_{\infty}\in C^{2,\alpha}(\overline{\Omega}_L)$ is the unique solution to \eqref{pdemethodcont} for $s$. Hence $s\in S_\varepsilon$, so $S_\varepsilon$ is closed.
		\end{proof}
		Since the $C^1$ a-priori estimates in Theorem \ref{thm_aprioriestimates} are independent on $\varepsilon$, we can pass via Arzel\`a--Ascoli to a subsequence $u_{\varepsilon_k}$, such that $u_{\varepsilon_k}\to u$ locally uniformly, where $u$ is locally Lipschitz.
		\begin{kor}\label{lem_sublimit}
			Let $(M,g,K)$ be an asymptotically flat, maximal initial data set and \linebreak $E_0 \subseteq M$ a  precompact $C^2$ domain. Then there exists a locally Lipschitz function \linebreak$u\colon M \setminus E_0\to \R$, such that:
			\begin{enumerate}
				\item[(i)] there exists a sequence $L_k\to\infty$, and a sequence $\varepsilon_k<\varepsilon_0(L)$ with $\varepsilon_k\to 0$, such that $u_{\varepsilon_k}\to u$ locally uniformly,
				\item[(ii)] $0\le u \le u^{(IMCF)}$, where $u^{(IMCF)}$ is the unique\footnote{with precompact level-sets as constructed in \cite{huiskenilmanen}} weak solution to inverse mean curvature flow on 
				$M \setminus E_0$,
				\item[(iii)] $u\to \infty$ as $\btr{x}\to \infty$,
				\item[(iv)] $\norm{u}_{C^{0,1}(\overline{B}_R(x))}\le \frac{C(n)}{R}+C(n,\norm{K},\norm{\nabla K})$, whenever $R<d(\partial E_0,x)$ and $R$ satisfies the assumptions of Theorem \ref{thm_interiorgradientestimate}.
			\end{enumerate}
		\end{kor}
		
\section{The limiting behavior of the downward-translating graphs $\Sigma_t^\varepsilon$}\label{sec_limitingbehavior}
	By the a-priori estimates \eqref{lipestimate} we concluded in Corollary \ref{lem_sublimit}, that there exists a locally Lipschitz function $u$ on $M\setminus E_0$, and a subsequence $(\varepsilon_i)$ such that $u_i\definedas u_{\varepsilon_i}$ converges to $u$ locally uniformly. We will show in this section that  for $n<6$ in addition to the function $u$ we can extract a limiting vectorfield $\nu \in C^{0, \alpha}((M\setminus E_0) \times \R)$ that will be crucial in defining a weak solution to the anisotropic equation (\ref{stimcf}). 
	
	The functions $U_i\definedas u_i-\varepsilon z$ defined on $\Omega_L\times \R$ converge locally uniformly to the locally Lipschitz function $U(y,z):= u(x)$. In this and the following section, we will concentrate on the objects on the cylinder $M\times \R$ and study the limiting behavior of the hypersurfaces $\widetilde{\Sigma}_t^\varepsilon\definedas\{U_\varepsilon=t\}$ in $M\times\R$.
	In Section \ref{sec_variationalformulation}, we will define weak solutions as minimizers to a parametric variation principle and want to argue that the sublimits $U$ and $u$ are indeed minimizers on $(M\setminus \overline{E}_0)\times\R$ and $(M\setminus \overline{E}_0)$ respectively. 
	As it is the case for inverse null mean curvature flow first studied by Moore in \cite{moore}, the introduced bulk term energy requires a notion of a unit vector field $\nu$ across jump regions. Following  their strategy, we first introduce a variational principle $\mathcal{J}_{U,\nu}$ for Caccioppoli sets on a compact subset $A$ inside a domain $\Omega$ defined as
	\begin{align}\label{comprintset}
	\mathcal{J}_{U,\nu}^A(F)\definedas\btr{\partial^*F\cap A}-\int\limits_{F\cap A}\sqrt{\btr{\nabla U}^2+P_\nu^2},
	\end{align}
	where $P_\nu\definedas\left(g^{ij}-\nu^i\nu^j\right)K_{ij}$ which reduces to $P_\nu=-\nu^i\nu^jK_{ij}$ here as we always impose that $\tr_MK\equiv0$.
	We say that $E$ minimizes \eqref{comprintset} in a set $\Omega$ (form the inside/ outside respectively), if 
	\begin{align}
	\mathcal{J}^A_{U,\nu}(E)\le \mathcal{J}^A_{U,\nu}(F)
	\end{align}
	for all $F$ ($F\subseteq E$, $E\subseteq F$ respectively), with $E\triangle F\subset\subset A\subset\subset \Omega$, where $\triangle$ denotes the symmetric difference. Since this does not depend on the particular choice of the compact set $A$, we will often omit the subscript $A$ in the following. Note that the well-known inequality
	\begin{align}\label{eq_reducedboundary}
	\btr{\partial^*(E_1\cup E_2)}+\btr{\partial^*(E_1\cap E_2)}\le \btr{\partial^*E_1}+\btr{\partial^*E_2},
	\end{align} 
	implies
	\begin{align}\label{ineqcaccio1}
	\mathcal{J}_{U,\nu}^A(E_1\cup E_2)+\mathcal{J}_{U,\nu}^A(E_1\cap E_2)\le \mathcal{J}_{U,\nu}^A(E_1)+\mathcal{J}_{U,\nu}^A(E_2),
	\end{align}
	for Caccioppoli sets $E_1$ and $E_2$ satisfying $E_1\triangle E_2\subset\subset A$. In particular, $E$ minimizes $\mathcal{J}_{U,\nu}$, if and only if $E$ minimizes $J_{u,\nu}$ from the outside and the inside.
	
	As already discussed, the bulk term $\sqrt{\btr{\nabla U}^2+P_\nu^2}$ requires a notion of unit normal $\nu$ on all of $\Omega$, since the canonical choice $\nu=\frac{\nabla U}{\btr{\nabla U}}$ fails across jump regions. Hence, the main task of this section will be to foliate the interior of jump regions of $U$ in $(M\setminus E_0)\times\R$, thus defining a notion of unit vector field $\nu$. In fact $\nu$ is translation invariant and in particular gives rise to a vector field $\nu_M\definedas\pi_{TM}\nu$ on $M$.
	
	To establish the existence of this foliation, we will draw heavily upon regularity theory for obstable problems \eqref{eq_obstacle} below. In particular, if a Caccioppoli set E minimizes \eqref{comprintset} and $\btr{\nabla U}$ admits an upper bound, $E$ is almost minimizing in the sense that
	\[
		\btr{\partial^*E\cap B_R}\le \btr{\partial^*F\cap B_R}+C(n,\norm{\nabla U}_{\infty},\norm{K}_{C^0})R^{n+2},
	\]
	for $E\triangle F\subset\subset B_R\subset M\times \R$. In particular, this allows us to apply regularity results of geometric measure theory to obtain higher regularity for $\partial^*E$. The following $C^{1,\alpha}$ result can be obtained by modifying the proof given in \cite{tamanini} for $\R^n$ to general Riemannian manifolds. We refer to the comments preceding \cite[Regularity Theorem 1.3]{huiskenilmanen} for a broad overview of references.
	\begin{thm}\label{thm_regularity}\emph{(Regularity Theorem)}\newline
		Let $f$ be a bounded, measurable function on a domain $\Omega \subset \widetilde{M}^m$ of a smooth Riemannian manifold $(\widetilde{M}^m, \widetilde{g})$ of dimension $m<8$. Suppose $E\subset \Omega$ contains an open set $\mathfrak{U}$ and minimizes the functional 
		\begin{align}\label{eq_obstacle}
		\btr{\partial^* F}+\int\limits_F f
		\end{align}
		with respect to competitors $F$, such that $\mathfrak{U}\subseteq F$ and $F\triangle E\subset\subset \Omega$. If $\partial \mathfrak{U}$ is $C^{1,\alpha}$, $0<\alpha<\frac{1}{2}$, then $\partial E$ is a $C^{1,\alpha}$-submanifold of $\Omega$ with $C^{1,\alpha}$ estimates only depending on the distance to $\partial \Omega$, $\mathrm{ess}\text{\,}\sup\btr{f}$, $C^{1,\alpha}$-bounds for $\partial \mathfrak{U}$ and $C^1$-bounds on the metric $\widetilde{g}$.
		
		When $m\ge 8$, this remains true away from a closed singular set $Z$ of dimension at most $m-8$, that is disjoint from $\overline{\mathfrak{U}}$.
	\end{thm}
	\begin{bem}
	From here onwards we will always assume that $n<6$ so the previous theorem applies in $M \times \R$ with $m=n+2$. If $n\geq 6$, the limit $u_{\varepsilon} \to u$ will lead to weak solutions of (\ref{stimcf}) with similar regularity properties away from the singular set.
	\end{bem}
	Another essential tool that will be used in this section is the following Compactness Theorem for Caccioppoli sets minimizing \eqref{comprintset}.
	\begin{thm}\label{thm_compactness1}\emph{(Compactness Theorem)}\newline
		Let $(\widetilde{M},\widetilde{g},\widetilde{K})$ be an initial data set, let $\Omega\subseteq \widetilde{M}$, and let $E_i\subseteq \Omega$ be a sequence of sets with $C^{1,\alpha}_{loc}$ boundary such that $\partial E_i\to\partial E$ locally in $C^{1,\alpha}$. Let $\nu_i$ be a unit vector field on $T\Omega$ satisfying $\nu_i\vert_{\partial E_i}=\nu_{\partial E_i}$, such that there exists a unit vector field $\nu\in T\Omega$, such that $\nu_i\to \nu$ a.e. locally uniformly and $\nu\vert_{\partial E}=\nu_{\partial E}$. Further, let $U_i\in C^{0,1}_{loc}(\Omega)$, such that $U_i\to U$ locally uniformly for an $U\in C^{0,1}_{loc}(\Omega)$ and
		\[
		\btr{\nabla U_i}\to\btr{\nabla U}\text{ in }\mathcal{L}^1_{loc}(\Omega).
		\]
		Then, if $E_i$ minimizes $\mathcal{J}_{U_i,\nu_i}$ in $\Omega$, $E$ minimizes $\mathcal{J}_{U,\nu}$ in $\Omega$.
	\end{thm}
	\begin{proof}
		As argued above, it suffices to show that $E$ minimizes $\mathcal{J}_{U,\nu}$ in $\Omega$ from the outside and from the inside. As both directions are similar in spirit, we merely show that $E$ minimizes $\mathcal{J}_{U,\nu}$ in $\Omega$ from the outside.
		
		So let $E\subseteq F$ such that $F\setminus E\subset\subset \Omega$ and let $G\subset\subset\Omega$ such that $F\setminus E\subset\subset G$. We further consider a compact set $\widetilde{G}\subset\subset\Omega$ with smooth boundary, such that $G\subset int(\widetilde{G})$,
		\[
		\btr{\partial^*(F\cup E_i)\cap\partial\widetilde{G}}=
		\btr{\partial^*(F\cap E_i)\cap\partial\widetilde{G}}=
		\btr{\partial^*E_i\cap\partial\widetilde{G}}=0
		\]
		for $i$ large enough, and $\int\limits_{\partial\widetilde{G}}\btr{\chi^-_{F\cup E_i}-\chi^+_{E_i}}\d \mathcal{H}^n\to 0$ as $i\to\infty$, which is possible since $F\cup E_i\to E$, $F\cap E_i\to E$, and $E_i\to E$ in $\mathcal{L}^1_{loc}(\Omega\setminus G)$. Setting $F_i\definedas E_i\cup (F\cap\widetilde{G})$, we see that
		\[
		\btr{\partial^*F_i\cap\Omega}
		=\btr{\partial^*E_i\cap\partial(\Omega\setminus\widetilde{G})}
		+\btr{\partial^*(F\cup E_i)\cap\widetilde{G}}+\int\limits_{\partial\widetilde{G}}\btr{\chi^-_{F\cup E_i}-\chi^+_{E_i}}\d \mathcal{H}^n.
		\]
		Furthermore $F_i\triangle E_i\subset\subset\Omega$, so $\mathcal{J}_{U_i,\nu_i}^{\widetilde{G}}(E_i)\le \mathcal{J}_{U_i,\nu_i}^{\widetilde{G}}(F_i)$. With the above identity, we can conclude
		\[
		\mathcal{J}_{U_i,\nu_i}^{\widetilde{G}}(E_i)\le \mathcal{J}_{U_i,\nu_i}^{\widetilde{G}}(F\cup E_i)+\int\limits_{\partial\widetilde{G}}\btr{\chi^-_{F\cup E_i}-\chi^+_{E_i}}\d \mathcal{H}^n.
		\]
		Together with inequality \eqref{ineqcaccio1}, this implies that
		\[
		\mathcal{J}_{U_i,\nu_i}^{\widetilde{G}}(F\cap E_i)-\int\limits_{\partial\widetilde{G}}\btr{\chi^-_{F\cup E_i}-\chi^+_{E_i}}\d \mathcal{H}^n
		\le \mathcal{J}_{U_i,\nu_i}^{\widetilde{G}}(F).
		\]
		Using that $E_i\cap F\to E$ in $\widetilde{G}$, $U_i\to U$ and $\nu_i\to \nu$ a.e. locally uniformly, $\btr{\nabla U_i}\to \btr{\nabla U}$ in $\mathcal{L}^{1}_{loc}$ and $\int\limits_{\partial\widetilde{G}}\btr{\chi^-_{F\cup E_i}-\chi^+_{E_i}}\d \mathcal{H}^n\to 0$ as $i\to\infty$, we can conclude that
		\[
		\mathcal{J}_{U,\nu}^{\widetilde{G}}(E)\le \mathcal{J}_{U,v}^{\widetilde{G}}(F),
		\]
		so $E$ minimizes $\mathcal{J}_{U,\nu}$ to the outside.
	\end{proof}
	We will see in Section \ref{sec_variationalformulation} (Lemma \ref{lem_smoothflowlemma}, Lemma \ref{lem_equivalence}) that the sublevelsets of smooth solutions of STIMCF satisfy the variational principle \eqref{comprintset}. Applying this to the smooth approximating solutions $U_\varepsilon$ constructed in Section \ref{sec_levelsetdescription} the Regularity Theorem \ref{thm_regularity} and the interior gradient estimate Theorem \ref{thm_interiorgradientestimate} yield the following:
	\begin{kor}\label{cor_c1alphabound}
		The downward translating graphs $\Sigma^\varepsilon_t=\{U_{\varepsilon}=t\}$ are locally uniformly bounded in $C^{1,\alpha}$ for sufficiently small $\varepsilon>0$.
	\end{kor}
	\begin{proof}
		By the Smooth Flow Lemma \ref{lem_smoothflowlemma}, we know that the sets $E^{\varepsilon}_t\definedas\{U_{\varepsilon}<t\}$ minimize $\mathcal{J}_{U_{\varepsilon},\nu_{\varepsilon}}$ on $E_b\setminus E_a$ for all $a\le t<b$, where $\nu_{\varepsilon}=\frac{\nabla U_{\varepsilon}}{\btr{\nabla U_{\varepsilon}}}$.
		
		Let $(y,z)\in (M^{n+1}\setminus E_0)\times\R$ and $d=dist((y,z),\partial E_0\times \R)=dist(y,\partial E_0)$. We now take $L$ large enough, such that $B^M_{2r}\subset\subset F_L$, $2r<d$, and $r$ and satisfies the assumptions of Theorem \ref{thm_interiorgradientestimate}. So for $\varepsilon<\varepsilon(L)$ we have an upper bound of $\sqrt{\btr{\nabla U_{\varepsilon}}^2+P_{\nu_\varepsilon}^2}$ on $B^{M\times\R}_r((y,z))$. 
		Then the Regularity Theorem \ref{thm_regularity} implies that the hypersurfaces $\Sigma_t^\varepsilon\cap B^{M\times\R}_r((y,z))$ are uniformly bounded in $C^{1,\alpha}$.
	\end{proof}
	Using the locally uniform bounds on the downward translating graphs we are able to construct limiting hypersurfaces in a jump region of $U$ in $(M\setminus E_0)\times\R$.
	\begin{prop}\label{prop_lamination}
		Let $\mathcal{K}_{t_0}$ denote the interior of a jump region $\{U=t_0\}$, at a jump time $t_0$. Then each point $X_0=(y_0,z_0)\in\mathcal{K}_{t_0}$ lies in a complete hypersurface $\widetilde{\Sigma}_{X_0}\subseteq \overline{\mathcal{K}}_{t_0}$ that is the limit of a sequence $\widetilde{\Sigma}^{\varepsilon_{i_j}}_{t_{i_j}}$ and locally uniformly bounded in $C^{1,\alpha}$.
	\end{prop}
	\begin{proof}
	We take a pointwise approach 	similar to Heidusch \cite{heidusch} and Moore \cite{moore}. We fix a target point $X_0=(y_0,z_0)\in\mathcal{K}_{t_0}$. Taking the sequence $(\varepsilon_i)\to 0$, such that the solutions to the elliptic regularisation $u_{\varepsilon_i}$ converge to $u$, we consider the corresponding sequence of times $(t_i)$ defined by $t_i\definedas U_{\varepsilon_i}(X_0)\in (-\infty.\infty)$, i.e., $X_0\in \widetilde{\Sigma}_{t_i}^{\varepsilon_i}$. Note that $t_i\to t_0$, since $U_i\to U$ locally uniformly. Let $\iota(X_0)$ denote the injectivity radius of $X_0$ in $(M\setminus E_0)\times\R$ and set 
		\[
		d=d(X_0)\definedas\min(\iota(X_0),r(X_0),dist(X_0,\partial \mathcal{K}_{t_0})),
		\]
		where $r(X_0)$ is chosen as in Corollary \ref{cor_c1alphabound}. Therefore it exists $\varepsilon'>0$, such that for all $t$ and $\varepsilon\le\varepsilon'$, the surface pieces $\widetilde{\Sigma}_{t_i}^{\varepsilon_i}\cap B^{M\times\R}_d(X_0)$ are $C^{1,\alpha}$ bounded uniformly in $t$ and $\varepsilon$. We now consider the exponential map
		\[
		\exp_{X_0}=(\exp_{y_0},\mathrm{id}_\R)\colon T_{X_0}(M\times\R)\cap B_d^{n+2}(0,z_0)\to B_d^{M\times\R}(X_0),
		\]
		and set $\widehat{\Sigma}_{t_i}^{\varepsilon_i}\definedas\exp_{X_0}^{-1}\left(\widetilde{\Sigma}_{t_i}^{\varepsilon_i}\cap B_d^{M\times\R}(X_0)\right)\subseteq T_{X_0}(M\times\R)\cong \R^{n+2}$. Then the surfaces $\widehat{\Sigma}_{t_i}^{\varepsilon_i}$ are $C^{1,\alpha}$ bounded uniformly in $t$ and $\varepsilon$. In particular, we have uniform $C^{0,\alpha}$ bounds on the unit normal $\widehat{v}_i(\hat{X_0})$ to $\widehat{\Sigma}_{t_i}^{\varepsilon_i}$ at $\hat{X_0}=(0,z_0)$. By sequence compactness, there exists limit $\widehat{v}(\hat{X_0})$, such that $\widehat{v}_{i}(\hat{X_0})\to \widehat{v}(\hat{X_0})$ uniformly (up to taking a subsequence). Then $v({\hat{X_0}})$ uniquely determines a hyperplane $\hat{T}$ centered at $\hat{X_0}$. By the uniform convergence of $\widehat{v}_i(\hat{X_0})$ to $v(\hat{X_0})$ and the uniform $C^{1,\alpha}$ bounds of the hypersurfaces $\widehat{\Sigma}_{t_i}^{\varepsilon_i}$, there exists $R\le d$, such that for $i>>1$ large enough, we can write each $\widehat{\Sigma}_{t_i}^{\varepsilon_i}$ locally as a graph over $\hat{T}\cap B_R(\hat{X_0})$. So there exists a $C^{1,\alpha}$ function $\hat{\omega}_i$ on $\hat{T}\cap B_R(\hat{X_0})$, such that $\widehat{\Sigma}_{t_i}^{\varepsilon_i}\cap B_R(\hat{X_0})=\graph(\hat\omega_i)$, and the functions $\hat{\omega}_i$ are uniformly $C^{1,\alpha}$ bounded. Using Arzel\`a--Ascoli and denoting the new H\"older exponent $0<\beta<\alpha$ again by $\alpha$ for convenience, there is a further subsequence $\hat{\omega}_{i_j}$ and a $C^{1,\alpha}$ function $\hat{\omega}\colon \hat{T}\cap B_R(\hat{X_0})\to \R$, such that
		\[
		\hat{\omega}_{i_j}\to \hat{\omega}\text{ in } C^{1,\alpha}\left(\hat{T}\cap B_R(\hat{X_0})\right),
		\]
		$\hat{\omega}$ satisfies the same $C^{1,\alpha}$ bounds, and $\hat{\omega}$ is locally the graph of a hypersurface $\widehat{\Sigma}_{X_0}$ around $\hat{X_0}$ with $\hat{T}=T_{\hat{X_0}}\widehat{\Sigma}$. Thus the hypersurface $\exp_{X_0}(\widehat{\Sigma}_{\hat{X_0}})$ in $M^{n+1}\times\R$ is uniformly $C^{1,\alpha}$ bounded. By successively taking subsequences, the hypersurfaces $\widetilde{\Sigma}_{t_i}^{\varepsilon_i}$ converge in $C^{1, \alpha}_{loc}$ to a complete hypersurface that we will henceforth denote by $\widetilde{\Sigma}_{X_0}$ satisfying $\widetilde{\Sigma}_{X_0}\cap B^{M\times\R}_R({X_0})=\exp_{X_0}(\widehat{\Sigma}_{\hat{X_0}})$.
		
		We want to conclude by showing that $\widetilde{\Sigma}_{X_0}\subseteq \overline{\mathcal{K}}_{t_0}$. Consider a point $Y\in \widetilde{\Sigma}_{\hat{X_0}}$. Then there exists a sequence $(Y_i)\to Y$ with $Y_i\in \widetilde{\Sigma}_{t_i}^{\varepsilon_i}$, and 
		\[
		\btr{U_{\varepsilon_i}(Y_i)-U(Y)}\le \btr{U_{\varepsilon_i}(Y_i)-U_{\varepsilon_i}(Y)}+\btr{U_{\varepsilon_i}(Y)-U(Y)}\to 0,
		\]
		so $U(Y)=\lim\limits_{i\to\infty}U_{\varepsilon_i}(Y_i)=\lim\limits_{i\to\infty}t_i=t_0$, so $Y\in\{U=t_0\}$.
	\end{proof}
	Using the $C^{1,\alpha}$ bounds on the limiting surfaces, we will dedicate the rest of this section on improving the above result. In particular, we will show that the hypersurfaces $\widetilde{\Sigma}_{X_0}$ indeed foliate $\mathcal{K}_{t_0}$ and are $C_{loc}^{2,\alpha}$ generalized apparent horizons that bound Caccioppoli sets that satisfy our comparison principle \eqref{comprintset}, see Theorem \ref{thm_apparenthorizons} and \ref{thm_foliation} below. To this end, we want to verify all assumptions to apply the Compactness Theorem \ref{thm_compactness1} on $\mathcal{K}_{t_0}$. First, we argue that the limiting surfaces are already sufficient to define a notion of unit normal on $\mathcal{K}_{t_0}$.
	\begin{prop}\label{prop_unitnormal}
		Let $\mathcal{K}_{t_0}$ be the interior of a jump region. Then there exists a H\"older continuous unit vector field $\nu$ on $\mathcal{K}_{t_0}$, such that $\nu_i\to\nu$ locally uniformly for a subsequence $\varepsilon_i\to0$. Moreover $\nu$ is translation invariant and everywhere normal to the hypersurfaces $\widetilde{\Sigma}_{X}$ as constructed in Proposition \ref{prop_lamination}.
	\end{prop}
	\begin{proof}
		Throughout the proof, let $\widetilde{\Sigma}_{X}$ denote a hypersurface through $X\in\mathcal{K}_{t_0}$ constructed in Proposition \ref{prop_lamination} as a sublimit of the level-sets $\{U_{\varepsilon_i}=t_i\}$ such that $U_{\varepsilon_i}(X)=t_i$. Note that the hypersurfaces $\widetilde{\Sigma}_{X}$ are not a-priori unique, as the construction relies on taking (subsequent) subsequences. However, comparing a point $X_0=(x_0,z_0)\in \mathcal{K}_{t_0}$ with a vertical translate $X_\alpha=(x_0,z_0+\alpha)$, and assuming that for a subsequence $i_j$
		\[X_0\in \widetilde{\Sigma}_{t_{i_j}}^{\varepsilon_{i_j}}=\graph\left(\frac{u_{i_j}}{\varepsilon_{i_j}}-\frac{t_{i_j}}{\varepsilon_{i_j}}\right)\to\widetilde{\Sigma}_{X_0},
		\]
		we have  $X_\alpha\in \widetilde{\Sigma}_{t_{i_j}-\alpha\varepsilon_{i_j}}^{\varepsilon_{i_j}}$, and 
		\[
		\widetilde{\Sigma}_{t_{i_j}-\alpha\varepsilon_{i_j}}^{\varepsilon_{i_j}}=\graph\left(\frac{u_{i_j}}{\varepsilon_{i_j}}-\frac{t_{i_j}-\alpha\varepsilon_{i_j}}{\varepsilon_{i_j}}\right)= \graph\left(\frac{u_{i_j}}{\varepsilon_{i_j}}-\frac{t_{i_j}}{\varepsilon_{i_j}}\right)+\alpha e_{n+2}\to\widetilde{\Sigma}_{X_0+\alpha e_{n+2}},
		\]
		so we have convergence for all vertical translates with respect to the same subsequence. Thus, it suffices to construct a unit vector field $\nu\in T\mathcal{K}_{t_0}$ along $\mathfrak{K}_{t_0}\definedas int\{u=t_0\}=\mathcal{K}_{t_0}\cap (M\times\{0\})$ that is normal to the hypersurfaces $\widetilde{\Sigma}_{X}$ for $X\in\mathfrak{K}_{t_0}$, which is then trivially extended in the $z$-direction and satisfies all the desired properties.
		
		Let $\zeta:=\{X_k\}$ be a dense, countable subset of $\mathfrak{K}_{t_0}$. Then, by taking subsequent subsequences we can choose the hypersurfaces $\widetilde{\Sigma}_{X_k}$, such that
		\[
		\widetilde{\Sigma}_{t^k_{i_j}}^{\varepsilon_{i_j}}\to\widetilde{\Sigma}_{X_k}
		\]
		locally uniformly in $C^{1,\alpha}$ with respect to the same subsequence $(\varepsilon_{i_j})$ via a diagonal sequence argument. In particular, the hypersurfaces $\widetilde{\Sigma}_{X_k}$ are locally uniform limits of the level-sets $\{U_{\varepsilon_{i_j}}=t^k_{i_j}\}$, which implies that if two surfaces touch each other, they have to do so tangentially and without intersecting anywhere.
		This yields a well-defined unit vector field on $\zeta$ dense in $\mathfrak{K}_{t_0}$. In view of the (locally) uniform $C^{1,\alpha}$ estimates established in Corollary \ref{cor_c1alphabound} and since the limiting hypersurfaces $\widetilde{\Sigma}_{X_k}$, $X_k\in\zeta$, do not intersect transversally, there is $r_0$ such that for $X_k,X_m\in\zeta$ with $\btr{X_k-X_m}<r_0$ all $\widetilde{\Sigma}^{\varepsilon_{i_j}}_{t_{i_j}}\cap B_{2r_0}(X_k,0)$ for $\varepsilon$ sufficiently small can be written as normal graphs over $\widetilde{S}_{X_k}$ with uniformly bounded $C^{1,\alpha}$ norm and 
		\[\btr{\nu_{\varepsilon_{i_j}}(X_k)-\nu_{\varepsilon_{i_l}}(X_m)}\le C\btr{X_k-X_m}^\alpha
		\] for all $j,l$ sufficiently large.
		Therefore, $\nu_{\varepsilon_{i_j}}\to\nu$ locally uniformly in $\zeta$ and $\nu$ can be extended to a Hölder continuous unit vector field on $\mathfrak{K}_{t_0}$.
		
		It remains to show that $\nu(X)$ is normal to the hypersurfaces $\widetilde{\Sigma}_X$ for all $X\in \mathfrak{K}_{t_0}\setminus \zeta$. To this end, we construct $\widetilde{\Sigma}_X$ as in Proposition \ref{prop_lamination} by taking a further subsequence $(i_{j_k})_{k\in\N}$. In particular $\nu_{\varepsilon_{i_{j_k}}}(X)\to\nu_{\widetilde{\Sigma}_X}(X)$, but we also necessarily have $\nu_{\varepsilon_{i_{j_k}}}\to\nu$, so
		\[
		\nu_{\widetilde{\Sigma}_X}(X)=\nu(X).
		\]
	\end{proof}
	Besides the existence of a measurable unit vector field $\nu$, the compactness theorems Theorem \ref{thm_compactness1} and Theorem \ref{thm_compactness2} below also require that $\btr{\nabla U_i}\to\btr{\nabla U}$ in $\mathcal{L}^1_{loc}$. The interior gradient estimate Theorem \ref{thm_interiorgradientestimate} implies that for any $L$ large enough, $\varepsilon_i<\varepsilon_0(L)$ and domain $\Omega\subseteq\Omega_L$, it holds that
	\[
	\sup_{A}\btr{\nabla{U_i}}\le C(A),
	\]
	for all $A\subset\subset\Omega$, where $C(A)$ is a positive constant only depending on $A$. Then the Compactness Theorem for BV functions, implies the weak convergence and semilowercontinuity of the gradient, i.e.,
	\[
	\nabla U_i\to \nabla U\text{ in }(C_0^0(\Omega))^*\text{, }
	\btr{\nabla U}_{\mathcal{L}^1}\le\liminf_{i\to\infty}\btr{\nabla U_i}_{\mathcal{L}^1}.
	\]
	In the interior of jump regions, the $\mathcal{L}^1_{loc}$ convergence is readily established by the weak convergence, but will demand a more delicate analysis away from jumps (see Lemma \ref{lem_l1convergence}).
	Before proving this, we want to point out that this improvement of the convergence is not required in the respective Compactness Theorems for inverse mean curvature flow \cite[Theorem 2.1]{huiskenilmanen} and for inverse null mean curvature flow \cite[Compactness Property 9]{moore}, where the proofs merely rely on the semilowercontinuity.
	\begin{lem}\label{lem_l1convergence_jumpregion}
		\[
		\btr{\nabla U_i}\to \btr{\nabla U}\text{ in }\mathcal{L}^1_{loc}\left(\mathcal{K}_{t_0}\right).
		\]
	\end{lem}
	\begin{proof}
		Let $X_0\in \mathcal{K}_{t_0}$, and for $\varepsilon>0$ small enough, we consider the geodesic ball $B_{\varepsilon}(X_0)\subseteq\subseteq \mathcal{K}_{t_0}$ such that $\nu_i\to\nu$ uniformly and $\btr{\nabla U_i}$ uniformly bounded in $B_{\varepsilon}(X_0)$. Let $\varphi\in C^0_0(B_{\varepsilon}(X_0))$, then
		\begin{align*}
			\int_{B_{\varepsilon}(X_0)}\varphi\btr{\nabla U_i}
			=\int_{B_{\varepsilon}(X_0)}\spann{\nabla U_i,\varphi\nu}-\varphi\spann{\nabla U_i,\nu_i-\nu}
			\to \int_{B_{\varepsilon}(X_0)}\spann{\nabla U,\varphi\nu}=0
		\end{align*}
		by the weak${}^*$ convergence implied by the Compactness Theorem for BV functions. Since $\btr{\nabla U_i}$ is uniformly (pointwise) bounded, the claim follows since $\varphi$ can be choosen arbitrarily close to $1$ in $\mathcal{L}^1_{loc}$ and $\btr{\nabla U}=0$ on $\mathcal{K}_{t_0}$.
	\end{proof}
	We are now in the position to use the Compactness Theorem \ref{thm_compactness1}, to show that the hypersurfaces $\widetilde\Sigma_{X_0}$ bound minimizing Caccioppoli sets, which will allow us to improve their regularity.
	\begin{thm}\label{thm_apparenthorizons}
		Each hypersurface $\widetilde{\Sigma}_{X_0}$ constructed in Proposition \ref{prop_lamination} bounds a Caccioppoli set that minimizes $\mathcal{J}_{U,\nu}$ in  $\mathcal{K}_{t_0}$ and is in fact a $C^{2,\alpha}$ generalized apparent horizon.
	\end{thm}
	\begin{proof}
		Employing the Compactness Theorem \ref{thm_compactness1}, Proposition \ref{prop_lamination} and Lemma \ref{lem_l1convergence_jumpregion} imply that $\widetilde{\Sigma}_{X_0}$ bounds a Caccioppoli set $\widetilde{E}_{X_0}$ such that $\partial\widetilde{E}_{X_0}=\widetilde{\Sigma}_{X_0}$, $\nu$ is the outward unit normal to $\partial\widetilde{E}_{X_0}$, and $\widetilde{E}_{X_0}$ minimizes $\mathcal{J}_{U,\nu}$ in $\mathcal{K}_{t_0}$. To complete the proof it remains to show, that the hypersurfaces $\widetilde{\Sigma}_{X_0}$ are $C^{2,\alpha}$ generalized apparent horizons.
		
		We recall the relationship between a function $\omega\in BV_{loc}(\Omega)$ and its subgraph
		\[
		W=\{(x,t)\in\Omega\times\R\colon t<\omega(x)\}.
		\]
		If in particular $\chi_W$ denotes the characteristic function of the subgraph $W$ of $\omega$, by \cite[Theorem 14.6]{giusti} it holds that
		\begin{align}\label{subgraph}
		\btr{\partial^*W}=\int\limits_{\Omega\times\R}\btr{\nabla \chi_W}=\int\limits_{\Omega}\sqrt{1+\btr{\nabla \omega}^2}.
		\end{align}
		As argued by the construction in Proposition \ref{prop_lamination}, we now choose a ball $B_R(X_0)$ around a point $X_0\in\widetilde{\Sigma}_{X_0}$, such that $\widetilde{\Sigma}_{X_0}\cap B_R(X_0)=\graph(\omega)$ for a function $\omega\in C^{1,\alpha}(T_{X_0}\widetilde{\Sigma}_{X_0}\cap B_R(X_0))$. Then $\widetilde{E}_{X_0}\cap B_R(X_0)=W$ is the subgraph of $\omega$, where $\Omega=B_R(X_0)$. Since $W$ minimizes $\mathcal{J}_{U,\nu}$ in $\mathcal{K}_{t_0}$, $\omega$ minimizes the functional
		\[
		\mathcal{J'}_{\nu}(\omega)\definedas\int\limits_{B_R(X_0)}\sqrt{1+\btr{\nabla \omega}^2}\d x
		-\int\limits_{B_R(X_0)}\int\limits_0^{\omega(x)}\btr{P_{\nu}}(x,s)\d s\d x,
		\]
		where we used the identity \eqref{subgraph} and the fact that $\btr{\nabla U}=0$ on $\mathcal{K}_{t_0}$. The corresponding Euler-Lagrange equation is
		\begin{align*}\label{eulersubgraph}
		\dive\left(\frac{\nabla \omega}{\sqrt{1+\btr{\nabla \omega}^2}}\right)+\btr{P_\nu}=0.
		\end{align*}
		Note that this is exactly the generalized apparent horizon, since $H=\dive(\nu)$ and by construction $\nu=\frac{(-\nabla \omega,1)}{\sqrt{1+\btr{\nabla \omega}^2}}$. Since $\omega\in C^{1,\alpha}(T_{X_0}\widetilde{\Sigma}_{X_0}\cap B_R(X_0))$, $\omega$ in particular weakly solves the uniformly elliptic equation
		\[
		a^{ij}\nabla_i\nabla_j\omega=f,
		\]
		where
		\begin{align*} a^{ij}\definedas\frac{1}{\sqrt{1+\btr{\nabla \omega}^2}^3}\left(g^{ij}-\frac{\nabla\omega^i\nabla\omega^j}{1+\btr{\nabla\omega}^2}\right),\text{ }
		f\definedas-\btr{K_{ij}\frac{\nabla\omega^i\nabla\omega^j}{1+\btr{\nabla\omega}^2}},
		\end{align*}
		are $C^{0,\alpha}$ functions on $T_{X_0}\widetilde{\Sigma}_{X_0}\cap B_R(X_0)$. Schauder Theory, \cite[Theorem 6.13]{gilbargtrudinger}, then implies that for $R(X_0)$ sufficiently small $\omega\in C^{2,\alpha}(T_{X_0}\widetilde{\Sigma}_{X_0}\cap B_R(X_0))$ is the unqiue solution which solves the equation in the strong sense. Therefore $\Sigma_{X_0}$ is an generalized apparent horizon.
	\end{proof}
	We can now use the apparent horizon equation on the hypersurfaces $\widetilde{\Sigma}_{X_0}$ to improve the result of Proposition \ref{prop_lamination}, which is the concluding statement of this Section. 
	\begin{thm}\label{thm_foliation}
		Let $\mathcal{K}_{t_0}$ denote the interior of a jump region $\{U=t_0\}$ at a jump time $t_0$. Then each point $X_0=(y_0,z_0)$ lies in a complete $C^{2,\alpha}$-hypersurface $\widetilde{\Sigma}_{X_0}$, such that the hypersurfaces $\widetilde{\Sigma}_{X_0}$ are generalized apparent horizons foliating $\mathcal{K}_{t_0}$ and possess unit normal $\nu\in C^{1,\alpha}_{loc}(\mathcal{K}_{t_0})$. If $(M,g,K)$ further satisfies the dominant energy condition \eqref{eq_dec}, then away from a set of $\mathcal{H}^{n}$ measure zero, the hypersurfaces are either vertical cylinders or translating graphs.
	\end{thm}
	\begin{bem}
		Note that, although a leave of the foliation $\widetilde{\Sigma}_{X_0}$ is always either locally a vertical cylinder or a translating graph, here we are able to fully characterize the set when  $\widetilde{\Sigma}_{X_0}$ changes character as a $\mathcal{H}^{n}$ zero measure set along the hypersurface, which is essentially given as the set, where $\btr{P}$ fails to be differentiable along the hypersurface (compare the precise definition of the set $S_{X_0}$ in the proof below). In fact, we expect that the Hausdorff dimension of the $\mathcal{H}^n$ zero measure set is at most $n-1$ and is (if non-empty) a rectifiable varifold with bounded variation.
	\end{bem}
	\begin{proof}
		By Proposition \ref{prop_unitnormal}, we already know that the hypersurfaces $\widetilde{\Sigma}_{X}$ can only touch tangentially and locally remain sheeted at the same side with respect to each other. Thus, since the hypersurfaces are in fact generalized apparent horizons, the possibility of them touching tangentially is immediately ruled out by the strong maximum principle.
		To show that the unit normal $\nu$ is indeed in $C^{1,\alpha}_{loc}(\mathcal{K}_{t_0})$, we fix a point $X_0\in \mathcal{K}_{t_0}$ and choose sufficiently small $r>0$, such that the geodesic ball $B_r(X_0)\subseteq\subseteq \mathcal{K}_{t_0}$ satisfies all the restrictions assumed in the previous propositions, and additionally such that each leave of the foliation $\widetilde{\Sigma}_{X}$ intersecting the ball can be written as the graph of a function $u^X$ over the tangent space of $T\left(\widetilde{\Sigma}_{X_0}\cap B_r(X_0)\right)$ (this follows from the $C^{1,\alpha}_{loc}$ a-priori estimates on $B_r$ independent of the hypersurface). Clearly, if $X\to X_0$, $u^X\to u^{X_0}$ uniformly on $\overline{B}_{\frac{3}{4}{r}}$ by virtue of construction. Then, the Schauder interior estimates \cite[Corollary 6.3]{gilbargtrudinger} yield $C^{2,\alpha}$ convergence on $B_{\frac{r}{2}}(X)$. In particular, $\nu\in C^{1,\alpha}_{loc}(\mathcal{K}_{t_0})$.
		
		We now assume that $(M,g,K)$ satisfies the dominant energy condition \eqref{eq_dec}. Then, ${(M\times\R, g+\d z^2,K)}$ also satisfies the dominant energy condition. Since $\nu\in C_{loc}^{1,\alpha}$ along the hypersurfaces, we have $P_\nu\in C_{loc}^{1,\alpha}$, and $\btr{P_\nu}$ extends to $C_{loc}^{1,\alpha}$, if $P\not=0$, or $P=0$ and $\nabla P=0$. For a leave of the foliation $\widetilde{\Sigma}_{X_0}$ we therefore consider the $\mathcal{H}^n$ zero measure set $S_{X_0}\definedas \overline{\{X\in \widetilde{\Sigma}_{X_0}\colon P=0,\nabla P\not=0\}}$. Thus $\btr{P_\nu}\in C^{1,\alpha}_{loc}(\widetilde{\Sigma}_{X_0}\setminus S_{X_0})$, which in particular implies that the hypersurfaces $\widetilde{\Sigma}_{X_0}$ are $C^{3,\alpha}_{loc}$ away from $S_{X_0}$ (and smooth away from $\{P_\nu=0\}$).
		Closely following the proof of \cite[Proposition 2]{schoenyau}, we can establish a Harnack inequality for the function $\gspann{\partial_z,\nu}$ along the connected components of $\widetilde{\Sigma}_{X_0}\setminus S_{X_0}$, which yields that $\widetilde{\Sigma}_{X_0}$ is either a vertical cylinder or a translating graph along these connected components. This concludes the proof.
	\end{proof}

\section{Variational formulation of weak solutions}\label{sec_variationalformulation}
	By freezing $\sqrt{\btr{\nabla u}^2+P_\nu^2}$ and treating it as a bulk term energy, we can interpret \eqref{mainPDE} as the Euler-Lagrange equation to the functional
	\begin{align}\label{comprinfunc}
	\mathcal{J}_{u,\nu}^A(v)\definedas\int\limits_{A}\btr{\nabla v}+v\sqrt{\btr{\nabla u}^2+P_\nu^2}.
	\end{align}
	If for all $A\subset\Omega$ compact
	\begin{align}\label{eq_weaksolutionsJ}
	\mathcal{J}_{u,\nu}^A(u)\le\mathcal{J}_{u,\nu}^A(v),
	\end{align}
	for all $v\in C^{0,1}_{loc}(\Omega)$ ($v\le u$, $v\ge u$ respectively), such that $\{v\not=u\}\subset\subset A\subset \Omega$, we call $u$ a \emph{weak solution (subsolution, supersolution respectively)} of \eqref{eq_weaksolutionsJ} in $\Omega$.
	\begin{bem}
		Appealing to the results of Section \ref{sec_limitingbehavior}, we will later define the concept of weak solutions of STIMCF, cf. Definition \ref{defi_weaksolutios}, on the cylinder $M\times \R$. We will thus frequently use the notion of weak solutions to \eqref{eq_weaksolutionsJ} for pairs $(U,\nu)$ on open subsets $\Omega\subseteq M\times\R$ for the functional $\mathcal{J}_{U,\nu}$ as defined by \eqref{comprinfunc}. It is easy to check that all the results and observations regarding weak solutions $u$ of \eqref{eq_weaksolutionsJ} on $M$ similarly hold for weak solutions $U$ of \eqref{eq_weaksolutionsJ} on $M\times\R$, and in fact remain true for variational principles with general, bounded bulk term energy, cf. \cite{visintin}.
		
		The relation between translation invariant weak solutions of \eqref{eq_weaksolutionsJ} on $M\times \R$ and weak solutions on $M$ is established in Lemma \ref{lem_projection} below.
	\end{bem}
	For any $v,w\in C^{0,1}_{loc}(\Omega)$ satisfying $\{v\not=w\}\subset\subset\Omega$, we find that
	\begin{align*}
	\mathcal{J}_{u,\nu}^A(\min(v,w))+\mathcal{J}_{u,\nu}^A(\max(v,w))=\mathcal{J}_{u,\nu}^A(v)+\mathcal{J}_{u,\nu}^A(w),
	\end{align*}
	so by choosing $w=u$, we can conclude that $u$ is a solution, if and only if $u$ is a subsolution and supersolution.
	Furthermore, smooth solutions to the corresponing Euler-Lagrange equation \eqref{mainPDE} are in fact minimizers of the comparison priniciple $\mathcal{J}_{u,\nu}$ as the following Lemma establishes.
	\begin{lem}[Smooth flow Lemma]\label{lem_smoothflowlemma}
		Let $u$ be a smooth solution of \eqref{mainPDE}, $E_t\definedas\{u<t\}$. Then the sets $E_t$ minimize $\mathcal{J}_{u,\nu}$ with $\nu=\frac{\nabla u}{\btr{\nabla u}}$ in $E_b\setminus E_a$ for all $a\le t<b$.
	\end{lem}
	\begin{proof}
		Follows exactly as in \cite[Lemma 2.3]{huiskenilmanen}, \cite[Lemma 15]{moore} by replacing the respective bulk term energies with $\sqrt{\btr{\nabla u}^2+P_{\nu}^2}$.
	\end{proof}
	We also see that the respective comparisons principles \eqref{comprintset} for Caccioppoli sets and \eqref{comprinfunc} for locally Lipschitz functions are indeed closely related.
	\begin{lem}\label{lem_equivalence}
		Let $u$ be a locally Lipschitz function on an open set $\Omega$. Then $u$ is a weak solution (subsolution, supersolution respectively) of \eqref{eq_weaksolutionsJ} in $\Omega$, if and only if for each $t>0$, the sets $E_t\definedas\{u<t\}$ minimize $\mathcal{J}_{u,\nu}$ in $\Omega$ (on the outside, the inside respectively).
	\end{lem}
	\begin{proof} Is proven in complete analogue to \cite[Lemma 1.1]{huiskenilmanen} and \cite[Lemma 12]{moore} replacing the respective bulk term energies by $B_{u,\nu}\definedas\sqrt{\btr{\nabla u}^2+\btr{P_\nu}^2}$.
	\end{proof}
	\begin{bem}\label{bem_equivalence}
		As it is the case for inverse mean curvature flow and inverse null mean curvature flow this equivalence also extends to the initial value problems
		\begin{align}\label{boundary1}
		\begin{split}
		u\in C^{0,1}_{loc}(M)\text{, }\nu\text{ a measurable vector field on }T(M\setminus E_0),\\
		E_0=\{u<0\}\text{, and }u\text{  is a weak solution of \eqref{eq_weaksolutionsJ} in }M\setminus E_0,
		\end{split}
		\end{align}
		and 
		\begin{align}\label{boundary2}
		\begin{split}
		u\in C^{0,1}_{loc}(M)\text{, }\nu\text{ a measurable vector field on }T(M\setminus E_0),\\
		\text{and }\forall t>0\text{ }E_t\definedas\{u<t\}\text{  minimizes \eqref{comprintset} in }M\setminus E_0.
		\end{split}
		\end{align}
		This follows directly from Lemma \ref{lem_equivalence}. Lastly, by approximating $s\searrow t$, we see that \eqref{boundary1} and \eqref{boundary2} are equivalent to
		\begin{align}\label{boundary3}
		\begin{split}
		u\in C^{0,1}_{loc}(M)\text{, }\nu\text{ a measurable vector field on }T(M\setminus E_0),\\
		\text{and }\forall t\ge 0\text{ }E_t^+\definedas\{u\le t\}\text{  minimizes \eqref{comprintset} in }M\setminus E_0,
		\end{split}
		\end{align}
		via the Compactness Theorem \ref{thm_compactness1}.
	\end{bem}
	Further, we can identify the weak mean curvature of the hypersurfaces $\widetilde{\Sigma}_t=\{u<t\}$ for weak solutions of \eqref{eq_weaksolutionsJ} outside of jump regions, where $u$ remains constant. 
	
	Recall that the first variation $\delta(\mu)$ of a $C^1$ hypersurface $\widetilde{N}$ in $M\times\R$ is defined as
	\[
	\delta(\mu)(X)\definedas\int\limits_{\widetilde{N}}\dive_{\widetilde{N}}X\d\mu
	\text{ }\forall X\in C^{\infty}_{C}(T(M\times\R)),
	\]
	where $\mu$ is the induced volume form on $\widetilde{N}$. For the following, we refer to \cite{simon} for a detailed introduction. If $\widetilde{N}$ has bounded variation, the Riesz Representation theorem implies that we can identify $\delta(\mu)$ with a vector valued measure. If the measure $\delta(\mu)$ is furthermore absolutely continuous with respect to $\mu$, the weak mean curvature vector $\vec{H}=-H\nu$ is defined via the Lebesque differentiation theorem,
	\[
	\vec{H}\definedas - D_{\mu}\delta(\mu).
	\]
	The weak mean curvature $H$ is thus characterized by the following identity
	\begin{align}\label{weakmeancur}
	\int\limits_{\widetilde{N}}\dive_{\widetilde{N}}X\d\mu
	=\int\limits_{\widetilde{N}}H\gspann{\nu,X}\d \mu
	\text{ }\forall X\in C^{\infty}_{C}(T(M\times\R)).
	\end{align}
	\begin{lem}\label{lem_weakmeancurvature}
		Let $\widetilde{\Sigma}_t=\{u<t\}$ minimize $J_{u,\nu}$ in $\widetilde{E}_b\setminus\overline{\widetilde{E}}_a$, where $u\in C^{0,1}_{loc}(\widetilde{E}_b\setminus\overline{\widetilde{E}}_a)$, and let $\Omega$ be an open set, such that $\Omega\cap \widetilde{E}_b\setminus\overline{\widetilde{E}}_a$ contains no jump regions. Then the surfaces $\widetilde{\Sigma}_t=\partial \widetilde{E}_t$ have weak mean curvature $H$ satisfying
		\[
		H=\sqrt{\btr{\nabla U}^2+P_{\nu}^2}
		\]
		a.e. in $\Omega\cap\widetilde{\Sigma}_t$ for a.e. $t\in(a,b)$.
	\end{lem}
	\begin{proof}
		Follows exactly as in \cite[Section 1]{huiskenilmanen} and \cite[Lemma 16]{moore}.
	\end{proof}
	As suggested by Moore in \cite{moore} for inverse null mean curvature flow, we will define weak solutions of STIMCF on $M$ one dimension higher in $M\times\R$, as pairs $(U,\nu)$ of translation invariant locally Lipschitz functions $U$ and measurable unit vector fields $\nu$, where $U$ minimizes $\mathcal{J}_{U,\nu}$ on $(M\setminus E_0)\times\R$. In analogue to \cite[Definition 15]{moore}, we will incorporate the results of Section \ref{sec_limitingbehavior} into the general definition of weak solutions. 
	\begin{defi}\label{defi_weaksolutios}
		Let $E_0\subset M$ be a precompact, open set with $C^2$ boundary $\Sigma_0=\partial E_0$. We call the pair $(U,\nu)$ a \emph{weak solution of inverse space-time mean curvature flow with initial condition $E_0$}, if $U\in C^{0,1}_{loc}(M\times\R)$ and $\nu$ is a measurable unit vector field which satisfy
		\begin{enumerate}
			\item[(i)] $U$ is translation invariant in the vertical direction. In particular, there exists a locally Lipschitz function $u:M\to\R$, such that $U(y,z)= u(y)$. Moreover $u$ satisfies
			\begin{align*}
			&u(x)\ge 0\text{ }\forall x\in M\setminus E_0,\\
			&u\vert_{\partial E_0}=0,\text{ }u(y)<0\text{ }\forall x\in E_0,\\
			&u(x)\to+\infty\text{ as }dist(x,E_0)\to\infty.
			\end{align*}
			\item[(ii)] The set $\widetilde{E}_t=\{U<t\}$ minimizes $\mathcal{J}_{U,\nu}$ in $(M\setminus E_0)\times\R$ for each $t>0$. At jump times $t_0$, each point $X_0=(y_0,z_0)$ in the interior $\mathcal{K}_{t_0}$ of the jump region $\{U=t_0\}$ lies in a $C^{1,\alpha}$ hypersurface $\widetilde{\Sigma}_{X_0}$, which is the boundary of a Caccioppoli set $\widetilde{E}_{X_0}$ that minimizes $\mathcal{J}_{U,\nu}$ in $\mathcal{K}_{t_0}$.
			\item[(iii)] $\nu$ is a translation invariant with
			\begin{align*}
			&\nu(X+\alpha e_z)=\nu(X)\text{ }\forall X\in(M\setminus E_0)\times\R,\text{ }\alpha\in\R;\\
			&\nu(X)\text{ is in }C^{0}_{loc}\text{ away from jump times and is the unit normal vector to }\partial\widetilde{E}_t;\\
			&\text{ at each point }X\in\partial\widetilde{E}_t\\
			&\nu(X)\text{ is in }C^{1,\alpha}_{loc}(\mathcal{K}_{t_0})\text{ and is the unit normal vector to }\partial\widetilde{E}_{X_0}\text{ at each point }X\in\partial\widetilde{E}_{X_0}\\
			&\text{ at jump times }t_0\text{ and points }X_0\in\mathcal{K}_{t_0}.
			\end{align*}
		\end{enumerate}
	\end{defi}
	\begin{bem}\label{bem_weaksolutions}\,
		\begin{enumerate}
			\item[(i)] As we require the variational principle \eqref{comprintset} for $\mathcal{J}_{U,\nu}$ to be satisfied everywhere, in particular in the interior of jump regions, we can argue as in Theorem \ref{thm_apparenthorizons} to conclude that the interior of any jump region is foliated by $C^{2,\alpha}$ generalized apparent horizons.
			\item[(ii)] By Lemma \ref{lem_equivalence}, we find that any weak solutions $(U,\nu)$ of inverse space-time mean curvature flow is a weak solution of \eqref{eq_weaksolutionsJ} on $(M\setminus E_0)\times\R$. Additionally, we formulated further restrictions that arise naturally in our construction that couple our choice of unit vector field $\nu$ to the function $U$ in an intuitively geometric way. Without these restrictions, there are in general several (translation invariant) weak solutions to \eqref{eq_weaksolutionsJ}, but where the function $U$ and the unit vector field $\nu$ are not coupled in any meaningful way. To see this, not that under the restriction $\tr_M K=0$ there exists a unit vector $\nu_p\in T_pM$ with $K_p(\nu_p,\nu_p)=0$ for all $p\in M$ and under fairly generic conditions on $K$ we can make that choice in a (uniformly) continuous way. In particular, $P_\nu=0$ for this choice of unit vector field\footnote{One could of course always choose $\nu=\partial_z$ on $M\times\R$, but we want to emphasize that this phenomenon also persists on a large class of initial data sets $(M,g,K)$ with $\nu$ tangent to $M$.} and hence $(U^{IMCF},\nu)$ is a translation invariant weak solution to \eqref{eq_weaksolutionsJ} where $U^{IMCF}$ denotes the translation invariant weak solution of inverse mean curvature flow as constructed by Huisken--Ilmanen \cite{huiskenilmanen}. In particular, we can no longer expect  any statement of uniqueness that solely relies on the comparison principle \ref{comprinfunc} and the methods of Huisken--Ilmanen developed for inverse mean curvature flow, cf. \cite[Uniquness Theorem 2.2]{huiskenilmanen}, can not be extended to the anisotropic case. Instead, a completely different approach has to be developed that also takes the the geometric restrictions of the unit vector field $\nu$ into account to conclusively answer the question of uniqueness for the concept of weak solutions to inverse space-time mean curvature flow. These observations are similarly relevant for the concept of weak solutions to inverse null mean curvature flow introduced by Moore \cite{moore}.
		\end{enumerate}
	\end{bem}
	\begin{bsp}
		We give an illustrative example for the above observation that also highlights that weak solutions to \eqref{eq_weaksolutionsJ} in general lack the geometric properties of weak solutions of STIMCF. We choose $(M,g)=(\R^3,\delta)$ and $E_0=B_1(0)$. In particular, the corresponding weak solution to inverse mean curvature flow on $M$ is the smooth expanding sphere solution $v=3\ln(\btr{x})$. We further choose $K:=\frac{6}{1+r^6}\left(\d r^2-r^2\d\theta^2\right)$ in polar coordinates, so $(M,g,K)$ is a maximal initial data set with $K\left(\frac{1}{r}\partial_\varphi,\frac{1}{r}\partial_\varphi\right)=0$. 
		
		Hence, $\left(V,\frac{1}{r}\partial_\varphi\right)$ is a weak solution to \eqref{eq_weaksolutionsJ} on $(M\setminus \overline{E_0})\times\R$ where $V(x,z)=v(x)$ and thus the solution does not exhibit any jumps. However
		\[
			H\vert_{\partial E_0}=2<3=\btr{P_{\vert\partial E_0}},
		\]
		so the weak solution constructed in Theorem \ref{thm_mainresult} has to immediately jump to a generalized apparent horizon $\partial E_0^+$, see Corollary \ref{kor_detecthorizon} below. In particular, $V\not=U$ and since $\R^3$ does not allow for any closed minimal surfaces we further have $\nu\not=\frac{1}{r}\partial_\varphi$ on $E_0^+\setminus\overline{E_0}$. Hence, we found translation invariant weak solutions $\left(V,\frac{1}{r}\partial_\varphi\right)$ and $(U,\nu)$ to \eqref{eq_weaksolutionsJ} such that neither the choice of function nor the choice of unit vector field agree.\qed
	\end{bsp}
	By projecting the translation invariant solution down to the original initial data set, we find that $(u,\nu_M)$ is a weak solution of \eqref{eq_weaksolutionsJ} on $M\setminus E_0$, where $\nu_M\definedas \nu\vert_{TM}$ is the restriction of the translation invariant vector field $\nu$ onto the tangential space of $M$. As elaborated above, we lose geometric information of the solution during the process. Nonetheless, as the level-sets of $u$ are precompact it will be more convenient from a technical perspective to formulate the formation of jumps of a weak solution $(U,\nu)$ in Section \ref{sec_jumpformation} and and the blow-down procedure in Section \ref{sec_asymptoticbehavior} below with respect to $(u,\nu_M)$.
	\begin{lem}\label{lem_projection}
		Let $(U(y,z)\definedas u(y), \nu)$ be a weak solution of inverse space-time mean curvature flow with initial condition $E_0$. Then the pair $(u,\nu_M\definedas\nu\vert_{TM})$ is a weak solution of \eqref{eq_weaksolutionsJ} on $M\setminus E_0$, and $E_t=\{u<t\}$ minimizes \eqref{comprintset} for each $t>0$.
	\end{lem}
	\begin{proof}
		Since we extended the tensor $K$ trivially in the $z$-direction, we find
		\[
		\widetilde{P}_{\nu}=(\widetilde{g}^{ij}-\nu^i\nu^j)\widetilde{K}_{ij}=({g}^{ij}-\nu_M^i\nu_M^j){K}_{ij}=P_{\nu_M}.
		\]
		In particular, we can conclude that $\sqrt{\btr{\nabla U}^2+P_\nu^2}(y,z)=\sqrt{\btr{\nabla u}^2+P_{\nu_M}^2}(y)$ for all $(y,z)$ in $(M\setminus E_0)\times\R$.
		The rest of the proof then proceeds as in the proof of Theorem 3.1 in \cite{huiskenilmanen} and \cite[Lemma 14]{moore}.
	\end{proof}
	We close this section with a Compactness Theorem for weak solutions of inverse space-time mean curvature flow.
	\begin{thm}\label{thm_compactness2}
		Let $(U_i)$ with $U_i\in C^{0,1}_{loc}(\widetilde{\Omega}_i)$ be a sequence of weak solutions of \eqref{eq_weaksolutionsJ} on open sets $\widetilde{\Omega}_i\subset M\times\R$ for a sequence of  measurable, a.e. locally uniformly continuous unit vector fields $(\nu_i)$, such that $\widetilde{\Omega}_i\to \Omega$, $U_i\to U$ in $C^{0,1}_{loc}$ and $\nu_i\to\nu$ a.e. locally uniformly for a pair $(U,\nu)$, and $\btr{\nabla U_i}\to\btr{\nabla U}$ in $\mathcal{L}^1_{loc}$.
		Then $(U,\nu)$ is a weak solution of \eqref{eq_weaksolutionsJ} on $\Omega$. 
		
		If in addition, $(U_i,\nu_i)$ is a sequence of weak solutions of inverse space-time mean curvature flow as in Definition \ref{defi_weaksolutios}, then $(U,\nu)$ is a weak solution of inverse space-time mean curvature flow.
	\end{thm}
	\begin{bem}\label{bem_compactness2}
		As in the Remark following \cite[Theorem 2.1]{huiskenilmanen} the statement of Theorem \ref{thm_compactness2} is still valid if we allow $(U_i,\nu_i)$ to be a weak solution of \eqref{eq_weaksolutionsJ} with respect to metrics $g_i$ and symmetric $(0,2)$-tensors $K_i$, such that $g_i\to g$ and $K_i\to K$ in $C^1_{loc}$.
	\end{bem}
	\begin{proof}
		By the stronger assumptions on the convergence of $\btr{\nabla U_i}$, we can replace the inductive structure of the proof of \cite[Theorem 2.1]{huiskenilmanen} by a more direct argument. Let $V$ be a locally Lipschitz function, such that $\{V\not=U\}\subset\subset \Omega$. Let $\Phi\in C^1_c(\Omega)$ with $\Phi=1$ on $\{V\not=U\}$. Then $V_i\definedas\Phi V+(1-\Phi)U_i$ is a locally Lipschitz comparison function for $U_i$, if $i>>1$ is large enough. For an appropriate open, precompact set $W\subset \Omega$, such that $\operatorname{supp}\Phi\subset W$, we have
		\[
			\mathcal{J}_{U_i,\nu_i}^W(U_i)\le \mathcal{J}_{U_i,\nu_i}^W(V_i).
		\]
		Therefore, by definition of $\mathcal{J}_{U_i,\nu_i}$, we find that
		\begin{align*}
			\int\limits_{W}\btr{\nabla U_i}+U_iB_{U_i,\nu_i}
			&\le \int\limits_W\btr{\nabla V_i}+V_iB_{U_i,\nu_i}\\
			&\le \int\limits_W\Phi\btr{\nabla V}+(1-\Phi)\btr{\nabla V}+\btr{\nabla \Phi}\btr{U_i-V_i}+(\Phi V+(1-\Phi)U_i)B_{U_i,\nu_i},
		\end{align*}
		which implies
		\begin{align*}
			\int\limits_W\Phi\left(\btr{\nabla U_i}+U_iB_{U_i,\nu_i}\right)
			\le \int\limits_W\Phi\left(\btr{\nabla V}+VB_{U_i,\nu_i}\right)+\int\limits_W\btr{\nabla \Phi}\btr{U_i-V_i}.
		\end{align*}
		By the choice of $\Phi$, we see that $\btr{\nabla \Phi}\btr{U_i-V_i}\to 0$ uniformly, and  by letting $i\to\infty$ we obtain
		\[
			\mathcal{J}_{U,\nu}^W(U)\le \mathcal{J}_{U,\nu}(V).
		\]
		If we assume the pairs $(U_i,\nu_i)$ to be weak solutions of STIMCF the translation invariance of $(U_i,\nu_i)$ implies the translation invariance of the pair $(U,\nu)$. Using Lemma \ref{lem_equivalence} and arguing as in Section \ref{sec_limitingbehavior}, we conclude that the pair $(U,\nu)$ satisfies all properties of Definition \ref{defi_weaksolutios}, so the pair is a weak solution to inverse space-time mean curvature flow.
	\end{proof}

\section{The existence of weak solutions}\label{sec_mainresults}
	We will now proof our main result. For this, we use the unit vector field $\nu$ constructed in Section \ref{sec_limitingbehavior} to extend the unit normal vector $\frac{\nabla U}{\btr{\nabla U}}$ across jump regions and employ the Compactness Theorem \ref{thm_compactness2} to the locally Lipschitz sublimit $U$.
	\begin{thm}\label{thm_mainresult}
		Let $(M^{n+1},g,K)$ be an asymptotically flat maximal initial data set as in Definition 2.1.\,. Then for any nonempty, precompact, open set $E_0\subset M^{n+1}$ with $C^2$ boundary, there exists a weak solution of inverse space-time mean curvature flow with initial condition $E_0$.
	\end{thm}
		\begin{bem}\label{bem_frauend}\,
		\begin{enumerate}
		\item[(i)]
		In initial data sets with multiple ends or inner boundary components the result holds analogously
		for initial data $E_0$ containing all but one end and all inner boundary components.
		\item[(ii)]
		Note that with the methods presented in this paper we can similarly   
		establish the existence of weak solutions of
		\[
		\frac{\d}{\d t}F=\frac{H}{H^2-P^2}
		\]
		under the same conditions imposed on $(M,g,K)$. The speed of this flow corresponds   
		precisely to the tangential projection of the codimension-$2$   
		formulation of inverse mean curvature flow first proposed by   
		Frauendiener \cite{frauendiener}. The level-set equation in this case takes the form
		\begin{equation}
		\divf_g \left(\frac{\nabla u}{\btr{\nabla u}_g}\right)
		=  \frac{1}{2}\btr{\nabla u}_g + \frac{1}{2}\sqrt{\btr{\nabla u}^2_g+ 4\left(\left(g^{ij}-\frac{\nabla^iu\nabla^ju}{\btr{\nabla u}^2}\right)K_{ij}\right)^2}
		\end{equation}
		and leads to the same singularities and  asymptotic behavior as STIMCF. In particular, their jumping behavior   
		is driven by the same outward optimization property.
		\end{enumerate}
	\end{bem}
	The proof proceeds as outlined above, where we essentially repeat the same strategy as in Section \ref{sec_limitingbehavior} and establish the corresponding results in the following lemmata.
	\begin{lem}\label{lem_positivegradient}
		For a.e. $t\ge0$, 
		\[
		\btr{\nabla u}>0 \text{ }\mathcal{H}^{n}\text{-a.e. on }\Sigma_t.
		\]
	\end{lem}
	\begin{proof}
		Follows as in \cite[Lemma 5.1]{huiskenilmanen} directly from the co-area formula.
	\end{proof}
	\begin{lem}\label{lemma_foliation}
		Let $t>0$.
		If is $t$ not a jump time, then $\widetilde\Sigma_t:=\partial\{U<t\}=\{U=t\}$ is a complete hypersurface that is locally uniformly bounded in $C^{1,\alpha}$. 
		If $t$ is a jump time, then $\widetilde\Sigma_t$, $\widetilde{\Sigma}_t^+\definedas\partial\{U>t\}$ are complete hypersurfaces that are locally uniformly bounded in $C^{1,\alpha}$.
	\end{lem}
	\begin{proof}
		Let $U_{\varepsilon}$ be the smooth solution of STIMCF on $\Omega_L\times\R$ defined via the smooth solutions of the elliptic regularisation $u_{\varepsilon}$, which exist under the above assumptions by Theorem \ref{thm_ellipticregexistence}, and let $U$ be the sublimit as in Corollary \ref{lem_sublimit}, where $L_i\to\infty$, $0<\varepsilon_i<\varepsilon_0(L_i)$ as $\varepsilon_i\to 0$. We treat the two cases separately.
		\begin{itemize}
			\item[(i)] In the case that $t$ is not a jump time, where $\widetilde{\Sigma}_t=\widetilde{\Sigma}_t^+$, the surface $\widetilde{\Sigma}_t$ is constructed by fixing a point $X_0=(y_0,z_0)\in \widetilde{\Sigma}_t$ and considering the sequence of times $t_i$, such that $X_0\in\widetilde{\Sigma}_{t_i}^{\varepsilon_i}$ for each $i$. It then follows exactly as in the proof of Proposition \ref{prop_lamination}, that $\widetilde{\Sigma}_{t_i}^{\varepsilon_i}$ converges locally uniformly to $\widetilde{\Sigma}_t$ in $C^{1,\alpha}$ for a subsequence, and $\widetilde{\Sigma}_t$ satisfies the same locally uniform $C^{1,\alpha}$ bounds, where we denote the new Hölder exponent $0<\beta<\alpha$ here and in the following again by $\alpha$ for convenience. Since in this case $\widetilde{\Sigma}_t=\{U=t\}$ and $U_i\to U$ locally  uniformly, the convergence holds for the whole sequence.
			\item[(ii)] If $t$ is a jump time, we will use a slightly different pointwise approach to construct $\widetilde{\Sigma}_{t}$ and $\widetilde{\Sigma}_{t}^+$, since $\widetilde{\Sigma}_{t}\not=\widetilde{\Sigma}_{t}^+$, and hence argue more carefully in this case. To this end, let $X_0\in\widetilde{\Sigma}^+_{t_0}$ at a jump time $t_0$. Since there are only countable many such $t_0$, there exists a sequence of points $X_i\in\widetilde{\Sigma}_{t_i}$ with $t_i>t_0$, such that $\lim\limits_{i\to\infty}X_i=X_0$, $\lim\limits_{i\to\infty}t_i=t_0$, and for $i>>1$ large enough $\widetilde{\Sigma}_{t_i}=\widetilde{\Sigma}_{t_i}^+$. As argued above, each surface piece $\widetilde{\Sigma}_{t_i}\cap B_R^{M\times\R}(X_i)$ can therefore be written via the exponential map as the graph of a $C^{1,\alpha}$ function $\hat\omega_i$ over $T_{X_i}\hat\Sigma_{t_i}$, where
			\[
				\hat\Sigma_{t_i}\definedas\exp_{X_i}^{-1}\left(\widetilde{\Sigma}_{t_i}\cap B_R^{M\times\R}(X_i)\right).
			\]
			Now consider the sequence $\nu_i$ of normal vectors to $\hat\Sigma_{t_i}$ at $\hat X_i$. By the uniform $C^{0,\alpha}$ bounds on $\hat\nu_i$, there exists a subsequence $\hat\nu_{i_k}$ and a unit vector $\hat\nu\in T_{X_0}M$, such that $\hat\nu_{i_j}\to\hat\nu$ uniformly. Let $\hat T$ denote the affine hyperplane orthogonal to $\hat\nu$ centered at $\hat X_0$. For $i>>1$ large enough, we can write each surface $\hat{\Sigma}_{t_i}$ locally as the graph of a $C^{1,\alpha}$ function $\hat\omega_i$ over $\hat T\cap B_R^{n+2}(\hat X_0)$. By Arzel\`a--Ascoli, there exists a further subsequence $\hat\omega_{i_j}$ and a $C^{1,\alpha}$ function $\hat\omega\colon\hat T\cap B_R^{n+1}(\hat X_0)\to\R$, such that
			\[
				\hat\omega_{i_j}\to\omega\text{ in }C^1(\hat T\cap B_R^{n+1}(\hat X_i)),
			\]
			where $\hat X_0\in \graph(\hat\omega)$ and $\hat T=T_{X_0}\graph(\hat\omega)$. We then consider $\omega\definedas \hat\omega\circ\exp_{X_0}^{-1}$. In order to recognize $\graph(\omega)$ as a piece of $\widetilde{\Sigma}_{t_0}^+$, we consider a point $Y\in \graph(\omega)$. By construction, there exists a sequence $Y_j\in \graph(\omega_{i_j})\subset \widetilde{\Sigma}_{t_i}$, such that $Y_j\to Y$. Hence $U(Y_i)=t_i$, which implies that $U(Y)=t_0$, so $Y\in \widetilde{E}_t^+=\{U\le t_0\}$. Assume that $Y\in int(\widetilde{E}_{t_0}^+)$, then there exists a $\delta>0$, such that $B_\delta^{M\times\R}(Y)\subset int(\widetilde{E}_{t_0}^+)$. In particular $Y_j\subset int(\widetilde{E}_{t_0}^+)$ for $j>>1$ large enough, which is a contradiction, since $U(Y_j)>t_0$. Thus $\graph (\omega)\subset\widetilde{\Sigma}_t^+$. We make an analogous argument in the case that $X_0\in \widetilde{\Sigma}_{t_0}$ for a sequence of points $X_i\in\widetilde{\Sigma}_{t_i}$, where $t_i\nearrow t_0$. Again the limit is independent of the choice of subsequence and hence the full sequence converges.
		\end{itemize}
	\end{proof}
	Since $\widetilde{\Sigma}^{\varepsilon_i}_{t_i}\to\{U=t\}$ for the whole sequence away from jump regions, we can now argue as in Proposition \ref{prop_unitnormal} that $\nu_i\to\nu$ locally uniformly away from jumps for the whole sequence, and we can thus construct a locally Hölder continuous unit normal $\nu$ away from jump regions. Since there are at most countable jump times, we consider the normal vector field $\nu$ constructed via Proposition \ref{prop_unitnormal} in the interior of each jump region and by taking successive subsequences once more, we obtained a measurable unit vector field $\nu$ on all of $(M\setminus E_0)\times \R$, such that $U_i\to U$ locally uniformly, and $\nu_i\to\nu$ a.e. locally uniformly with respect to a fixed sequence $(\varepsilon_i)_{i\in\N}$. More precisely, case (ii) in particular shows that $\nu$ is continuous on $M\setminus E_0$ away from the interior of jump regions. Moreover, we will show in the next section that the exterior boundary of any jump region $\widetilde{\Sigma}^+_t\cap\widetilde{\Sigma}^C_t$ is a generalized apparent horizon itself, so arguing with the Schauder interior estimates as in Theorem \ref{thm_foliation}, we see that $\nu$ is continuous across the exterior boundary of any jump region. Thus the continuity of $\nu$ only fails at the interior boundary $\widetilde{\Sigma}_t\cap\left(\widetilde{\Sigma}^+_t\right)^C$ of jump regions. To reconcile, we have constructed $\nu$ such that a.e.
	\begin{align}\label{eq_unitnormal}
	\nu(X)\definedas
	\begin{cases}
	\frac{\nabla U}{\btr{\nabla U}}(X)&\text{if }X\in\widetilde{\Sigma}_t\text{  at regular times }t,\\
	\widetilde{\nu}(X)&\text{if }X\in\mathcal{K}_{t_0}\text{ at a jump time }t_0,\\
	\lim\limits_{l\to\infty}\frac{\nabla U}{\btr{\nabla U}}(X_l)&\text{if }X\in\widetilde{\Sigma}_{t_0},\text{ for }X_l\to X,\text{ }t_l\nearrow t_0,\\
	\lim\limits_{l\to\infty}\frac{\nabla U}{\btr{\nabla U}}(X_l)&\text{if }X\in\widetilde{\Sigma}_{t_0}^+,\text{ for }X_l\to X,\text{ }t_l\searrow t_0.
	\end{cases}.
	\end{align}
	To employ the Compactness Theorem \ref{thm_compactness2}, it remains to show that the gradients converge in  $\mathcal{L}^1_{loc}$.
	\begin{lem}\label{lem_l1convergence}
		\[
		\btr{\nabla U_i}\to \btr{\nabla U}\text{ in }\mathcal{L}^1_{loc}.
		\]
	\end{lem}
		Note that, we will in fact prove convergence of $\nabla U_i$ in $\mathcal{L}^2_{loc}$, which will imply the proposition, since we are claiming convergence on compact regions. More precisely, by the fact that the sequence $\btr{\nabla U_i}$ is locally uniformly bounded in $\mathcal{L}^\infty$, dominated convergence yields convergence in $\mathcal{L}^p_{loc}$ for all $p\ge 1$.
	\begin{proof}
		Since there at most countable many jump regions with boundary given by the union of two $C^{1,\alpha}_{loc}$ hypersurfaces and since the claim is already proven inside jump regions, see Lemma \ref{lem_l1convergence_jumpregion}, it suffices to prove it on precompact open sets strictly away from any jump region.
		
		Let $X_0=(y_0,z_0)\in (M^{n+1}\setminus\overline{E_0})\times\R$ away from any jump region and let\linebreak ${d\definedas\min(\iota(X_0),dist(X_0,E_0\times\R), r(X_0))}$, where $\iota(X_0)$ again denotes the injectivity radius, and $r(X_0)$ is chosen such that the geodesic ball of radius $2r(X_0)$ centered at $X_0$ satisfies the assumptions of Theorem \ref{thm_interiorgradientestimate} and is compactly contained in the complement of the jump regions. Let $\phi\in C_c^{\infty}(\R)$ be a compactly supported function with $\operatorname{supp}\phi\subseteq[z_0-3d,z_0+3d]$, $\phi\equiv 1$ on $[z_0-2,z_0+2]$ and $\btr{\partial_z\phi}\le 2$, and consider $\Phi\definedas\phi^3$. Therefore
		\[
		B^{M\times\R}_d(X_0)\subseteq Z\definedas(\overline{B}^{M}_d(y_0)\times\R)\cap
		(M^{n+1}\times \operatorname{supp} \Phi).
		\]
		Since $t_{\varepsilon_i}=U_{\varepsilon_i}(X_0)\to U(X_0)=t_0$ locally uniformly, we can choose $L>>1$ large enough and $\varepsilon'<\varepsilon(L)$, such that there exist a $\delta>0$ such that
		\[Z\subseteq \bigcup\limits_{t\in[t_0-\delta,t_0+\delta]}\widetilde{\Sigma}_{t}^{\varepsilon_i}\cap (M^{n+1}\times \operatorname{supp}\Phi),\]
		with $\widetilde{\Sigma}_{t}^{\varepsilon_i}\cap(M^{n+1}\times \operatorname{supp}\Phi)$ staying strictly away from any jump region and \linebreak $\partial \widetilde{\Sigma}_{t}^{\varepsilon_i}\cap(M^{n+1}\times \operatorname{supp}\Phi)=\emptyset$ for all $t\in[t_0-\delta,t_0+\delta]$ and all $\varepsilon_i<\varepsilon'$. Furthermore, there exists $R(t_0)>r(t_0)>0$, such that for all $t\in[t_0-\delta,t_0+\delta]$ and all $\varepsilon_i<\varepsilon'$
		\begin{align*}
		\widetilde{\Sigma}_{t}^{\varepsilon_i}\cap(M^{n+1}\times \operatorname{supp}\Phi)\subseteq S(t_0)\definedas(G_{R(t_0)}(X_0)\setminus G_{r(t_0)})\times[z_0-3d,z_0+3d],
		\end{align*}
		where $G_r\definedas \{y\in M\setminus E_0\colon dist(E_0,y)< r\}$. In particular, since the sublevelsets \linebreak $\widetilde{E}_t^{\varepsilon_i}=\{U_{\varepsilon_i}<t\}$ are minimizing $\mathcal{J}_{U_{\varepsilon_i},\nu_{\varepsilon_i}}$ from the outside, we can conclude for all\linebreak ${t\in[t_0-\delta,t_0+\delta]}$ and all $\varepsilon_i<\varepsilon'$, that
		\[
		\btr{	\widetilde{\Sigma}_{t}^{\varepsilon_i}\cap(M^{n+1}\times \operatorname{supp}\Phi)}
		\le \btr{\partial^*S(t_0)'}-\int\limits_{S(t_0)'\setminus \widetilde{E}_t^{\varepsilon_i}}\sqrt{\btr{\nabla U_{\varepsilon_i}}^2+\btr{P_{\nu_{\varepsilon_i}}}^2}\le\btr{\partial^*S(t_0)'}=: C(t_0)
		\]
		for $S(t_0)'\definedas G_{R(t_0)}\times[z_0-3d,z_0+3d]$, where we used $S(t_0)'\cup \widetilde{E}_i^{\varepsilon_i}$ as a competitor.
		Additionally, since $S(t_0)$ is compact, the a-priori upper bound on $\btr{\nabla U_{\varepsilon_i}}$ implies that
		\[
		\sqrt{H_{\varepsilon_i}^2-P_{\nu_{\varepsilon_i}}^2}\le C(t_0,g,n,\norm{K},\norm{\nabla K}),
		\]
		for all $t\in[t_0-\delta,t_0+\delta]$ and all $\varepsilon_i<\varepsilon'$.
		Let $p\ge1$ be fixed. Following the strategy of the calculation of the monotonicity in \cite{huiskenilmanen} and \cite[Lemma 23]{moore}, we use that the functions $U_i$ induces a smooth graphical solution of STIMCF. We then calculate the evolution of
		\[
		\int\limits_{\widetilde{\Sigma}_{t}^{\varepsilon_i}}\Phi(z)\left(H^2-P^2\right)^p
		\]
		using the evolution equations in Lemma \ref{lem_evolutioneq}. After a straightforward but rather lengthy computation, which we will therefore omit, we see that we can make the favorable choice $p=2$ which yields
		\begin{align*}
		\frac{\d}{\d t}\int\limits_{\widetilde{\Sigma}_{t}^{\varepsilon_i}}\Phi(z)\left(H^2-P^2\right)^2
		=&
		-4\int\limits_{\widetilde{\Sigma}_{t}^{\varepsilon_i}}\Phi\btr{D\left(\sqrt{H^2-P^2}\right)}^2\left(\frac{2H}{\sqrt{H^2-P^2}}+\frac{\sqrt{H^2-P^2}}{2H}\right)\\
		&-4\int\limits_{\widetilde{\Sigma}_{t}^{\varepsilon_i}}\Phi\frac{P}{H}D PD\left(\sqrt{H^2-P^2}\right)+HD\Phi D\left(\sqrt{H^2-P^2}\right)\\
		&-4\int\limits_{\widetilde{\Sigma}_{t}^{\varepsilon_i}}\Phi\sqrt{H^2-P^2}\left(H\left(\Ric(\nu_{\varepsilon_i},\nu_{\varepsilon_i})+\btr{A}^2\right)+P\tr_{\widetilde{\Sigma}_{t}^{\varepsilon_i}}\nabla_{\nu_{\varepsilon_i}}K\right)\\
		&+8\int\limits_{\widetilde{\Sigma}_{t}^{\varepsilon_i}}\Phi(z)K_{\nu_{\varepsilon_i}j}P D\left(\sqrt{H^2-P^2}\right)^j\\
		&+\int\limits_{\widetilde{\Sigma}_{t}^{\varepsilon_i}}\left(H^2-P^2\right)^{\frac{3}{2}}\left(\frac{\partial\Phi}{\partial z}\nu_{\varepsilon_i}+\Phi H\right).
		\end{align*}
		Since $\btr{\Ric}\le C(t_0)$ on $S(t_0)$ by compactness, the upper bounds on $\sqrt{H^2-P^2}$ and on $	\btr{	\widetilde{\Sigma}_{t}^{\varepsilon_i}\cap(M^{n+1}\times \operatorname{supp}\Phi)}$ as well as a global bound on $K$ following from its asymptotics imply
		\begin{align*}
		\int\limits_{\widetilde{\Sigma}_{t}^{\varepsilon_i}}\left(H^2-P^2\right)^{\frac{3}{2}}\left(\frac{\partial\Phi}{\partial z}\nu_{\varepsilon_i}+\Phi H\right)-4\Phi\sqrt{H^2-P^2}\left(H\Ric(\nu_{\varepsilon_i},\nu_{\varepsilon_i})+P\tr_{\widetilde{\Sigma}_{t}^{\varepsilon_i}}\nabla_{\nu_{\varepsilon_i}}K\right)\le C(t_0,K, \nabla K).
		\end{align*}
		Furthermore, we can use the Peter--Paul inequality to estimate
		\begin{align*}
		-4\int\limits_{\widetilde{\Sigma}_{t}^{\varepsilon_i}}\Phi\btr{D\left(\sqrt{H^2-P^2}\right)}^2&\left(\frac{2H}{\sqrt{H^2-P^2}}+\frac{\sqrt{H^2-P^2}}{2H}\right)\le -8\int\limits_{\widetilde{\Sigma}_{t}^{\varepsilon_i}}\Phi\btr{D\left(\sqrt{H^2-P^2}\right)}^2,\\
		8\int\limits_{\widetilde{\Sigma}_{t}^{\varepsilon_i}}\Phi(z)K_{\nu_{\varepsilon_i}j}P D\left(\sqrt{H^2-P^2}\right)^j&\le \int\limits_{\widetilde{\Sigma}_{t}^{\varepsilon_i}}\Phi\btr{D\left(\sqrt{H^2-P^2}\right)}^2+16\int\limits_{\widetilde{\Sigma}_{t}^{\varepsilon_i}}\Phi P^2\btr{K_{\nu_{\varepsilon_i} j}}^2,\\
		-2\int\limits_{\widetilde{\Sigma}_{t}^{\varepsilon_i}}\Phi\frac{P}{H}D PD\left(\sqrt{H^2-P^2}\right)&\le \int\limits_{\widetilde{\Sigma}_{t}^{\varepsilon_i}}\Phi\btr{D\left(\sqrt{H^2-P^2}\right)}^2+2\int\limits_{\widetilde{\Sigma}_{t}^{\varepsilon_i}}\Phi\frac{P^2}{H^2}\btr{D P}^2,\\
		-4\int\limits_{\widetilde{\Sigma}_{t}^{\varepsilon_i}}HD\Phi D\left(\sqrt{H^2-P^2}\right)&\le \int\limits_{\widetilde{\Sigma}_{t}^{\varepsilon_i}}\Phi\btr{D\left(\sqrt{H^2-P^2}\right)}^2+4\int\limits_{\widetilde{\Sigma}_{t}^{\varepsilon_i}}\frac{\btr{D \Phi}^2}{\Phi}H^2.
		\end{align*}
		Note that $\frac{P^2}{H^2}\le 1$ since $\sqrt{H^2-P^2}>0$, and $\frac{\btr{\nabla \Phi}^2}{\Phi}\le36\phi$. It follows that
		\[
		\frac{\d}{\d t}\int\limits_{\widetilde{\Sigma}_{t}^{\varepsilon_i}}\Phi(z)\left(H^2-P^2\right)^2
		\le -5\int\limits_{\widetilde{\Sigma}_{t}^{\varepsilon_i}}\Phi\btr{D\left(\sqrt{H^2-P^2}\right)}^2+C(t_0,K,\nabla K),
		\]
		and integrating yields
		\begin{align}\label{rellich1}
		\int\limits_{t_0-\delta}^{t_0+\delta}\int\limits_{\widetilde{\Sigma}_{t}^{\varepsilon_i}}\Phi\btr{D\left(\sqrt{H^2-P^2}\right)}^2\le C(t_0,K,\nabla K).
		\end{align}
		Applying Fatou's Lemma, there exists a subsequence $i_j$ (henceforth just denoted as $i$ for convenience), such that
		\begin{align}\label{rellich2}
		\int\limits_{\widetilde{\Sigma}_{t}^{\varepsilon_i}}\Phi\btr{D\left(\sqrt{H^2-P^2}\right)}^2<\infty
		\end{align}
		for almost every $t\in[t_0-\delta,t_0+\delta]$, but since $U_i$ are translating solutions, we can arrange this to be the case for any sequence of times, in particular \eqref{rellich2} indeed holds for any $t\in[t_0-\delta,t_0+\delta]$. Hence, $\btr{\nabla U_i}=\sqrt{H^2-P^2}$ is uniformly bounded in $W^{1,2}_{loc}(\widetilde{\Sigma}_t^{\varepsilon_i})$ for any $t$, so by Rellich--Kondrachov (writing everything locally as a graph over a fixed tangent plane), there exists a further subsequence $i_j$ and a function $f_t\in\mathcal{L}^{\infty}(\widetilde{\Sigma}_t)$ such that $\btr{\nabla U_{i_j}}\to f_t$ in $\mathcal{L}^2_{loc}$ (up to composition with the local graphs over the fixed tangent plane), which also implies a similarly defined pointwise a.e. convergence.
		
		However, by the locally uniform $C^{1,\alpha}$ convergence and the locally uniform gradient bound Theorem \ref{thm_interiorgradientestimate}, we know that the surfaces $\widetilde\Sigma_t=\{U=t\}$ have a weak mean curvature vector $\vec{H_t}=-H_t\nu$ for a function $H_t\in\mathcal{L}^\infty_{loc}$, such that $\vec{H^{\varepsilon_i}_{t_i}}\to \vec{H_t}$ weakly in $\mathcal{L}^2$. Since $\nu_i\to\nu$ locally uniformly, this implies that
		\[
			\sqrt{\btr{\nabla U_i}^2+K(\nu_i,\nu_i)^2}\to H_t\text{ weakly in }\mathcal{L}^2,
		\]
		where we used that $U_i$ solves equation \eqref{mainPDE} as a strong solution. Since $\btr{\nabla U_{i_j}}$ moreover converges strongly against $f_t$, we see that we in fact have strong convergence for $\sqrt{\btr{\nabla U_{i_j}}^2+K(\nu_{i_j},\nu_{i_j})^2}$ and by the uniqueness of weak limits we find $f_t=\sqrt{H_t^2-K(\nu,\nu)^2}$. In particular, $f_t$ is independent of the choice of subsequence, so in fact $\btr{\nabla U_i}$ and $\vec{H^{\varepsilon_i}_{t_i}}$ converge strongly in $\mathcal{L}^2_{loc}$ against $f_t$ and $\vec{H}_t$ respectively for the full sequence.
		
		We now locally define $f(X):=f_t(X)$ if $X\in \widetilde{\Sigma}_t$ and notice that
		\[
			f(X)=\lim\limits_{i\to\infty}\btr{\nabla U_i}\circ \Pi_i(X)
		\]
		pointwise a.e., where $\Pi_i$ denotes the projection of a point $X\in\widetilde{\Sigma_t}$ to the corresponding point $X_i\in\widetilde{\Sigma}_{t_i}^{\varepsilon_i}$ such that $\widetilde{\Sigma}_{t_i}^{\varepsilon_i}\to\widetilde{\Sigma}_t$ locally in $C^{1,\alpha}$.  Since $U_i\to U$ locally uniformly, $\Pi_i$ is continuous for $i$ sufficiently large and thus $f$ is a measurable, locally a.e. bounded function.
		
		Using the weak * convergence of $\nabla U_i$ similar as in Lemma \ref{lem_l1convergence_jumpregion}, the locally uniform convergence $\nu_i\to\nu$ and the coarea formula, we find
		\begin{align}
			\int_\Omega\btr{\nabla U_i}&\to\int_{\Omega}\btr{\nabla U},\\
			\int_{\Omega}\btr{\nabla U_i}\btr{\nabla U}&\to\int_\Omega\btr{\nabla U}^2,\\
			\int_{\Omega}\btr{\nabla U_i}^2&\to \int_\Omega f\btr{\nabla U}
		\end{align}
		for any open set $\Omega$ compactly contained in the complement of jump regions and these convergences remain true if we integrate with respect to an arbitrary smooth, compactly supported test function. Using first the Hölder inequality with $\btr{\nabla U_i}$ and $\btr{\nabla U}$, and then taking a limit gives
		\begin{align}\label{eq_integral1}
			\int_\Omega\btr{\nabla U}^2\le \int_\Omega f\btr{\nabla U}.
		\end{align}
		Assume that the inequality in \eqref{eq_integral1} is strict for some open set $\Omega$ as above. Then by the above convergences there exists $\delta>0$, such that 
		\[
			0<\delta\le \int_\Omega\btr{\nabla U_i}^2-\int_\Omega\btr{\nabla U_i}\btr{\nabla U}=\int_{\Omega}\btr{\nabla U_i}\left(\btr{\nabla U_i}-\btr{\nabla U}\right)
		\]
		for all $i$ sufficiently large. Let $\Omega_+:=\{X\in\Omega\vert \left(\btr{\nabla U_i}-\btr{\nabla U}\right)> 0\}$, then 
		\[
			0<\delta \le\int_{\Omega_+}\btr{\nabla U_i}\left(\btr{\nabla U_i}-\btr{\nabla U}\right)\le C\int_{\Omega_+}\btr{\nabla U_i}-\btr{\nabla U},
		\]
		where $C$ denotes the uniform bound on $\btr{\nabla U_i}$. Hence, there exists $\delta'>0$, such that
		\[
			\delta'\le \int_{\Omega_+}\btr{\nabla U_i}-\btr{\nabla U}.
		\]
		Consider an open set $\Omega'$ with $\Omega\subseteq\subseteq\Omega'$ and $\Omega'$ also compactly contained in the complement of any jump region. Let $\eta_k\in C^\infty_0(\Omega')$ be a sequence of smooth mollifiers converging in $\mathcal{L}^2$ to the characteristic function $\chi_{\Omega_+}$ of $\Omega_+$ in $\Omega'$. Using the Hölder inequality, we see that
		\[
			\btr{\int_{\Omega'}(\btr{\nabla U_i}-\btr{\nabla U})(\eta_k-\chi_{\Omega_+})}
			\le \norm{(\btr{\nabla U_i}-\btr{\nabla U})}_{\mathcal{L}^2}\cdot \norm{\eta_k-\chi_{\Omega_+}}_{\mathcal{L}^2}\to0,
		\]
		since $\btr{\nabla U_i}$, $\btr{\nabla U}$ are uniformly bounded. Hence, there exists $k_0\in\N$, such that 
		\[
			0<\frac{\delta}{2}\le \int_{\Omega'}\eta_k(\btr{\nabla U_i}-\btr{\nabla U})
		\]
		for all $k\ge k_0$, a contradiction to the the weak * convergence of the gradients. Thus
		\[
			\int_\Omega\btr{\nabla U}^2=\int_{\Omega}f\btr{\nabla U}
		\]
		for all $\Omega$ as above and in particular 
		\[
			\int_\Omega\btr{\nabla U_i}^2\to\int_\Omega\btr{\nabla U}^2.
		\]
		Invoking the weak $\mathcal{L}^2$ convergence one last time, we conclude
		\begin{align*}
			\int_{\Omega}\btr{(\btr{\nabla U_i}-\btr{\nabla U})}^2
			&=\int_{\Omega}\btr{\nabla U_i}^2-2\int_\Omega \spann{\nabla U_i,\nabla U}+\int_\Omega \btr{\nabla U}^2\to0,
		\end{align*}
		concluding the proof.
		
	\end{proof}
	Theorem \ref{thm_mainresult} now follows by using the Compactness Theorem \ref{thm_compactness2} for weak solutions.

\section{The outward optimizing property and the selection of jump times}\label{sec_jumpformation}
	We already know that the interior of jump regions of a weak solution $(U,\nu)$ is foliated by $C^{2,\alpha}$ generalized apparent horizons in $(M\setminus E_0)\times\R$, but did not discuss when such jumps occur. In this section, we will study the selection of jump times of weak solutions of STIMCF via an outward optimization property.
	
	Let $\Omega$ be an open set in $M^{n+1}$, $\nu$ a measurable vector field on $M$. We call the set $E$ outward optimizing in $\Omega$ with respect to $\nu$, if $E$ minimizes area minus bulk term energy $\btr{P_{\nu}}$ on the outside in $\Omega$. That is, if
	\begin{align}\label{outopt}
		\btr{\partial^* E\cap A}\le \btr{\partial^*F\cap A}-\int\limits_{F\setminus E}\btr{P_{\nu}},
	\end{align}
	for any set $F$ containing $E$, such that $F\setminus E\subset A\subset\subset \Omega$, where $P_\nu=\left(g^{ij}-\nu^i\nu^j\right)K_{ij}$ as usual. We further call the set $E$ strictly outward optimizing in $\Omega$, if equality in \eqref{outopt} implies that $F\cap \Omega=E\cap\Omega$ up to a set of measure zero.
	Further, we define the strictly outward optimizing hull (in $\Omega$) $E'=E'_{\Omega}$ of a measurable set $E\subset \Omega$ to be the intersection of Lebesque points of all strictly outward optimizing sets in $\Omega$ that contain $E$. Up to a set of measure zero, $E'$ may be realized as a countable intersection, so $E'$ is in particular strictly outward optimizing and open. Due to the asymptotic decay of $g$ and $K$, the existence of strictly outward optimizing sets follows from the isoperimetric inequality if we allow $\Omega$ to be sufficiently large. In particular, any set admits a precompact outward optimizing hull if the domain $\Omega$ is sufficiently large.
	
	Note that, unlike the corresponding outward optimization property of inverse null mean curvature flow introduced by Moore \cite[Section 6]{moore}, the bulk term energy in \eqref{outopt} is non-positive everywhere. This suggests that in non time-symmetric initial data sets, the additional energy induced by the second fundamental form $K$ makes the evolving surfaces jump sooner than it is the case for inverse mean curvature flow. In fact, it is immediate, that (strictly) outward optimizing surfaces are (strictly) outward minimizing as defined by Huisken--Ilmanen \cite{huiskenilmanen}. This was to be expected as inverse mean curvature flow acts as an upper barrier, cf. Corollary \ref{lem_sublimit}. However, as we can not  expect STIMCF to be a simple reparameterization of inverse mean curvature flow, this underlines the role of the unit normal and the anisotropy of the problem, as the level sets are not only outward minimizing with respect to area, but outward optimizing with respect to the unit normal and second fundamental form $K$. Analogue to the respective optimization properties of inverse mean curvature flow and inverse null mean curvature flow, we establish the following:
	\begin{thm}[Outward optimization property]\label{thm_outwardoptimziation}
		Suppose that $(U,\nu)$ is a weak solution of inverse space-time mean curvature flow with initial condition $E_0$, such that the level-sets $E_t=\{u<t\}$ of the projection $(u,\nu_M)$ onto $M$ are precompact for all $t>0$. Suppose further that $M$ has no compact components.\newline
		Then $E_t\definedas\{u<t\}$ is outward optimizing in $M$ with respect to $\nu_M$ for $t>0$, and\linebreak $E_t^+\definedas\{u\le t\}$ is outward optimizing in $M$ with respect to $\nu_M$ for $t\ge 0$. Furthermore, we have that
		\begin{itemize}
			\item[(i)] $E_t^+$ is strictly outward optimizing in $\Omega$ with respect to $\nu_M$ for all $t\ge 0$, where $\Omega$ is an open set containing $E_t^+$ such that $\Omega$ does not contain any jump regions $\mathcal{K}_{t'}$ for $t'>t$.
			\item[(ii)] $(E_t)'_\Omega=E_t^+$ ,
			\item[(iii)] 
			\[
			\btr{\partial E_t^+}=\btr{\partial E_t}+\int\limits_{E_t^+\setminus E_t}\btr{P_{\nu_M}},
			\]
			for all $t> 0$. This precisely extends to $E_0$, if $E_0$ is outward optimizing.
		\end{itemize}
	\end{thm}
	\begin{bem}\label{bem_outwardoptimziation}\,
		\begin{itemize}
			\item[(i)] Note that (iii) in Theorem \ref{thm_outwardoptimziation} implies that $\btr{\Sigma_t}+\int\limits_{\{u\le t\}\setminus E_0}\btr{P_{\nu_M}}$ is continuous on $(0,\infty)$, and continuous in $0$ precisely when $\Sigma_0$ is outward optimizing. Moreover, the quantity is monotone under the smooth flow with
			\[
				\frac{\d}{\d t}\left(\btr{\Sigma_t}+\int\limits_{\{u\le t\}\setminus E_0}\btr{P_{\nu_M}}\right)=\int\limits_{\Sigma_t}\sqrt{\frac{H+\btr{P}}{H-\btr{P}}}\ge \btr{\Sigma_t}.
			\]
			\item[(ii)] By the continuity of the solution, there exists an open set $\Omega$ with the desired properties $\forall t\ge 0$. In particular, if $t'$ denotes the next time the solution will jump, we can choose $\Omega\definedas\{u<t''\}$ for any $t<t''<t'$. If we drop this restriction, then the outward optimizing hull $E_t'$ will agree with $E_t^+$ up to a disjoint union of open, ''cost free'' surfaces that are confined to other jump regions. We will make this statement precise in the following Proposition \ref{lem_costfree}.
			\item[(iii)] Similarly, the sets $\{U<t\}$ and $\{U\le t\}$ are outward optimizing in $(M\setminus E_0)\times\R$ with respect to $\nu$. Moreover, if there exists a family of smooth solutions $(U_i,\nu_i)$ such that $U_i\to U$ locally uniformly and $\nu_i\to\nu$ locally uniformly in jump times of $U$ (which in particular is satisfied by the weak solutions constructed in Theorem \ref{thm_mainresult}) then $\{U\le t\}$ is strictly outward optimizing in $(M\setminus E_0)\times\R$.
		\end{itemize}
	\end{bem}
	\begin{proof}
		The fact that $E_t$ is outward optimizing immediately follows from the results of Lemma \ref{lem_equivalence} and Lemma \ref{lem_projection}, since $E_t$ in particular minimizes $\mathcal{J}_{u,\nu}$ from the outside for all $t>0$, so for all sets $F$ such that $E_t\subset  F$ and $F\setminus E\subset A\subset\subset M$, it holds that
		\begin{align*}
			\btr{\partial^* E_t\cap A}
			\le \btr{\partial^* F\cap A}-\int\limits_{F\setminus E_t}\sqrt{\btr{\nabla u}^2+P_{\nu_M}^2}\le \btr{\partial^* F\cap A}-\int\limits_{F\setminus E_t}\btr{P_{\nu_M}}.
		\end{align*}
		By Remark \ref{bem_equivalence}, we also have for all $t\ge 0$ that
		\begin{align}\label{hullcontra}
		\btr{\partial^* E_t^+\cap A}-\int\limits_{E_t^+\cap A}\sqrt{\btr{\nabla u}^2+P_{\nu_M}^2}
		\le \btr{\partial^* F\cap A}-\int\limits_{F\cap A}\sqrt{\btr{\nabla u}^2+P_{\nu_M}^2},
		\end{align}
		for $F$ such that $E_t^+\triangle F\subset A\subset\subset M\setminus E_0$. In particular, for any $t\ge 0$, we find that $E_t^+$ is outward optimizing in $M$, i.e.
		\begin{align}\label{hilfe1}
		\btr{\partial^* E_t^+\cap A}\le \btr{\partial^* F\cap A}-\int\limits_{F\setminus E_t^+}\btr{P_{\nu_M}},
		\end{align}
		for any $E_t^+\subset F$, $F\setminus E_t^+\subset A\subset\subset M$. We now proof $(i)-(iii)$.
		\begin{itemize}
			\item[(i)] Let $\Omega$ be as above, suppose there exists $F\subset A\subset\subset \Omega$ containing $E_t^+$, such that 
				\[
					\btr{\partial^* E_t^+\cap A}= \btr{\partial^* F\cap A}-\int\limits_{F\setminus E_t^+}\btr{P_{\nu_M}}
				\]
				Assume that $F\not=E_t^+$ and is not a set of measure zero. With regards to \eqref{hullcontra}, this implies that $\btr{\nabla u}=0$ a.e. on $F\setminus E_t^+$ and $F$ is also outward optimizing. Since the Lebesque points of an outward optimizing set form an open outward optimizing set, we can assume by a modification of measure zero, that $F$ is open. Then $u$ is constant on every connected component of $F\setminus E_t^+$. A contradiction, as $\Omega$ contains no jump regions outside of $E_t^+$. Hence, $E_t^+$ is outward optimizing in $\Omega$.
			\item[(ii)] Since $E_t^+$ is strictly outward optimizing in $\Omega$, we have $(E_t)'_\Omega\subseteq E_t^+$ by definition. If
			\[
				\btr{\partial^*E_t'\cap A}=\btr{\partial^*E_t^+\cap A}-\int\limits_{E_t^+\setminus E_t'}\btr{P_{\nu_M}},
			\]
			the strict optimization property of $E_t'$ implies that $E_t'=E_t^+$. Otherwise 
			\[
			\btr{\partial^*E_t'\cap A}<\btr{\partial^*E_t^+\cap A}-\int\limits_{E_t^+\setminus E_t'}\btr{P_{\nu_M}},
			\]
			but this contradicts \eqref{hullcontra}, since $\btr{\nabla u}=0$ on $E_t^+\setminus E_t'$.
			\item[(iii)] Since $E_t$ is outward optimizing, we can use $E_t^+$ as a competitor to obtain
			\[
				\btr{\partial^*E_t\cap A}\le\btr{\partial^*E_t^+\cap A}-\int\limits_{E_t^+\setminus E_t}\btr{P_{\nu_M}},
			\]
			for $t>0$, and for $t=0$ if $E_0$ happens to be outward optimizing itself.
			Again, since $\btr{\nabla u}=0$ on $E_t^+\setminus E_t$, strict inequality would contradict \eqref{hullcontra}.
		\end{itemize}
	\end{proof}
	\begin{prop}\label{lem_costfree}
		Let $(U,\nu)$ be a weak solution of STIMCF, such that $(U,\nu)$ and $M$ satisfy all assumptions of Theorem \ref{thm_outwardoptimziation}. Let $\Omega$ be a domain in $M\setminus \overline{E_0}$ such that $E_t^+=\{u<t\}\subset \Omega$, and assume that $E_t$ admits a precompact outward optimizing hull $(E_t)'_\Omega$ in $\Omega$ with respect to $\nu_M$. Then $E_t^+=\{u\le t\}\subset (E_t)'_\Omega$ up to a measure zero, and up to a set of measure zero $(E_t)'_\Omega\setminus E_t^+$ may be realized as the disjoint union of open sets $F$ contained in jump regions $\{u=t_0\}$ for $t_0>t$, where $F$ satisfies
		\begin{align}\label{costfree}
			\btr{\partial^*F}=\int_F\btr{P_{\nu_M}}.
		\end{align}
		If any of these sets $F$ satisfy $\overline{F}\subset int\{u=t_0\}$, then $\nu_{\partial^*F}=\nu_M=\nu$ a.e. along $\partial^*F$, so the the unit normal $\nu$ in jump regions of the flow aligns with the unit normal to $\partial^*F$.
	\end{prop}
	\begin{bem}\label{bem_costfree}
		We may allow $\Omega=M\setminus \overline{E_0}$ and see that $E_t^+$ is precisely the union of the connected components of $(E_t)'_\Omega$ that each contain a connected component of $E_t$. Intuitively, this suggests that the flow does not want to increase the number of connected components of the level sets, even if the outward optimization principle would prefer to jump farther. See also the corresponding connectedness lemma, Lemma \ref{lem_connectedness}, for weak solutions of STIMCF below.
	\end{bem}
	\begin{proof}
		We define $P_{\nu_M}$ and $B_{u,\nu_M}=\sqrt{\btr{\nabla u}^2+P_{\nu_M}}$ as before and we will write $E_t'=(E_t)'_\Omega$ for convenience. By assumption $E_t'$ is precompact. To establish the above claim, we will largely exploit the general formula \eqref{eq_reducedboundary}, which allows us state \eqref{hilfe1} more generally, as we have
		\begin{align}\label{eq_costfree1}
			\btr{\partial^*(E_1\cup E_2)}-\int\limits_{E_1\cup E_2}f+\btr{\partial^*(E_1\cap E_2)}-\int\limits_{E_1\cap E_2}f\le \btr{\partial^*E_1}-\int\limits_{E_1}f+\btr{\partial^*E_2}-\int\limits_{E_2}f,
		\end{align}
		for any (locally) integrable function $f$ and precompact Caccioppoli sets $E_1$, $E_2$.
		As $E_t^+$ minimizes $\mathcal{J}_{u,\nu_M}$, we have
		\[
			\btr{\partial^*E_t^+}-\int\limits_{E_t^+}B_{u,\nu_M}\le \btr{\partial^*(E_t^+\cap E_t')}-\int\limits_{E_t^+\cap E_t'}B_{u,\nu_M},
		\]
		and since $\btr{\nabla u}=0$ on $E_t^+\setminus E_t$, we can conclude that in fact
		\begin{align}\label{eq_costfree2}
			\btr{\partial^*E_t^+}-\int\limits_{E_t^+}\btr{P_{\nu_M}}\le \btr{\partial^*(E_t^+\cap E_t')}-\int\limits_{E_t^+\cap E_t'}\btr{P_{\nu_M}}.
		\end{align}
		Choosing $f=\btr{P_{\nu_M}}$, $E_1=E_t^+$, $E_2=E_t'$ in \eqref{eq_costfree1}, we can conclude using \eqref{eq_costfree2} that
		\[
			\btr{\partial^*E_t^+}-\int\limits_{(E_t^+\cup E_t')\setminus E_t'}\btr{P_{\nu_M}}\le \btr{\partial^*E_t'}.
		\]
		Hence $E_t'=E_t^+\cup E_t'$ up to a set of measure zero, as $E_t'$ is strictly outward optimizing. Therefore $E_t^+\subset E_t'$ up to a set of measure zero. Using that $E_t^+$ minimizes $\mathcal{J}_{u,\nu_M}$ from the outside and is outward optimizing in $\Omega$, we find
		\begin{align}
			\begin{split}\label{eq_costfree3}
				\btr{\partial^*E_t^+}&\le \btr{\partial^*E_t'}-\int\limits_{E_t'\setminus E_t^+}B_{\nu_M},\\
				\btr{\partial^*E_t^+}&\le \btr{\partial^*E_t'}-\int\limits_{E_t'\setminus E_t^+}\btr{P_{\nu_M}}.
			\end{split}
		\end{align}
		For $E_1=E_t^+$ and $E_2=E_t'\setminus E_t^+$, and choosing both $f=B_{u,\nu_M}$ and $f=\btr{P_{\nu_M}}$, \eqref{eq_costfree1} allows us to conclude that
		\begin{align}\label{eq_costfree4}
			\btr{\partial^+(E_t'\setminus E_t^+)}=\int\limits_{E_t'\setminus E_t^+}B_{u,\nu_M}=\int\limits_{E_t'\setminus E_t^+}\btr{P_{\nu_M}}.
		\end{align}
		If $int(E_t'\setminus E_t^+)\not=\emptyset$, $\btr{\nabla u}=0$ a.e. on $E_t'\setminus E_t^+$, and therefore each connected component $F$ of $E_t'\setminus E_t^+$ is contained in a jump region $\{u=t_0\}$ for $t_0>t$ and satisfies \eqref{eq_costfree4} by itself.
		
		It remains to show that $\nu_{\partial^*F}=\nu_M=\nu$ a.e. along $\partial^*F$ for any such surface $F$ with $\overline{F}\subset int\{u=t_0\}$. By Theorem \ref{thm_foliation}, the interior $\mathcal{K}_{t_0}$ of the jump region $\{U=t_0\}$ is foliated by generalized apparent horizons with $\nu\in C^{1,\alpha}_{loc}(\mathcal{K}_{t_0})$. By the translation invariance, we find
		\[
			\btr{P_{\nu_M}}=\btr{P_\nu}=\dive_{M\times\R}\nu=\dive_M\nu_{M},
		\]
		and the claim follows by the divergence theorem.
	\end{proof}
	\begin{prop}
		For $t\ge 0$, we find that
		\begin{enumerate}
			\item[(i)] $H=\btr{P_{\nu_M}}$ on $\partial E_t^+\cap {\partial E}_t^C$,
			\item[(ii)] $H\ge \btr{P_{\nu_M}}$ in the weak sense on $\partial E_t^+\cap{\partial E}_t$.
		\end{enumerate}
	\end{prop}
	\begin{proof}\,
		\begin{enumerate}
			\item[(i)] Since $E_t^+$ is outward optimizing, we know that $E_t^+$ minimizes the functional 
			\begin{align}\label{horizonfunc}
			\btr{\partial^*E_t^*\cap A}-\int\limits_{E_t^+\cap A}\btr{P_{\nu_M}}
			\end{align}
			against any competitor $F$, such that $E_t^+\subseteq F$, $F\setminus E_t^+\subset A\subset\subset \Omega$. Now let $y\in\partial E_t^+$ such that $y\not\in\overline{E}_t$. Then locally around $y$ $\partial E_t$ and $\partial E_t^+$ are separated by a jump region, and since $\partial E_t^+$ is $C^{1,\alpha}$, there exists an $R>0$, such that $B_R(y)\cap \partial E_t^+$ is given by the graph over a $C^{1,\alpha}$ function $\omega$ and $W\definedas B_R(y)\cap int (E_t^+)\subset E_t^+\setminus \overline{E}_t$ is the subgraph of $\omega$, and $\btr{\nabla u}=0$ on $W$. Using that $E_t^+$ minimizes \eqref{comprintset}, in particular from the inside, we conclude that $E_t^+$ minimizes \eqref{horizonfunc} from the inside and outside for $K=\overline{B}_R(y)$. We conclude that $W$ minimizes \eqref{horizonfunc} and therefore, as in Theorem \ref{thm_apparenthorizons}, we see that $\omega$ minimizes the functional
			\[
			\mathcal{J}'_{\nu_M}(\omega)\definedas\int\limits_{A}\sqrt{1+\btr{\nabla \omega}^2}-\int\limits_A\int\limits_0^{\omega(x)}\btr{P_{\nu_M}}\d s\d x.
			\]
			In particular $\omega\in C^{2,\alpha}_{loc}$, and $\partial E_t^+\cap\overline{E_t}^C$ satisfies $H=\btr{P_{\nu_M}}$.
			\item[(ii)] If $\partial E_t$ and $\partial E_t^+$ only agree on a set of measure zero, there is nothing to show. So we assume there exists an $y\in \partial E_t$, such that $E_t\cap B_R(y)\subset E_t^+$ for some $R>0$ and w.l.o.g. that $\partial E_t\cap B_R(y)$ is given as the graph over a $C^{1,\alpha}$ function $\omega$. Then similar to before, we can conclude that $\omega$ is a supersolution for $\mathcal{J}'_{\nu_M}$, i.e.,
			\[
			0\le \frac{\d}{\d \varepsilon}\mathcal{J}'(\omega+\varepsilon\eta)=\int\limits_K\frac{\nabla\eta\cdot\nabla \omega}{\sqrt{1+\btr{\nabla \omega}^2}}-\eta\btr{P_\nu},
			\]
			for all $\eta\in C^{\infty}_c(\partial E_t\cap B_R(y))$, such that $\eta\ge 0$. As $\partial E_t$ admits a weak mean curvature $H$, we can conclude that
			\[
			0\le \int\limits_K\eta\left(H-\btr{P_\nu}\right)
			\]
			for all $\eta\in C^{\infty}_c(\partial E_t\cap B_R(y))$, such that $\eta\ge 0$. This establishes the claim.
		\end{enumerate}
	\end{proof}
	Choosing appropriate initial data $E_0$, the properties of jump formation allow us to formulate conditions that will force any weak solution to jump. Therefore the existence of weak solutions Theorem \ref{thm_mainresult} yields a method of detecting generalized apparent horizons.
	\begin{kor}\label{kor_detecthorizon}\,
		\begin{enumerate}
			\item[(i)] If $E_0$ is not outward minimizing, then $E_0^+\not=E_0$,
			\item[(ii)] If $E_0$ satisfies $H_{\partial E_0}<\btr{\tr_{\partial E_0}K}$, then $\partial E_0^+$ is a $C^{2,\alpha}$ generalized apparent horizon disjoint from $\partial E_0$.
		\end{enumerate}
	\end{kor}
	\begin{bem}
		The conclusion in (ii) is generally stronger than in (i), as we can not guarantee that $\partial E_0^+\cap E_0=\emptyset$ in the first case. However, the assumption in (i) is more flexible as it does not depend on the sign of the mean curvature and is completely unrelated to $K$. Note that (i) also follows from Corollary \ref{lem_sublimit}, as $E_0^{+,(IMCF)}\subseteq E_0^+$, so the solution has to jump immediately, if inverse mean curvature flow jumps at $t=0$.
		
		Regarding (ii), note that hypersurfaces satisfying $H<\btr{\tr_{\partial E_0}K}$ are often referred to as \emph{trapped surfaces} in the context of General Relativity, and weakly trapped if the inquality is not assumed to be strict. Assuming the existence of a weakly trapped surface, Eichmair \cite{eichmaier} can similarly provide the existence of a generalized apparent horizon using the Perron method. Moreover, Eichmair can in fact construct an outermost generalized apparent horizon. Our result suggests that this outermost generalized apparent horizon is outward minimizing with respect to area (at least if all conditions of Theorem \ref{thm_mainresult} are statisfied).
	\end{bem}
		
	\section{Asymptotic behavior of weak solutions}\label{sec_asymptoticbehavior}
		Finally, we discuss the asymptotic behavior of weak solutions $(U,\nu)$ of STIMCF as constructed in Section \ref{sec_mainresults} via the projection $(u,\nu_M)$ on $M$ as defined in Section \ref{sec_variationalformulation}. We use a blowdown argument similar to \cite[Section 7]{huiskenilmanen} to see that the level-sets of $u$ become asymptotically round. Moreover, we can use the notion of unit normal $\nu_M$ to show that the surfaces become in fact uniformly starshaped outside of a compact set if that set contains all jump regions.
		
		To begin, we establish the connectedness Lemma \cite[Lemma 4.2]{huiskenilmanen} in the case of STIMCF.
		\begin{lem}\label{lem_connectedness}\,
			\begin{itemize}
				\item[(i)] If $(u,\nu_M)$ is a weak solution of \eqref{eq_weaksolutionsJ}, then $u$ has no strict local maxima or minima.
				\item[(ii)] Suppose $M$ is connected and simply connected\,\footnote{Note that we may relax this assumption and allow curves that loop around $E_0$.} with a single, asymptotically flat end, and $(E_t)_{t>0}$ is a solution of \eqref{comprintset} for all $t>0$ with initial condition $E_0$. If $\partial E_0$ is connected, $\Sigma_t=\partial E_t$ is connected, as long as it remains compact.
			\end{itemize}
		\end{lem}
		\begin{proof}
			We obtain (i) be arguing in analogue to \cite[Lemma 19]{moore}. Once (i) is established, (ii) follows exactly as in the proof of \cite[Lemma 4.2]{huiskenilmanen}.
		\end{proof}
		Let $(U,\nu)$ be a weak solution of STIMCF on $(M,g,K)$ with initial condition $E_0\subseteq M$ and assume that $M$ has exactly one asymptotically flat end. We recall the set $\mathcal{O}_{R_0}\subseteq M$ as defined in Lemma \ref{lem_subsolution} and may assume without loss of generality that $E_0\subseteq M\setminus \mathcal{O}_{R_0}$. Thus we can regard $(U,\nu)$ (up to identification via the asymptotic chart $\Phi$) as a weak solution to STIMCF on $(\Omega,g,K)$ where $\Omega=\Phi(\mathcal{O}_{R_0})$ is an open subset of $\R^{n+1}$. For $\lambda>0$ we define the blowdown objects
		\begin{align*}
			\Omega^\lambda\definedas\lambda\cdot\Omega,\text{ } g^\lambda(y)\definedas \gspann{\frac{y}{\lambda}},\text{ }K^{\lambda}(y)\definedas\frac{1}{\lambda}K\left(\frac{y}{\lambda}\right),\text{ }u^\lambda(y)\definedas u\left(\frac{y}{\lambda}\right),\text{ }\nu_{\Omega}^\lambda(y)\definedas\nu_{\Omega}\left(\frac{y}{\lambda}\right), 
		\end{align*}
		where $\lambda\cdot A\definedas\{x\colon\frac{x}{\lambda}\in A\}$, and $E_t^\lambda\definedas\lambda \cdot E_t=\{u^\lambda<t\}$. Then $(u^\lambda,\nu_\Omega^\lambda)$ is a weak solution of \eqref{eq_weaksolutionsJ} on $(\Omega^\lambda, g^\lambda,K^\lambda)$.
		\begin{prop}\label{prop_blowdown}
			Let $(M,g,K)$ be a triple, such that $(M,g)$ is a Riemannian manifold, $K$ a symmetric $(0,2)$ tensor on $M$ with $\tr_MK=0$, s.t. $M$ has exactly one asymptotically flat end with
			\begin{align}\label{asympdecay2}
				\btr{g-\delta}=o(1),\text{ }\btr{\nabla g}=o\left(\frac{1}{\btr{y}}\right),\text{ }K=o\left(\frac{1}{\btr{y}}\right), \text{ }\btr{\nabla K}=o\left(\frac{1}{\btr{y}^2}\right).
			\end{align}
			Let $E_0$ be a compact $C^2$ domain in $M$ and $(U,\nu)$ a weak solution of STIMCF with initial condition $E_0$ constructed as in Theorem \ref{thm_mainresult}. Then there exist constants $c_\lambda\to\infty$ such that
			\[
				u^\lambda-c_\lambda\to v\definedas n\ln(\btr{y})
			\]
			locally uniformly on $R^{n+1}\setminus\{0\}$ as $\lambda\to 0$, where $v$ is the standard expanding sphere solution on of inverse mean curvature flow on $(R^{n+1}\setminus\{0\},\delta,0)$.
		\end{prop}
		\begin{proof}
			Let everything be as above. The asymptotic conditions \eqref{asympdecay2} and Remark \ref{bem_interiorgradientestimate} imply that there is an $R_1\ge R_0$, such that
			\begin{align*}
				r\ge c\btr{y},\text{ }dist(y,\partial E_0)\ge c\btr{y}
			\end{align*}
			for all $y\in \R^{n+1}\setminus B_{R_1}(0)$, where $r$ satisfies the assumptions of Theorem \ref{thm_interiorgradientestimate} for $y$. Arguing as in \cite[Lemma 7.1]{huiskenilmanen}, we can use Theorem \ref{thm_interiorgradientestimate} and the fact that there is a suitable subsolution $\left(B_{\exp(\alpha t)}(0)\right)_{t_1\le t<\infty}$ in the asymptotic region even under weaker decay assumptions, cf. Remark \ref{bem_subsolution}, to establish a gradient and eccentricity estimates that are preserved under the blowdown by scaling invariance. Defining the constants $c_\lambda\definedas\max\limits_{\mathbb{S}_1(0)}\btr{u^\lambda}$ for $\lambda$ sufficiently small we can use Arzel\`a--Ascoli to conclude that there exists as subsequence $(\lambda_{k_l})$, again denoted by $(\lambda_k)$, and local Lipschitz function $v\in C^{0,1}_{loc}(\R^{n+1}\setminus\{0\})$, such that
			\[
				\widetilde{u}^{\lambda_k}\definedas u^{\lambda_k}-c_{\lambda_k}\to v\text{ locally uniformly in }R^{n+1}\setminus\{0\},
			\]
			and there exists $y_0\in\mathbb{S}_1(0)$ with $v(y_0)=0$. Moreover the eccentricity estimates imply that the level-sets of $v$ are non-empty and compact for all $t\in\R$, cf. \cite[Lemma 7.1]{huiskenilmanen}. We are now left to show that $v$ is a weak solution of \eqref{eq_weaksolutionsJ} in $(\R^{n+1}\setminus\{0\},\delta,0)$, i.e a solution for the corresponding comparison principle for inverse mean curvature flow in $(R^{n+1}\setminus\{0\},\delta)$. Then by \cite[Proposition 7.2]{huiskenilmanen}, $v(y)=n\ln(\btr{y})$ is the standard expanding sphere solution. Since this is then in particular true for  a subsequence of any subsequence, it follows that the full sequence converges, proving the claim.
			
			However to conclude this, we need to confirm the stronger assumptions for the Compactness Theorem \ref{thm_compactness2}, in particular the local $\mathcal{L}^1$ convergence of the gradients. To this end, we consider the solutions of the ellptic regularisation $u_{\varepsilon_n}$ on $F_{L_n}\setminus \overline{E_0}$, where $L_n\to\infty$ s.t. $u_{\varepsilon_n}\to u$ locally uniformly on $\Omega_0$. Now choose a subsequence $(L_{n_k})$, s.t. $L_{n_k}\ge \alpha\ln(\lambda^{-2})-\alpha\ln(R_0)$ and a (possibly different) subsequence $(\varepsilon_{n_k})$, s.t. $\varepsilon_{n_k}<\varepsilon_0(L_{n_k})$ and
			\[
				\norm{u_{\varepsilon_{n_k}}-u}_{C^0(\overline{F_{L_{n_k}}\cap \mathcal{O}_{R_3}})}\le \lambda_k.
			\]
			Now we define the subsets $\Omega_k\definedas\lambda_k\cdot\left( F_{L_{n_k}}\cap \mathcal{O}_{R_3}\right)=B_{e^{\alpha}\lambda^{-1}}(0)\setminus\overline{B}_{R_3\lambda}(0)\to\R^{n+1}\setminus\{0\}$ and consider the functions $\widetilde{u}_k(y)\definedas u_{\varepsilon_{n_k}}\left(\frac{y}{\lambda_k}\right)-c_{\lambda_k}$. Then $\widetilde{u}_k\to v$ locally uniformly and $\widetilde{u}_k$ solves \eqref{ellreg} in $(\Omega_k,g^{\lambda_k},K^{\lambda_k})$. In particular $\widetilde{U}_k(y,z)\definedas\widetilde{u}_k(y)-\varepsilon z\to V(y,z)\definedas v(y)$ locally uniformly in $(\R^{n+1}\setminus\{0\})\times\R$ and $\widetilde{U}_k$ is a smooth solution to STIMCF on $(\widetilde{M}_k,g^{\lambda_k}+\d z^2,\widetilde{K}^k)$, where $\widetilde{M}_k\definedas\Omega_k\times\R$ and $\widetilde{K}^k_{ij}=K^{\lambda_k}_{ij}$, $\widetilde{K}^k_{iz}=\widetilde{K}^k_{zz}=0$. Since the local gradient estimate is also satisfied for $\widetilde{U}_k$ on $\widetilde{M}_k$, we are in the same situation as in Corollary \ref{lem_sublimit}. Then arguing as in Section \ref{sec_limitingbehavior} and Section \ref{sec_mainresults}, and applying a modified Compactness Theorem as in Remark \ref{bem_compactness2}, we conclude that there exists an a.e. locally uniformly continuous unit vector field $\eta$, such that $(V,\eta)$ is a weak solution of \eqref{eq_weaksolutionsJ} in $((\R^{n+1}\setminus\{0\})\times\R, \delta,0)$. By Lemma \ref{lem_equivalence}, $v$ solves \eqref{eq_weaksolutionsJ} on $(R^{n+1},\delta,0)$, concluding the proof. In particular, the convergence holds for any choice $(\lambda_k,L_{n_k},\varepsilon_{n_k})$ as above.
		\end{proof}
		\begin{kor}
			Away from jump regions, the level-sets $\Sigma_t=\{u=t\}$ become uniformly starshaped as $t\to\infty$. More precisely, if there exists $R_{reg}>0$ such that $u$ has no jumps on $\{\btr{x}\ge R_{reg}\}$ in the asymptotic chart, then for any $\delta<1$, there exists $R(\delta)\ge R_{reg}$, such that
			\begin{align}
				\spann{v(y),y}\ge (1-\delta)\btr{y} \text{ for all }\btr{y}\ge R(\delta),
			\end{align}
			in the asymptotic chart.
		\end{kor}
		\begin{proof}
			Let $A\definedas\{1\le \btr{x}\le 3\}$ and choose $\lambda_k=2^{-k}$. We choose $L_{n_k}\ge \alpha\ln(\lambda^{-2})-\alpha\ln(R_0)$, such that $A\subset \lambda_k\Omega_{k}$, where $\Omega_{k}$ is defined as in the proof of Proposition \ref{prop_blowdown}, but we replace the constant $R_3$ by $R_4\definedas\max\{R_3,R_{reg}\}$. Since $g$ satisfies the decay assumptions \eqref{asympdecay2}, $\delta$ and $g$ are equivalent on $\Omega_k$ with constants independent of $k$ and the locally uniform convergence of $\nu_\varepsilon$ implies, that we can choose a subsequence $(\varepsilon_{n_k})$ such that
			\[
			\norm{\nu_{\varepsilon_{n_k}}-\nu}_{(\Omega_{k},\delta)}\le \frac{\delta}{2}
			\]
			for all $k\ge k_1$, where we now have uniform convergence on all of $\Omega_k$, as it contains no jump regions. Further, Proposition \ref{prop_blowdown} implies that there exists a $k_2\in\N$ such that 
			\[
				\norm{\nu_{\varepsilon_{n_k}}^{\lambda_k}-\frac{x}{\btr{x}}}_A\le \frac{\delta}{2}
			\]
			for all $k\ge k_2$. Let $k_0\definedas\max(k_1,k_2)$, then for all $x\in A$, for all $k\ge k_0$, we have
			\begin{align*}
				\spann{\nu^{\lambda_k},\frac{x}{\btr{x}}}
				&=\spann{\nu^{\lambda_k}-\nu^{\lambda_k}_{\varepsilon_{n_k}}+\nu^{\lambda_k}_{\varepsilon_{n_k}}-\frac{x}{\btr{x}}+\frac{x}{\btr{x}},\frac{x}{\btr{x}}}\\
				&\ge 1-\btr{\nu^{\lambda_k}-\nu^{\lambda_k}_{\varepsilon_{n_k}}}-\btr{\nu^{\lambda_k}_{\varepsilon_{n_k}}-\frac{x}{\btr{x}}}\\
				&\ge 1-\norm{\nu_{\varepsilon_{n_k}-\nu}}_{(\lambda_k\Omega_{L_{n_k}},\delta)}-\norm{\nu_{\varepsilon_{n_k}}^{\lambda_k}-\frac{x}{\btr{x}}}_A\\
				&\ge 1-\frac{\delta}{2}-\frac{\delta}{2}=1-\delta.
			\end{align*}
			Hence 
			\[
				\spann{\nu^{\lambda_k},x}\ge (1-\delta)\btr{x}\text{ for all }x\in A,\text{ and all } k\ge k_0.
			\]
			Since $\nu^{\lambda_k}\left(x\right)=\nu\left(\frac{x}{\lambda_k}\right)$, a rescaling by $\lambda_k$ gives
			\[
				\spann{\nu(x),x}\ge (1-\delta)\btr{x}\text{ for }\lambda_k^{-1}\le \btr{x}\le 3\lambda_k^{-1},
			\]
			where $k\ge k_0$. Note that $\lambda_{k+1}^{-1}<3\lambda_k^{-1}$, so we conclude that
			\[
			\spann{\nu(x),x}\ge (1-\delta)\btr{x}\text{ for }\btr{x}\ge R(\delta),
			\]
			with $R(\delta)\definedas \lambda_{k_0}^{-1}$.
		\end{proof}
	\begin{bem}\,
		\begin{enumerate}
			\item[(i)] Similar as in \cite{huiskenilmanen}, one can see that in the case of $n=2$ $\newbtr{\newbtr{\accentset{\circ}{A}}}^2_{\mathcal{L}^2}\to 0$ as $t\to\infty$, so the level-sets approach coordinate spheres in $W^{2,2}$, cf. \cite{delellismueller}.
			\item[(ii)] We expect the solutions to be smooth outside some compact set based on arguments similar as in \cite{huiskenilmanen2}. In particular, all jump regions are contained in a compact set such that the level-sets are uniformly starshaped outside of that set.
			\item[(iii)] In fact, we expect the solution to be well adapted to a new concept of center-of-mass proposed by Cederbaum--Sakovich \cite{cederbaumsakovich} as $t\to\infty$, as we expect the level-sets of our solutions to asymptotically approach the foliation of STCMC surfaces constructed by them in the asymptotic region. The center-of-mass proposed by Cederbaum--Sakovich remedies some of the deficiencies of the center-of-mass formulation via surfaces of constant mean curvature first proposed by Huisken--Yau \cite{huiskenyau}. In particular, we expect our flow to exhibit better asymptotic behavior in non time-symmetric initial data sets than inverse mean curvature flow.
		\end{enumerate}
	\end{bem}
	\nocite{*}
\nopagebreak

\end{document}